\documentclass[10.5pt,reqno]{amsart}
\usepackage{longtable}
\usepackage{hyperref}
\usepackage[T1]{fontenc}
\usepackage[utf8]{inputenc}
\usepackage[english]{babel}
\usepackage{textcomp}
\usepackage{dsfont}
\usepackage{latexsym}
\usepackage{amssymb}
\usepackage{amsthm}
\usepackage{amsmath}
\DeclareMathAlphabet{\mathpzc}{OT1}{pzc}{m}{en}
\usepackage{yfonts}
\usepackage{xfrac}
\usepackage{newlfont}
\usepackage{graphicx}
\usepackage{mathtools}
\usepackage{comment}
\usepackage{indentfirst}
\usepackage{braket}
\usepackage{mathrsfs}
\usepackage{xcolor}

\usepackage{etoolbox}
\apptocmd{\lim}{\limits}{}{}
\apptocmd{\sup}{\limits}{}{}
\apptocmd{\inf}{\limits}{}{}
\apptocmd{\liminf}{\limits}{}{}
\apptocmd{\limsup}{\limits}{}{}

\usepackage{scalerel}[2014/03/10]
\usepackage[usestackEOL]{stackengine}
\newcommand{\dashint}{\,\ThisStyle{\ensurestackMath{%
			\stackinset{c}{.2\LMpt}{c}{.5\LMpt}{\SavedStyle-}{\SavedStyle\phantom{\int}}}%
		\setbox0=\hbox{$\SavedStyle\int\,$}\kern-\wd0}\int}

\DeclareMathOperator{\pr}{pr}
\DeclareMathOperator{\rk}{rk}
\DeclareMathOperator{\supp}{Supp}

\DeclareMathOperator{\Tr}{Tr}

\DeclareMathOperator{\Hol}{Hol}
\DeclareMathOperator{\Cl}{Cl}

\renewcommand{\Re}{\mathrm{Re}\,}
\renewcommand{\Im}{\mathrm{Im}\,}
\newcommand{\HB}{\mathrm{HB}}
\newcommand{\Supp}[1]{\supp\left( #1\right) }

\newcommand{\ee}{\mathrm{e}}

\newcommand{\Poisson}{\mathrm{P}}
\newcommand{\Aut}{\mathrm{Aut}}

\newcommand{\bDs}{{\mathrm{b}\Ds}}
\newcommand{\bD}{{\mathrm{b}D}}
\newcommand{\Ds}{\mathscr{D}}
\newcommand{\Ps}{\mathscr{P}}
\newcommand{\Cs}{\mathscr{C}}
\newcommand{\Ss}{\mathscr{S}}
\newcommand{\Ns}{\mathscr{N}}

\newcommand{\loc}{\mathrm{loc}}
\newcommand{\vect}[1]{\mathbf{{#1}}}
\newcommand{\dd}{\mathrm{d}}

\DeclarePairedDelimiter{\abs}{\lvert}{\rvert}

\DeclarePairedDelimiter{\norm}{\lVert}{\rVert}

\let\originalleft\left
\let\originalright\right
\renewcommand{\left}{\mathopen{}\mathclose\bgroup\originalleft}
\renewcommand{\right}{\aftergroup\egroup\originalright}

\newcommand{\N}{\mathds{N}}

\newcommand{\Db}{\mathds{D}}
\newcommand{\C}{\mathds{C}}

\newcommand{\R}{\mathds{R}}

\newcommand{\T}{\mathds{T}}

\newcommand{\Ff}{\mathfrak{F}}

\newcommand{\Uf}{\mathfrak{U}}

\newcommand{\Cc}{\mathcal{C}}

\newcommand{\Fc}{\mathcal{F}}

\newcommand{\Hc}{\mathcal{H}}
\newcommand{\Ic}{\mathcal{I}}

\newcommand{\Lc}{\mathcal{L}}
\renewcommand{\Mc}{\mathcal{M}}

\newcommand{\Pc}{\mathcal{P}}

\newcommand{\Sc}{\mathcal{S}}

\newcommand{\meg}{\leqslant}
\newcommand{\Meg}{\geqslant}
\newcommand{\eps}{\varepsilon}
\renewcommand{\phi}{\varphi}
\newcommand{\mi}{\mu}

\title{Positive Pluriharmonic Functions on Symmetric Siegel Domains}

\date{}

\begin{document}

\theoremstyle{definition}
\newtheorem{deff}{Definition}[section]

\newtheorem{oss}[deff]{Remark}

\newtheorem{ass}[deff]{Assumptions}

\newtheorem{nott}[deff]{Notation}

\newtheorem{ex}[deff]{Example}

\theoremstyle{plain}
\newtheorem{teo}[deff]{Theorem}

\newtheorem{lem}[deff]{Lemma}

\newtheorem{prop}[deff]{Proposition}

\newtheorem{cor}[deff]{Corollary}

\author[M. Calzi]{Mattia Calzi}

\address{Dipartimento di Matematica, Universit\`a degli Studi di
	Milano, Via C. Saldini 50, 20133 Milano, Italy}
\email{{\tt mattia.calzi@unimi.it}} 

\keywords{Symmetric Siegel domains, plurihamonic functions, Nevanlinna theorem, Clark measures.}
\thanks{{\em Math Subject Classification 2020}: 32M15, 31C10  
}
\thanks{The author is a member of the 	Gruppo Nazionale per l'Analisi
	Matematica, la Probabilit\`a e le	loro Applicazioni (GNAMPA) of
	the Istituto Nazionale di Alta Matematica (INdAM). The author  was partially funded by the INdAM-GNAMPA Project CUP\_E53C22001930001.
} 

\begin{abstract}
	Given a symmetric Siegel domain $\Ds$ and a positive plurihamonic function $f$ on $\Ds$, we study the largest positive Radon measure $\mi$ on the \v Silov boundary $\bDs$ of $\Ds$ whose Poisson integral $\Ps \mi$ is $\meg f$. If $\Ds$ has no \emph{tubular} irreducible factors of rank $\Meg 2$, we show that  $\Ps \mi$ is plurihamonic, and that $f-\Ps \mi$ is linear.
	As an application, we describe a possible analogue of the family of Clark measures associated with a holomorphic function from $\Ds$ into the unit disc in $\C$.
\end{abstract}
\maketitle

\section{Introduction}

Let $\Db$ be the unit disc in $\C$. The classical Riesz--Herglotz  representation theorem states that, if  $f\colon \Db\to \C$ is a holomorphic function with positive real part, then there is a unique positive Radon measure $\mi$ on the torus $\T=\partial \Db$ such that
\[
f(z)= i \Im f(0)+ \int_{\T} \frac{\alpha+z}{\alpha-z}\,\dd \mi(\alpha),
\]
for every $z\in \Db$. Equivalently, 
\[
\Re f(z)= \int_\T \frac{1-\abs{z}^2}{\abs{\alpha-z}^2}\,\dd \mi(\alpha)
\]
for every $z\in \Db$.  This result may then be transferred to the upper half plane $\C_+=\R+i \R_+^*$ by means of the Cayley transform, so that one may get the classical Nevanlinna's representation theorem: if $f\colon \C_+\to \C$ is a holomorphic function with positive \emph{imaginary} part, then there are a unique positive `Poisson summable' Radon measure $\mi$ on $\R$ and a unique $\lambda\Meg 0$ such that
\[
f(z)=\Re f(i) + \lambda z + \frac{1}{\pi}\int_\R \left( \frac{1}{x-z}-\frac{x}{1+x^2}\right) \,\dd \mi(x)
\]
for every $z\in \C_+$. Equivalently, 
\[
\Im f(z)=  \lambda \Im z + \frac{1}{\pi}\int_\R\frac{\Im z}{\abs{z-x}^2} \,\dd \mi(x).
\]
While it is relatively simple to extend to former result to more general bounded domains (e.g., bounded symmetric domains, as essentially noted in~\cite{KoranyiPukanszki}), the latter is more delicate and requires considerably more care. As a matter of fact, to the best of our knowledge, it has been so far extended only to products of upper half-planes, and with considerable effort (cf.~\cite{LugerNedic1,LugerNedic2}; see, also,~\cite{AMY, ATDY} for a different perspective). 
This asymmetry may at first appear peculiar, expecially since the proof of Nevanlinna's representation  theorem essentially amounts to carefully transfer the Riesz--Herglotz representation theorem from the disc to the upper half-plane by means of the Cayley transform. We shall return to this point later.

Concerning the extensions of the previous results, observe that $\Db$ may be considered as the simplest example of a bounded symmetric domain. Then, let $D$ be a circular (hence convex) bounded symmetric domain in $\C^n$, and let $f\colon D\to \R$ be a \emph{positive plurihamonic} function on $D$ (so that $f$ is the real part of a holomorphic function on $D$). Let $\bD$ be the \v Silov boundary of $D$, and denote by $\Pc\colon D\times \bD\to \R$ the Poisson kernel on $D$. Then, there is a unique positive  measure $\mi$ on $D$ such that
\[
f(z)= (\Pc \mi)(z)\coloneqq \int_{\bD} \Pc(z,\zeta)\,\dd \mi(\zeta)
\]
for every $z\in D$. Since the Poisson integral $\Pc \mi$ of $\mi$ is a pluriharmonic function, we shall also say that $\mi$ is a pluriharmonic measure. 
As noted, e.g., in~\cite{Aleksandrov2,AleksandrovDoubtsov,Calzi}, pluriharmonic measures on $\bD$ enjoy rather peculiar properties; for example, every plurihamonic measure on $\bD$ is necessarily absolutely continuous with respect to the Hausdorff measure $\Hc^{m-1}$, where $m$ denotes the dimension of (the real algebraic set) $\bD$ (the Hausdorff measure being relative to the Euclidean distance on $\C^n$). 

Now,   $D$ is biholomorphic to a symmetric Siegel domain $\Ds$, and this correspondence generalizes the one between $\Db$ and $\C_+$; it is also possible to describe a birational biholomorphism $\gamma\colon D\to \Ds$ (the generalized Cayley transform) in a reasonable way. Notice, on the one hand, that $\gamma^{-1}$ canonically extends to a (rational, everywhere defined) mapping $\bDs\to \bD$. On the other hand, $N\coloneqq \bD\setminus \gamma^{-1}(b  \Ds)$ is \emph{never} empty. 
If $D$ has no irreducible factors of tube type, then one may prove that $N$ is a (real) algebraic set of dimension $\meg m-2$, so that no pluriharmonic measure can have mass on $N$. 
Consequently, \emph{in this case}, one may transfer without issues the generalized Riesz--Herglotz  representation theorem and show that for every positive pluriharmonic function $f\colon \Ds\to \R$ there is a unique positive `Poisson-summable' Radon measure $\mi$ on $\bDs$ such that
\begin{equation}\label{eq:4}
f(z)= (\Ps \mi)(z)= \int_{\bDs} \Ps(z,\zeta)\,\dd \mi(\zeta),
\end{equation}
for every $z\in \Ds$,
where $\Ps\colon \Ds\times \bDs\to \R$ denotes the Poisson kernel on $\Ds$. If, otherwise, $D$ has some irreducible factors of tube type, then (in general) one has to add a (positive) remainder $g$ to~\eqref{eq:4}. Unfortunately, it may happen that neither $\Ps \mi$ nor  $g$ is pluriharmonic, so that this representation appears to be quite unsatisfactory. In fact, it may happen that $\Ps \mi$ is \emph{not} pluriharmonic even when $\mi$ is induced by the restriction of a function which is pluriharmonic on a neighbourhood  of the closure of $\Ds$, as counterintuitive as it may seem (cf.~Example~\ref{ex:1bis}).
On the positive side, $\mi$ may still be characterized as the largest positive `Poisson-summable' Radon measure on $\bDs$ such that $\Ps \mi \meg f$, and still enjoys a few of the properties of plurihamonic measures. For example, it admits some natural disintegrations (cf.~Proposition~\ref{prop:25}), and is absolutely continuous with respect to $\Hc^{m-1}$, with the above notation. 

As we noted above, it was   proved in~\cite{LugerNedic1,LugerNedic2}  that, if $\Ds=\C_+^k$ for some $k\in\N$, then $\Ps\mi$ is pluriharmonic, and there are $\lambda_1,\dots, \lambda_k\Meg 0$ such that  $g(z_1,\dots, z_k)=\sum_{j=1}^k \lambda_j \Im z_j$ for every $(z_1,\dots, z_k)\in \C_+^k$. With a similar (but somewhat simplified) proof, one may actually show that the same happens if $\Ds=\Ds_0\times \C_+^k$, where $\Ds_0$ is a product of irreducible symmetric Siegel domains which are \emph{not} tubular (in this case, the remainder $g$ only depends on the last $k$ variables, cf.~Theorem~\ref{teo:2}). Since  examples where $\Ps \mi$ is \emph{not} plurihamonic arise already in every irreducible symmetric tube domain of rank $2$ (cf.~Example~\ref{ex:1}), it seems   that the previous characterization is the best one may hope to prove. On the positive side, if $\Ps \mi$ is pluriharmonic, then the remainder term $g$ is necessarily linear (and vanishes on $\Set{0}\times i \Phi(E)$) as in the case in which $\Ds$ has no irreducible tubular factor of rank $\Meg 2$ (cf.~Corollary~\ref{cor:10}). In particular, $\mi=0$ if and only if $f$ is linear and vanishes on $\Set{0}\times i \Phi(E)$ (cf.~Corollary~\ref{cor:11}).

As an application, we propose a possible notion of Clark measures associated with a holomorphic function $f\colon \Ds\to \Db$. As a matter of fact, the preceding analysis suggests that one should define the Clark measure $\mi_\alpha[\phi]$ as the largest positive Radon measure on $\bDs$ such that $\Ps \mi_\alpha[\phi]\meg \frac{1-\abs{\phi}^2}{\abs{\alpha-\phi}^2}$, for every $\alpha\in \T$. Then, the basic results one may prove for Clark measures on $D$ transfer without particular difficulty, and one may even show that $\Ps\mi_\alpha[\phi]$ is actually pluriharmonic (and, furthermore, equal to $\frac{1-\abs{\phi}^2}{\abs{\alpha-\phi}^2}$) for almost every $\alpha\in \T$.

\smallskip

Here is an outline of the paper. In Section~\ref{sec:2}, we recall some basic facts on pluriharmonic measures on the \v Silov boundary $\bD$ of the \emph{bounded} symmetric domain $D$.
In Section~\ref{sec:3}, we recall the formalism of Jordan triple systems, since it will be necessary to perform the computations of Sections~\ref{sec:5} and~\ref{sec:6}. 
In Section~\ref{sec:4}, we recall a notion of restricted limits introduced in~\cite{MackeyMellon}; even though we shall not make use of the results of~\cite{MackeyMellon} on (restricted) angular limits and derivatives, we shall need some technical estimates which follow from the construction of a family of canonical `projective' (or `projection') devices, which may be of independented interest.
Sections~\ref{sec:5} and~\ref{sec:6} constitute the core of the paper. Here we discuss the notion of pluriharmonic measures on $\bDs$ and prove the representations theorem for positive pluriharmonic functions announced above.
In Section~\ref{sec:7}, we discuss briefly the notion of Clark measures on $\bDs$ introduced earlier, and show how some basic results may be transferred from the corresponding results for Clark measures on $\bD$.
Finally, in the Appendix~\ref{sec:app} we recall some basic facts on the disintegration of Radon measures, for the reader's convenience.

\section{The Bounded Case}\label{sec:2}

In this section we recall some basic results on positive pluriharmonic measures and functions on  a bounded symmtric domain and its \v Silov boundary, respectively. Cf.~\cite{Calzi} for the proofs.

Throughout the section,  $D$ will denote a bounded, circular (hence convex) symmetric domain in a finite-dimensional complex vector space $Z$;\footnote{In other words, for every $z\in D$ there is an involutive biholomorphism of $D$ with $z$ as an isolated fixed point.} we shall denote by   $\bD$ the \v Silov boundary of $D$, that is, the smallest closed subset of $\partial D$ such that $\sup_{z\in D} \abs{f(z)}\meg \sup_{\zeta\in \bD} \abs{f(\zeta)}$ for every (bounded) holomorphic function on $D$ which extends by continuity to the closure $\Cl(D)$ of $D$. We denote by $K$ the group of linear automorphisms of $D$, so that $K$ acts transitively on $\bD$. 
We denote by $\beta_{\bD}$ the unique $K$-invariant probability measure  on $\bD$.

We denote by $\Db$ the unit disc in $\C$, and by $\T$ its boundary, that is, the torus. We denote by $\beta_\T$ the normalized Haar measure on $\T$.
	
In addition, we denote by $\widehat{\bD}$ the quotient of $\bD$ by the canonical action of the torus $\T$ (that is, $(\alpha,\zeta)\mapsto \alpha \zeta$),   by $\pi\colon \bD\to \widehat{\bD}$ the canonical projection, and by $\widehat\beta_{\bD}\coloneqq \pi_*(\beta_{\bD})$ the image of $\beta_{\bD}$ under $\pi$.

	We denote by $\Pc_\R$ the space of polynomial mappings from the real space underlying $Z$  into $\C$. In addition, we denote by $\Pc_\C$ and $\overline{\Pc_\C}$ the spaces of holomorphic and anti-holomorphic polynomials on $Z$, respectively.

\begin{deff}
	We denote by $\Cc$ the Cauchy--Szeg\H o kernel of $D$, that is, the reproducing kernel of the Hardy space\footnote{In other words, $\Cc(\,\cdot\,,z)\in H^2(D)$  and $f(z)=\langle f\vert\Cc(\,\cdot\,,z)\rangle_{H^2(D)}$ for every $f\in H^2(D)$ and for every $z\in D$.}
	\[
	H^2(D)=\Set{f\in \Hol(D)\colon \sup_{0\meg \rho<1} \int_{\bD} \abs{f(\rho \zeta)}^2\,\dd \beta_{\bD}(\zeta)<\infty },
	\]
	and define
	\[
	\Pc\colon D\times \bD\ni (z,\zeta)\mapsto \frac{\abs{\Cc(z,\zeta)}^2}{\Cc(z,z)}\in \R_+,
	\]
	the Poisson--Szeg\H o kernel. Observe that $\Cc(0,\zeta)=1$ for every $\zeta\in \bD$.
	Define, for every (Radon) measure $\mi$ on $\bD$,
	\[
	\Cc(\mi)\colon D\ni z\mapsto \int_{\bD} \Cc(z,\zeta)\,\dd \mi(\zeta)\in \C,
	\]
	\[
	\Pc(\mi)\colon D\ni z\mapsto \int_{\bD} \Pc(z,\zeta)\,\dd \mi(\zeta)\in \C
	\]
	and
	\[
	\Hc(\mi)\colon D\ni z\mapsto \int_{\bD} (2\Cc(z,\zeta)-1)\,\dd \mi(\zeta)\in \C.
	\]
	Finally, for every function $f\colon D\to \C$ and for every $\rho\in [0,1)$, we set $f_\rho\colon \bD \ni \zeta \mapsto f(\rho\zeta)\in \C$.
\end{deff}

\begin{prop}\label{prop:1}
	The following hold:
	\begin{enumerate}
		\item[\textnormal{(1)}]  $\Cc$ extends to a sesqui-holomorphic function on a neighbourhood of $D\times \Cl(D)$;  
		
		\item[\textnormal{(2)}]  $\Pc$ is continuous and  nowhere vanishing on $D\times \bD$;
		
		\item[\textnormal{(3)}]  $\Pc(z,\,\cdot\,)\cdot \beta_{\bD}$ is a probability measure on $\bD$ for every $z\in D$;
				
		\item[\textnormal{(4)}]  for every Radon measure $\mi$ on $\bD$,
		\[
		\lim_{\rho\to 1^-} (\Pc \mi)_\rho \cdot \beta_{\bD}=\mi
		\]
		in the vague topology, and $\lim_{\rho \to 1^-} \norm{(\Pc \mi)_\rho}_{L^1(\beta_{\bD})}=\norm{\mi}_{\Mc^1(\bD)}$.
	\end{enumerate}
\end{prop}

Recall that a function $f\colon D\to \R$ is said to be pluriharmonic if its restriction to every complex line is harmonic. Equivalently, $f$ is pluriharmonic if and only if it is the real part of a holomorphic function (cf., e.g.,~\cite[Proposition 2.2.3]{Krantz}).

\begin{prop}\label{prop:4}
	Let $\mi$ be a real measure on $\bD$. Then, the following conditions are equivalent:
	\begin{enumerate}
		\item[\textnormal{(1)}] $\Pc(\mi)$ is pluriharmonic;
		
		\item[\textnormal{(2)}]  $\Pc(\mi)=\Re \Hc(\mi)$;
		
		\item[\textnormal{(3)}]  $\mi$ is in the vague closure of $\Re \Pc_\C \cdot \beta_{\bD}$;
		
		\item[\textnormal{(4)}]  $\int_{ \bD} P\,\dd \mi=0$ for every $P\in \Pc_\R\cap (\Pc_\C\cup\overline{\Pc_\C})^\perp$, that is, for every $P\in \Pc_\R$ such that $\langle P\vert Q \rangle_{L^2(\beta_{\bD})}=0$ for every $Q\in \Pc_\C\cup \overline{\Pc_\C}$.
	\end{enumerate}
\end{prop}

\begin{deff}
	We say that a complex measure $\mi$ on $\bD$ is pluriharmonic  $\Re \mi$ and $\Im \mi$  satisfy the equivalent conditions of Proposition~\ref{prop:4}.
\end{deff}

\begin{prop}\label{prop:5}
	The mapping $\mi \mapsto \Pc(\mi)$  induces a bijection of the space of positive pluriharmonic measures on $\bD$ onto the space of positive pluriharmonic functions on $D$.
	
	Equivalently, the mapping $\mi\mapsto \Hc(\mi)$ induces a bijection of the space of positive pluriharmonic measures on $\bD$ onto the space of holomorphic functions on $D$ with positive real part and which are real at $0$.
\end{prop}

\begin{prop}\label{prop:8}
	Let $\mi$ be a non-zero pluriharmonic measure on $\bD$. Then, $\widehat \beta_{\bD}$ is a pseudo-image measure of $\mi$ under $\pi$ and $\mi$ is absolutely continuous with respect to $\Hc^{\dim \bD-1}$. 
	
	In addition,  the following hold: 
	\begin{itemize}
		\item if  $(\mi_\xi)$ is a disintegration of $\mi$ relative to $\widehat \beta_{\bD}$, then for $\widehat \beta_{\bD}$-almost every $\xi\in \widehat{ \bD}$ and for every $z\in \Db_\xi$,
		\begin{equation}\label{eq:6}
		(\Pc \mi)(z)= \int_{\bD} \frac{1-\abs{z}^2}{\abs{\zeta-z}^2}\,\dd \mi_\xi(\zeta);
		\end{equation}
		
		\item if $\mi$ is positive, then $\mi$ has a unique vaguely continuous disintegration $(\mi_\xi)$ relative to $\widehat \beta_{\bD}$ such that~\eqref{eq:6} holds for \emph{every} $\xi \in \widehat{\bD}$.
	\end{itemize} 
\end{prop}

We conclude this section with a simple remark which will be of use in the proof of Theorem~\ref{teo:2}.

\begin{oss}\label{oss:8}
	Let $\mi$ be a non-zero positive pluriharmonic measure on $\bD$, and fix $\zeta\in \bD$. Then,
	\[
	\lim_{\rho \to 1^-} (1-\rho)(\Hc \mi)(\rho \zeta)=	2\mi_{\pi(\zeta)}(\Set{\zeta}),
	\]
	where $(\mi_\xi)$ denotes a vaguely continuous disintegration of $\mi$ relative to its pseudo-image measure $\widehat \beta_{\bD}$. Furthermore, $(1-\rho)\abs{(\Hc\mi)(\rho \zeta)}\meg 2(\Hc \mi)(0)$ for every $\rho \in [0,1)$.
\end{oss}

\begin{proof}
	By Propositions~\ref{prop:4} and~\ref{prop:8},  there is $a_\zeta\in \R$ such that
	\[
	(\Hc \mi)(\rho \zeta)=\int_{\bD} \frac{1+\rho \langle \zeta\vert\zeta'\rangle  }{1-\rho \langle \zeta\vert\zeta'\rangle} \,\dd \mi_{\pi(\zeta)}(\zeta')+ i a_\zeta
	\]
	for every $\rho \in [0,1)$. Evaluating at $\rho=0$ then shows that $a_\zeta=0$. Therefore, it is sufficient to observe that
	\[
	\lim_{\rho \to 1^-} (1-\rho)\int_{\bD} \frac{1+\rho \langle \zeta\vert\zeta'\rangle  }{1-\rho \langle \zeta\vert\zeta'\rangle} \,\dd \mi_{\pi(\zeta)}(\zeta')=
	2\mi_{\pi(\zeta)}(\Set{\zeta})
	\]
	by dominated convergence,
	since $(1-\rho)\abs*{\frac{1+\rho \langle \zeta\vert\zeta'\rangle  }{1-\rho \langle \zeta\vert\zeta'\rangle}}\meg 2 $ and $(1-\rho)\frac{1+\rho \langle \zeta\vert\zeta'\rangle  }{1-\rho \langle \zeta\vert\zeta'\rangle}\to 2\chi_{\Set{\zeta}}(\zeta') $ for $\mi_{\pi(\zeta)}$-almost every $\zeta'\in \bD$. The last assertion follows from the previous computations, since $\norm{\mi_{\pi(\zeta)}}_{\Mc^1(\bD)}=(\Hc \mi)(0)$.
\end{proof}

\section{Jordan Triple Systems}\label{sec:3}

In this section we collect some background information on (positive Hermitian) Jordan triple systems. We shall mainly follow~\cite{Loos}, even though we shall slightly depart from its notation (cf.~\cite{MackeyMellon}). \emph{We shall consider only finite-dimensional Jordan triple systems and finite-dimensional Jordan algebras.}

\begin{deff}
	A (finite-dimensional) positive Hermitian Jordan triple system is a (finite-dimensional) complex vector space $Z$ endowed with a `triple product' $\{\,\cdot\,,\,\cdot\,,\,\cdot\,\}\colon Z\times Z\times Z\to Z$ such that the following hold:
	\begin{itemize}
		\item $\{x,y,z\}$ is $\C$-bilinear and symmetric in $(x,z)$, and $\C$-antilinear in $y$;
		
		\item $[D(a,b), D(x,y)]=D(D(a,b) x,y)-D(x,D(b,a)y) $ for every $a,b,x,y,z\in Z$, where $D(a,b)=\{a,b,\,\cdot\,\}$ for every $a,b\in Z$;
		
		\item if $\{x,x,x\}=\lambda x$ for some non-zero $x\in Z$ and some $\lambda\in \C$, then $\lambda>0$.
	\end{itemize}
	We also define  $Q(x)\coloneqq \{x,\,\cdot\,,x\}$ for every $x\in Z$.
	
	Denote by $\Aut(Z)$ the group of automorphisms of $Z$, that is, the group of the $\C$-linear automorphisms $u$ of the vector space $Z$ such that $\{u x,u y, uz\}=u\{x,y,z\}$ for every $x,y,z\in Z$, and let $\Aut_0(Z)$ be the component of the identity in $\Aut(Z)$. 
	
	An $x\in Z$ is called a tripotent if $\{x,x,x\}=x$.\footnote{There is no general agreement on the definition of a tripotent; for instance, $x$ is a tripotent according to~\cite{Loos} if and only if $\sqrt 2 x $ is a tripotent according to our definition. In general, there are also other powers of $\sqrt 2$ that may differ from one reference to another.} Two tripotents $x,y\in Z$ are said to be orthogonal if $D(x,y)=0$ (equivalently, if $\{x,x,y\}=0$; both conditions are symmetric in $x,y$, cf.~\cite[Lemma 3.9]{Loos}).
	A non-zero tripotent is primitive (resp.\ maximal) if it cannot be written as the sum of two non-zero orthogonal tripotents (resp.\ if it is not ortoghonal to any non-zero tripotent). The rank $\rk(e)$ of a tripotent $e$ is the maximal length (i.e., the number of terms)  of a sequence of pairwise orthogonal non-zero tripotents $(e_j)$ such that $e=\sum_j e_j$. 
	A frame is a maximal sequence of pairwise orthogonal non-zero tripotents. The length of such sequences is called the rank $r$ of $Z$.
\end{deff}

The vector spaces generated by any two frames are conjugate under $\Aut_0(Z)$, as well as any two maximal tripotents (cf.~\cite[Proposition 5.2 and Theorem 5.3]{Loos}). In particular, the rank is well defined.
If, in addition, $Z$ is simple (i.e., cannot be decomposed as the product of two non-trivial positive Hermitian Jordan triple systems), then any two frames are conjugate under $\Aut_0(V)$  (cf.~\cite[Theorem 5.9]{Loos}). 

\begin{prop}\label{prop:20}
	Let $Z$ be a positive Hermitian Jordan triple system. Then, the following hold:
	\begin{enumerate}
		\item[\textnormal{(1)}] the mapping $(x,y)\mapsto \Tr D(x,y)$ is a (non-degenerate) $\Aut(Z)$-invariant scalar product on $Z$;
		
		\item[\textnormal{(2)}] $D(x,y)^*=D(y,x)$ with respect to the scalar product in~\textnormal{(1)}, that is,
		\[
		\langle \{x,y,z\}\vert u \rangle=\langle x\vert \{y,z,u\}\rangle
		\]
		for every $x,y,z,u\in Z$;
		
		\item[\textnormal{(3)}]  if $x,y$ are two orthogonal tripotents in $Z$, then $D(x,x)$ and $D(y,y)$ commute, and $x+y$ is a tripotent;
		
		\item[\textnormal{(4)}] (`spectral decomposition') for every non-zero $x\in Z$ there are a unique strictly increasing family $(\lambda_j)_{j=1,\dots, k}$ of elements of $(0,+\infty)$, and a unique family $(e_j)_{j=1,\dots,k}$ of pairwise orthogonal non-zero tripotents such that
		\[
		x=\sum_{j=1}^k \lambda_j e_j;
		\]
		
		\item[\textnormal{(5)}] with the notation of~\textnormal{(4)}, $D(x,x)=\sum_{j=1}^k \lambda_j^2 D(e_j,e_j)$ and $\norm{D(x,x)}=\lambda_k^2$ with respect to the scalar product in~\textnormal{(1)};
		
		\item[\textnormal{(6)}] the circular convex open set $\Set{x\in Z\colon \norm{D(x,x)}<1}$ is a bounded symmetric domain with \v Silov boundary equal to the set of maximal tripotents.
	\end{enumerate}
\end{prop}

\begin{proof}
	(1)--(2) This is~\cite[Corollary 3.16]{Loos}.
	
	(3) Cf.~\cite[Lemma 3.9]{Loos}.
	
	(4) Cf.~\cite[Corollary 3.12]{Loos}.
	
	(5) Cf.~\cite[Theorem 3.17 and Theorem 6.5]{Loos}.
\end{proof}

\begin{deff}
	Let $Z$ be a positive Hermitian Jordan triple system. Then, we shall endow $Z$ with the ($\Aut(Z)$-invariant)  scalar product defined in (1) of Proposition~\ref{prop:20} (denoting by $\abs{\,\cdot\,}$ the corresponding norm), and we shall denote by $\norm{\,\cdot\,}$ the ($\Aut(Z)$-invariant) `spectral norm' $x \mapsto \norm{D(x,x)}$. The bounded symmetric domain $\Set{x\in Z\colon \norm{x}<1}$ is said to be associated with $Z$.
\end{deff}

Notice that every circular convex bounded symmetric domain is associated to a unique positive Hermitian Jordan triple system (the triple product may be determined by means of the unweighted Bergman kernel), cf.~\cite[Theorem 4.1]{Loos}. Consequently, every positive Hermitian Jordan triple system is induced by a (finite-dimensional) `$JB^*$-triple'. We may therefore apply the results of, e.g.,~\cite{MackeyMellon}.

\begin{deff}
	A (real or complex) Jordan algebra is a (real or complex) commutative algebra $A$ such that
	\[
	x^2(xy)=x(x^2 y)
	\]
	for every $x,y\in A$.
	If $A$ is real, then it is called formally real if $\Set{(x,y)\in A^2\colon x^2+y^2=0}=\Set{(0,0)}$.
	
	An $x\in A$ is said to be an idempotent if $x^2=x$. Two idempotents $x,y\in A$ are said to be orthogonal if $x y=0$. The rank $\rk(e)$ of an idempotent of $A$ is the maximal length of a sequence $(e_j)$ of pairwise orthogonal non-zero idempotents such that $\sum_j e_j=e$. The rank of $A$ is the maximal rank of its idempotents.
	
	An $x\in A$ is said to be invertible if the endomorphism $P x\colon y \mapsto 2 x(x y)-x^2 y$ of $A$ is invertible, in which case $x^{-1}\coloneqq (Px)^{-1}x$ is the inverse of $x$.
\end{deff}

\begin{oss}
	Let $A$ be a real (resp.\ complex) Jordan algebra with identity $e$. Then, $x\in A$ is invertible if and only if it is invertible in the (associative) algebra $\R[x]$ (resp.\ $\C[x]$) generated by $x$ (and $e$), in which case $x^{-1}$ coincides with the inverse of $x$ in $\R[x]$ (resp.\ $\C[x]$). (Cf.~\cite[Proposition II.3.1]{FarautKoranyi2}.)
\end{oss}

\begin{prop}
	Let $A$ be a formally real Jordan algebra. Then, the complex Jordan algebra $A_\C$ becomes a positive Hermitian Jordan triple system with the triple product
	\[
	\{x,y,z\}=(x \overline y) z+x (\overline y z)-(x z)\overline y.
	\]
	In addition, every idempotent of $A$ is a tripotent in $A_\C$, and two idempotents $x,y$ of $A$ are orthogonal as idempotents (i.e., $xy=0$) if and only if they are orthogonal as tripotents (i.e., $D(x,y)=0$).
	
	In addition, the scalar product of $A_\C$ satisfies
	\[
	\langle x y\vert z \rangle=\langle y\vert \overline x z \rangle
	\]
	for every $x,y,z\in A_\C$.
\end{prop}

The Jordan triple system so defined on $A_\C$ is called the Hermitification of $A$.

\begin{proof}
	The first assertion is contained in~\cite[3.7]{Loos}. The second assertion follows from the fact that $\{x,x,y\}=x^2 y+x (x y)- (x y) x=x y$ if $x, y$ are idempotents in $A$. For what concerns the third assertion, observe first that it is sufficient to prove the stated equality for $x,y,z\in A$. In this case, $\Tr D(x,y)= 2\Tr(L(x y)+L(x)L(y)-L(y)L(x))=2\Tr L(x y)$, denoting by $L(x)$ the endomorphism $ y\mapsto x y$ of $A$ (the $2$ is related to the fact that $D(x,y)$ is the $\C$-linear extension of $L(x y)+L(x)L(y)-L(y)L(x)$ to $A_\C$). The assertion then follows from~\cite[Proposition II.4.3]{FarautKoranyi2}.
\end{proof}

\begin{prop}\label{prop:21}
	(`Peirce decomposition') Let $Z$ be a positive Hermitian Jordan triple system, and let $e$ be a tripotent in $Z$. Denote by $Z_\alpha$ the eigenspace of $D(e,e)$ corresponding to the eigenvalue $\alpha\in \C$. 
	Then, $Z=Z_1(e)\oplus Z_{1/2}(e) \oplus Z_0(e)$ (orthogonal sum) and $\{ Z_\alpha(e), Z_\beta(e), Z_\gamma(e) \}\subseteq Z_{\alpha-\beta+\gamma}(e)$ for every $\alpha,\beta,\gamma\in \Set{0,1/2,1}$. 
	Furthermore, $Z_1(e)$ is a complex Jordan algebra with product $(x,y)\mapsto \{x,e,y\}$ and identity $e$; with respect to the conjugation $x \mapsto x^*\coloneqq  \{e,x,e\}  $, $Z_1(e)$ is the Hermitification of the formally real Jodarn algebra $A(e)\coloneqq \Set{x\in Z_1(e)\colon x=x^* }$.
	Finally, the orthogonal projectors of $Z$ onto $Z_1(e)$ and $Z_{1/2}(e)$ are $Q(e)^2$ and $2(D(e,e)-Q(e)^2)$, respectively.
\end{prop}

Cf.~\cite[Theorem 3.13]{Loos} for a proof of the previous result (cf.~\cite[p.~5]{MackeyMellon} for the last assertion). Notice that, if $e'$ is another tripotent with $Z_1(e)=Z_1(e')$, then an element of $Z_1(e)$ is invertible in $Z_1(e)$ if and only if it is invertible in $Z_1(e')$ (with the same inverse, if $A(e)=A(e')$). This is a consequence of the fact that invertibility only depends on the triple product, since $(P x)y=\{x, \{e,y,e\} ,x\}$ in $ Z_1(e)$ and $(Px)y=\{x, \{e',y,e'\} ,x\}$ in $Z_1(e')$, and the conjugations $y\mapsto \{e,y,e\}$ and $y\mapsto \{e',y,e'\}$ are invertible (and equal if $A(e)=A(e')$).

\begin{oss}\label{oss:12}
	Let $e$ be a tripotent in $Z$. Then, the orthogonal projector $Q(e)^2\colon Z\to Z_1(e)$ has norm $\norm{e}$ \emph{with respect to the spectral norm}.  
\end{oss}

Notice that $\norm{e}=1$ unless $e=0$.

\begin{proof}
	It suffices to use the  inequality $\norm{\{x,y,z\}}\meg \norm{x}\norm{y}\norm{z}$ for every $x,y,z\in Z$ (cf.~\cite[Corollary 3]{FriedmanRusso}). 
\end{proof}

\begin{prop}\label{prop:22}
	Let $Z$ be a positive Hermitian Jordan triple system, and let $e$ be a maximal tripotent of $Z$. Set $E\coloneqq Z_{1/2}(e)$ and $F\coloneqq A(e)$ (cf.~Proposition~\ref{prop:21}), and identify $F_\C$ with $Z_1(e)$ and $E\times F_\C$ with $Z$. 
	Endow $Z$ with the product $(x,y)\mapsto \{x,e,y\}$, and define $\Phi\colon E\times E\ni (\zeta,\zeta')\mapsto 2\{\zeta,\zeta',e\}\in F_\C$ and $\Omega$ as the interior of $ \Set{x^2\colon x\in A(e)}$. Then, $\Omega$ is a symmetric cone,\footnote{In other words, $\Omega$ is an open convex cone which is self-dual (that is, $ \Omega=\Set{x\in A(e)\colon \forall y\in \overline \Omega\setminus \Set{0} \:\: \langle x, y \rangle > 0}$) and such that the group of linear automorphisms of $A(e)$ which preserve $\Omega$ acts  transitively on $\Omega$.} $\Phi$ is Hermitian, non-degenerate, and $\overline{\Omega}$-positive, and
	\[
	\Ds=\Set{(\zeta,z)\colon \Im z-\Phi(\zeta)\in \Omega}
	\]
	is a symmetric Siegel domain of type II. In addition, the mapping
	\[
	\gamma\colon D\ni (\zeta,z) \mapsto (2 (e- z)^{-1} \zeta, i (e-z)^{-1}(e+z))\in \Ds 
	\]
	is a birational biholomorphism (which maps $(0,0)$ to $(0,i  e)$), with inverse
	\[
	\Ds \ni (\zeta,z)\mapsto (4 i (z+i e)^{-1} \zeta,  (z+i e)^{-1}(z-i e) )\in D.
	\]
	
	In addition, the involution of $\Ds$ with fixed point $(0, i e)$ is $\gamma\circ (-\gamma^{-1})$, that is,
	\[
	\Ds \ni (\zeta,z)\mapsto (-2 i z^{-1} \zeta,-z^{-1}).
	\]
\end{prop}

Cf.~\cite[(7) of 10.1, Lemma 10.2, Proposition 10.3, and Corollary 10.9]{Loos} (with several substitutions) for a proof.

\begin{deff}
	The Siegel domain $\Ds$ of Proposition~\ref{prop:22} is said to be associated with the tripotent $e$. We say that $Z$ (or $D$) is of tube type if $E=\Set{0}$ (this does not depend on the choice of $e$).
\end{deff}

\begin{prop}\label{prop:23}
	Let $A$ be a formally real Jordan algebra with identity $e$ and rank $r$. Then, there is a unique polynomial $\det$ on $A$ such that, if $x\in A$ and $x=\sum_j \lambda_j e_j$ is the spectral decomposition of $x$ (with the $e_j$ pairwise orthogonal non-zero idempotents and the $\lambda_j$ distinct real numbers), then $\det(x)= \prod_{j} \lambda_j^{\rk(e_j)}$. 
	In addition, the following hold:
	\begin{enumerate}
		\item[\textnormal{(1)}] $\det$ is irreducible if and only if $A$ is simple;
		
		\item[\textnormal{(2)}] if $A$ is simple and $Z$ is its Hermitification, then the (holomorphic) extension of $\det$ to $Z$ is the unique holomorphic polynomial of degree $r$ such that $\det (e)=1$ and  such that $\det(k z)=\chi(k) \det z$ for every $k\in \Aut_0(Z)$, where $\chi$ is a character of   $\Aut_0(Z)$;
		
		\item[\textnormal{(3)}]  if $(e_1,\dots, e_r)$ is a maximal family of pairwise orthogonal idempotents of $A$, then $\det\big(\sum_j \lambda_j e_j\big)=\prod_j \lambda_j$ for every $(\lambda_j)\in \C^r$;
		
		\item[\textnormal{(4)}]  if $A$ is simple and  $(e_1,\dots, e_r)$ is a frame of tripotents of $Z$, then $\abs*{\det\big(\sum_j \lambda_j e_j\big)}=\prod_j \abs{\lambda_j}$ for every $(\lambda_j)\in \C^r$;
		
		\item[\textnormal{(5)}] if $A$ is simple, then  ${\det}_\C P(x)=(\det x)^{2 n/r}$ for every $x\in Z$, where $P(x)=\{x,\overline{\,\cdot\,},x \}$, and $n$ is the (complex) dimension of $Z$  (note that $2n/r$ is an integer).
	\end{enumerate}
\end{prop}

\begin{proof}
	The first assertion follows from~\cite[Proposition II.2.1 and Theorem III.1.2]{FarautKoranyi2}. (1) then follows from~\cite[Lemma 2.3]{KoranyiPukanszki}, while (2) is a consequence of ~\cite[Theorem 2.1]{FarautKoranyi}. Next, 	(3) follows from~\cite[Theorem III.1.2]{FarautKoranyi2} when $(\lambda_j)\in \R^r$, and by holomorphy in the general case. For (4), Observe that $\abs{\det}$ is necessarily invariant under the compact group $\Aut_0(Z)$ by (2), so that the assertion follows from (3) and the fact that any two frames of tripotents are conjugate under $\Aut_0(Z)$ (cf.~\cite[Theorem 5.9]{Loos}).
	Finally, (5) follows from~\cite[Proposition III.4.2]{FarautKoranyi2} when $x\in A$, since in this case ${\det}_\C P(x)={\det}_\R (P(x)\vert_A)$, and by holomorphy in the general case.
\end{proof}

\begin{prop}\label{prop:38}
	Let $A$ be a formally real Jordan algebra with identity $e$ and rank $r$, and let $(e_1,\dots, e_r)$ be a frame of idempotents of $A$. Let $\Omega$ be the symmetric cone associated with $A$. Then, there is a unique holomorphic family $(\Delta^{\vect s})_{\vect s\in \C^r}$ of holomorphic functions on $\Omega+ i A$ such that
	\[
	\Delta^{\vect s}(x)=\prod_{j=1}^r {\det}_{Z_1(e_1+\cdots+e_j)}(\pr_j(x))^{s_j-s_{j+1}}
	\]
	for every $x\in \Omega+ i A$ and for every $\vect s\in \C^r$, where $Z$ is the Hermitification of $A$, $\pr_j\colon Z\to Z_1(e_1+\cdots+ e_j)$ is the canonical projection, and $s_{r+1}=0$.
\end{prop}

\begin{proof}
	If $x\in \Omega$, then the stated formula makes sense   since ${\det}_{Z_1(e_1+\cdots+e_j)}(\pr_j(x))>0$ for every $j=1,\dots, r$ (cf.~\cite[p.~122]{FarautKoranyi2}). The fact that $\Delta^{\vect s}$ extends by holomorphy to $\Omega + i A$ is then a consequence of the fact that $\Delta^{\vect s}$ (up to an inversion) may be interpreted as the Laplace transform of a suitable `Riesz distribution' supported on $\overline \Omega$, cf.~\cite[Proposition VII.1.2 and Theorem VII.2.6]{FarautKoranyi} (where only the case in which $A$ is simple is considered; the general case may be deduced from this one with some care. Cf.~also~\cite{Ishi} for the general case, treated from a different perspective).
\end{proof}

\begin{ex}\label{ex:2}
	As an example, we shall now describe the Jordan triple system corresponding to the irreducible bounded symmetric domains of type $(I_{p,q})$ ($p\meg q$). Let $Z$ be the space of complex $p\times q$ matrices, endowed with the triple product defined by
	\[
	\{x,y,z\}=\frac{1}{2}(x y^* z+z y^* x)
	\] 
	for every $x,y,z\in Z$, where $y^*$ denotes the conjugate transpose of $y$. In this case (identifying $Z$ with the space of $\C$-linear mappings from $\C^p$ to $\C^q$), a tripotent is simply a  partial isometry, and a maximal tripotent is an isometry. The spectral decomposition of an element of $Z$ is then (essentially) its singular value decomposition, so that the spectral norm is simply the operator norm. In addition, the rank of $Z$ is $p$.
	
	If we let $e$ be the canonical inclusion $(I,0)$ of $\C^p$ in $\C^q$, then $e$ is a maximal tripotent, $Z_1(e)$ is naturally identified with the algebra of endomorphisms of $\C^p$, and $Z_{1/2}(e)$ is naturally identified with the space of $\C$-linear mappings from $\C^p$ into $\C^{q-p}$. The real Jordan algebra $A(e)$ is naturally identified with the space of $p\times p$ Hermitian matrices, and the cone $\Omega$ is naturally identified with the space of positive non-degenerate $p\times p$ Hermitian matrices. Then, the Siegel domain associated with $e$ is\footnote{Mind that $\{\zeta,e,z\}=\frac 1 2 z\zeta $ and that $\{\zeta,\zeta',e\}= \frac 1 2 \zeta\zeta'^*$ for every $\zeta,\zeta'\in Z_{1/2}(e)$ and for every $z\in Z_1(e)$.}
	\[
	\Set{(\zeta,z)\colon \frac{z-z^*}{2 i}- \zeta \zeta^*\in \Omega },
	\]
	and the associated Cayley transform is
	\[
	D\ni (\zeta,z)\mapsto ( (I-z)^{-1}\zeta, i (I-z)^{-1}(I+z) ) \Ds
	\]
	with inverse
	\[
	\Ds\ni (\zeta,z)\mapsto (2 i (z+ i I)^{-1}\zeta, (z+ i I)^{-1}(z- i I))\in D.
	\]
	Notice that an element of $Z_{1}(e)$ is invertible if and only if it is invertible in the usual sense (with the same inverse), since the product of $Z_{1}(e)$ (that is, $(x,y)\mapsto \frac 1 2 (x y+ y x)$) and the usual product coincide on the algebra generated by any element of $Z_{1}(e)$.  
	
	Finally, $\det$ on $Z_1(e)$ is the usual (complex) determinant.
\end{ex}

\section{Restricted Limits}\label{sec:4}

\begin{deff}
	Let $Z$ be a positive Hermitian Jordan triple system. Define   the Bergman operator $B$ by
	\[
	B(x,y)\coloneqq I_Z-2 D(x,y)+Q(x) Q(y)
	\]
	for every $x,y\in Z$. Given a non-zero tripotent $e\in Z$ and $k>0$, define the generalized angular region
	\[
	D_k(e)\coloneqq \Set{x\in Z\colon \norm{x}<1, \norm{B(x,e) Q(e)^2 }'^{1/2}<k (1-\norm{w}^2) },
	\]
	where $\norm{B(x,e)Q(e)^2}'=\max_{\norm{y}\meg 1} \norm{B(x,e) Q(e)^2y}$.
\end{deff}

\begin{lem}\label{lem:8}
	Let $e$ be a non-zero tripotent, and let $\pi$ be the orthogonal projector of $Z$ with range $\C e$. Then, $\pi(D)\subseteq \Db e$ and $\pi(D_k(e))\subseteq D_k(e)$ for every $k>0$.
\end{lem}

Thus, $\pi$ is a `projection device' (according to~\cite{Abate}) or a `projective device' (according to~\cite{MackeyMellon}).
We first need a simple lemma.

\begin{lem}\label{lem:10}
	Take a tripotent $e$ in $Z$. Then, for every $x\in Z$,
	\[
	\abs{\Tr D(x,e)}\meg \abs{e}^2\norm{Q(e)^2 x}. 
	\]
\end{lem}

\begin{proof}
	Observe that $Q(e)^2 x$ is the component of $x$ in $Z_1(e)$, so that $\Tr D(x,e)=\Tr D(Q(e)^2 x,e)$ by Proposition~\ref{prop:21}. Therefore, we may reduce to the case in which $Z=Z_1(e)$.  In addition, we may assume that $e\neq 0$, so that the statement may be rephrased saying that the orthogonal projector $\pi_e$ of $Z$ onto $\C e$, namely $x \mapsto \frac{1}{\abs{e}^2} \langle x \vert e \rangle e= \frac{e}{\abs{e}^2} \Tr D(x,e)$ has norm $\meg 1$ with respect to the spectral norm. Now, let $(e_1,\dots, e_k)$ be a family of pairwise orthogonal \emph{primitive} tripotents such that $e=e_1+\cdots+e_k$, and observe that $\pi'=\pi_{e_1}+\cdots+\pi_{e_k}$ is the orthogonal projector of $Z$ onto $Z'\coloneqq \bigoplus_{j=1}^k \C e_j$ (defining $\pi_{e_j}$ as the orthogonal projector of $Z$ onto $\C e_j$), and that $\pi_e=\pi_e \circ\pi'$. Now, observe that $\pi_{e_j}=Q(e_j)^2$ since $Z_1(e_j)=\C e_j$ (cf., e.g.,~\cite[5.4 and 4.14]{Loos}), for every $j=1,\dots, k$. Therefore, $\norm{\pi'(x)}=\max_{j=1,\dots, k} \norm{\pi_{e_j}(x)}\meg \norm{x}$ for every $x\in Z$, thanks to Remark~\ref{oss:12}. In addition, the mapping $\iota \colon \C^k \ni z \mapsto \sum_j z_j e_j\in Z'$ is an isometry when $\C^k$ is endowed with the $\ell^\infty$ norm and $Z'$ is endowed with the spectral norm, and clearly $(\pi_e\circ \iota)(z)=e \sum_{j=1}^k \frac{\abs{e_j}^2}{\abs{e}^2} z_j$ for every $z\in \C^k$. The assertion follows since $\sum_{j=1}^k\frac{\abs{e_j}^2}{\abs{e}^2}=1$.  
\end{proof}

\begin{proof}[Proof of Lemma~\ref{lem:8}.]
	Take $x\in Z$, and take $x_1\in Z_1(e)$, $x_{1/2}\in Z_{1/2}(e)$, and $x_0\in Z_0(e)$ so that $x=x_1+x_{1/2}+x_0$ (cf.~Proposition~\ref{prop:21}). Observe that
	\[
	\pi(x)= \frac{e}{\abs{e}^2} \Tr D(x,e) = \frac{e}{\abs{e}^2} \Tr D(x_1,e) .
	\]
	Set $\lambda_x\coloneqq  \frac{1}{\abs{e}^2} \Tr D(x_1,e) $ for simplicity. Observe that 
	\[
	\norm{\pi(x)}= \abs{\lambda_x}=\frac{1}{\abs{e}^2} \abs{\Tr D(x,e)}\meg 	\norm{x}
	\]
	thanks to Lemma~\ref{lem:10}. In particular, $\pi(x)\in \Db e$ if $x\in D$.
	
	Now, observe that $Q(e)^2$ is the orthogonal projector of $Z$ onto $Z_1(e)$ by Proposition~\ref{prop:21} and that 
	\[
	B(\pi(x),e)Q(e)^2=(I_Z- 2\lambda_x D(e,e)+\lambda_x^2 Q(e)^2) Q(e)^2=(1-\lambda_x)^2 Q(e)^2
	\]
	since $D(e,e)=Q(e)^2=I_{Z_1(e)}$ on $Z_1(e)$. In addition,
	\[
	Q(e)^2B(x,e)Q(e)^2= Q(e)^2- 2 Q(e)^2 D(x_1,e) Q(e)^2+Q(e)^2Q(x_1)Q(e)^3=B(x_1,e) Q(e)^2,
	\]
	thanks to Proposition~\ref{prop:21}. Since $\norm{Q(e)^2}'=1$ by Remark~\ref{oss:12}, this shows that
	\[
	\norm{B(x,e)Q(e)^2}'\Meg \norm{Q(e)^2B(x,e)Q(e)^2}'= \norm{B(x_1,e) Q(e)^2}'\Meg \norm{e-x_1}^2,
	\]
	where the last inequality follows from~\cite[Lemma 3.10]{MackeyMellon}. 
	Then,
	\[
	\norm{B(\pi(x), e) Q(e)^2}'^{1/2}=\abs*{ 1-\frac{1}{\abs{e}^2} \Tr(D(x_1,e)) }=\frac{1}{\abs{e}^2}\abs{\Tr D(e-x_1,e)}\meg  	\norm{e-x_1},
	\]
	thanks to Lemma~\ref{lem:10} again. Consequently, if $x\in D_k(e)$, then
	\[
	\norm{B(\pi(x),e) Q(e)^2}'^{1/2}\meg \norm{B(x,e) Q(e)^2}'^{1/2}<k(1-\norm{x}^2)\meg k(1-\norm{\pi(x)}^2),
	\]
	that is,	 $\pi(x)\in D_k(e)$. This completes the proof.
\end{proof}

\begin{oss}
	Let $P_1,\dots, P_h$ be the orthogonal projectors onto the simple factors of $Z$. If $e$ is a non-zero tripotent in $Z$ and $\pi'$ is the orthogonal projector of $Z$ onto $\C e$ \emph{with respect to a scalar product of the form $\langle x \vert y \rangle'\coloneqq \sum_{j=1}^h a_j \langle P_j x\vert P_j y \rangle$} for some $a_1,\dots, a_h>0$, then minor modifications in the proof of Lemma~\ref{lem:10} show that $\pi'(D)\subseteq \Db e$  and that $\pi(D_k(e))\subseteq D_k(e)$ for every $k>0$, so that also $\pi'$ is a `projective' (or `projection') device. 
	
	Observe that the scalar products such as $\langle\,\cdot\,\vert \,\cdot\,\rangle'$ may be characterized as the ones which are  $\Aut_0(Z)$-invariant (notice that $\Aut_0(Z)\cong \Aut_0(P_1 Z)\times \cdots \times \Aut_0(P_h Z)$).
\end{oss}

As a consequence of~\cite[Theorems 4.1 (and the Note following it) and 4.2]{MackeyMellon}, we may then state the following result. Even though we shall not use it in the sequel, it is a natural consequence of Lemma~\ref{lem:8} in view of the machinery developed in~\cite{MackeyMellon}. Besides that, since it may be tempting to replace radial limits in $D$ (or vertical limits in $\Ds$) with more general ones, it is important to stress the fact that, in the case of symmetric domains of rank $\Meg 2$, the conditions which allow an `automatic' extension are far more restrictive than in the case of domains of rank $1$, and that one cannot do reasonably any better, at least in tube domains (cf.~Example~\ref{ex:7}).

\begin{prop}\label{prop:28}
	Take a holomorphic function $f\colon D\to \C$ and $e\in \bD$ such that $f$ is bounded on $D_k(e)$ for every $k>0$. Let $\Gamma\colon [0,1)\to D$ be a curve such that the following hold:
	\begin{itemize}
		\item $\lim_{t\to 1^-} \Gamma(t)=e$ (`$\Gamma$ is an $e$-curve');
		
		\item $\lim_{t\to 1^-} \frac{\norm{\Gamma(t)-\gamma(t)}}{1-\norm{\gamma(t)}}=0$, where $\gamma=\pi_e\circ \Gamma$ and $\pi_e$ is the orthogonal projection of $Z$ onto $\C e$ (`$\Gamma$ is special').
	\end{itemize}
	Then, 
	\[
	\lim_{t\to 1^-} [(f\circ \Gamma)(t)-(f\circ \gamma)(t)]=0.
	\]
	In particular, if:
	\begin{itemize}
		\item $\gamma([t_0,1))\subseteq D_k(e)$ for some $t_0\in [0,1)$ and some $k>0$ (`$\Gamma$ is restricted');\footnote{Because of the computations in the proof of Lemma~\ref{lem:8}, this means that $\gamma([t_0,1))\subseteq \Db_k(1) e$, where $\Db_k(1)$ is defined as $D_k(e)$, replacing $D$ with $\Db$ and $e$ with $1$, so that $\Db_k(1)=\Set{w\in \Db\colon  \abs{1-w}<k (1-\abs{w}^2)}$.}
	\end{itemize}
	then the limits of $(f\circ \Gamma)(t)$, $(f\circ \gamma)(t)$, and $f(t e)$, for $t\to 1^-$, exist simultaneously (that is, the existence of any one of them implies the existence of the other two) and are equal.
\end{prop}

\begin{ex}\label{ex:7}
	When $D$ is the \emph{Euclidean} unit ball of $\C^n$, then one may allow the special curves to be the ones such that 
	\[
	\lim_{t\to 1^-} \frac{\norm{\Gamma(t)-\gamma(t)}^2}{1-\norm{\gamma(t)}}=0,
	\]
	with the above notation; Proposition~\ref{prop:28} then holds by \v Cirka's theorem (cf., e.g.,~\cite[Theorem 8.4.4]{Rudin}). 
	Nonetheless, in more general domains the same assertion would be incorrect (as essentially noted in~\cite{MackeyMellon}). 
	For example, assume that $\Ds=\C_+^2$ and $D=\Db^2$, so that we may take, as inverse Cayley transform, $\gamma=\gamma_0\times \gamma_0\colon  \C_+^2\to \Db^2$, where $\gamma_0\colon \C_+ \ni w \mapsto \frac{w-i}{w+i}\in \Db$. Consider
	\[
	f\colon \C_+^2\ni (w_1,w_2)\mapsto \frac{ w_1}{w_1+w_2}\in \C
	\]
	and observe that $f$ is well defined since $\Im (w_1+w_2)>0$ for every $(w_1,w_2)\in \C_+^2$. In addition, observe that $D_k((-1,-1))\subseteq \Db_k(-1)^2$ since $\norm{(w_1,w_2)}=\max(\abs{w_1},\abs{w_2})$ and $B_D((w_1,w_2),(-1,-1))=B_\Db(w_1,-1) \times B_\Db(w_2,-1)=[(1-w_1)^2 I_\Db ]\times [(1-w_2)^2 I_\Db]$ for every $(w_1,w_2)\in \Db^2$. Hence, $\gamma^{-1}(D_k(-1,-1))$ is contained in the product of two angles of the form $\abs{\Re w}\meg C \Im w$ in $\C_+$, so that $f\circ \gamma^{-1}$ is bounded in $D_k(-1,-1)$ for every $k>0$. 
	Now, for every $(a,b)\in \Omega=(\R_+^*)^2$, consider the curve $\Gamma_{(a,b)}\colon [0,1)\ni t \mapsto \gamma((1-t) i (a,b))=\big( \frac{(1-t) a-1}{(1-t) a +1}, \frac{(1-t)b-1}{(1-t) b+1} \big)\in \Db^2$, and observe that $\Gamma_{(a,b)}$ is a $(-1,-1)$-curve, and that
	\[
	\gamma_{(a,b)}(t)=\frac{1}{2}\left(\frac{(1-t) a-1}{(1-t) a +1}+ \frac{(1-t)b-1}{(1-t) b+1} \right) (1,1)
	\]
	for every $t\in [0,1)$. Consequently, 
	\[
	1-\norm{\gamma_{(a,b)}(t)}= \frac{(1-t) a}{(1-t) a +1}+ \frac{(1-t)b}{(1-t) b+1} \sim (1-t) (a+b)
	\]
	and
	\[
	\norm{\Gamma_{(a,b)}(t)-\gamma_{(a,b)}(t)}=(1-t) \frac{\abs{a-b}}{((1-t)a+1)((1-t)b+1)}\sim (1-t) \abs{a-b}
	\]
	for $t\to 1^-$. Therefore, unless $a=b$ (in which case  $\gamma_{(a,b)}=\Gamma_{(a,b)}$),
	\[
	\lim_{t\to 1^-} \frac{	\norm{\Gamma_{(a,b)}(t)-\gamma_{(a,b)}(t)}}{1-\norm{\gamma_{(a,b)}(t)}}>0=\lim_{t\to 1^-} \frac{	\norm{\Gamma_{(a,b)}(t)-\gamma_{(a,b)}(t)}^2}{1-\norm{\gamma_{(a,b)}(t)}}.
	\]
	In addition,
	\[
	\lim_{t \to 1^-} f(\gamma^{-1}(\gamma_{(a,b)}(t)))=\lim_{t \to 0^+} f (i t ,i t)=\frac 1 2
	\]
	since $\gamma^{-1}\circ\gamma_{(a,b)}$ is a $0$-curve with range  in $i \R_+^*(1,1)$, while
	\[
	\lim_{t\to 1^-} f(\gamma^{-1}(\Gamma_{(a,b)}(t)))=\lim_{t\to 0^+} f(i a t, i b t)=\frac{a}{a+b},
	\]
	so that the two limits are different (unless $a=b$).
\end{ex}

For future reference, we state the following result, which is a direct consequence of the arguments of the proof of~\cite[Lemma 4.4]{MackeyMellon}.

\begin{lem}\label{lem:11}
	Let $e$ be a non-zero tripotent in $Z$, and take $k\in (0,1)$ and $k'>0$. If $x\in Z$, $\norm{x}<1$, $x'\coloneqq \frac{1}{\abs{e}^2}\langle x\vert e \rangle$,  $\norm{x- x' e}\meg k (1-\abs{x'})$, and $\abs{1-x'}\meg k'(1-\abs{x'}^2)$, then $x\in D_{k''}(e)$, where $k''=2 \sqrt 3 k' \frac{1+k}{1-k}$.
\end{lem}

\section{Pluriharmonic Functions on Symmetric Siegel Domains}\label{sec:5}

In this and the following two sections, we shall fix a positive Hermitian Jordan triple system $Z$ and a maximal tripotent $e$ of $Z$. We shall then set $E\coloneqq Z_{1/2}(e)$ and $F=A(e)$, and we shall endow $Z_1(e)=F_\C$ with the structure of a Jordan algebra with product $(z,z')\mapsto \{z,e,z'\}$ and identity $e$. Notice that, interpreting $Z$ as $E\times F_\C$, $e$ corresponds to $(0, e)$. 
We shall denote with $\Omega$  the interior of $\Set{x^2\colon x\in F}$,  we shall set $\Phi\colon E\times E\ni (\zeta,\zeta')\mapsto 2 \{\zeta,\zeta',e\}\in F_\C$, and we shall consider the symmetric Siegel domain
\[
\Ds\coloneqq\Set{(\zeta,z)\in E\times F_\C\colon \Im z -\Phi(\zeta)\in \Omega}.
\]
Notice that the \v Silov boundary of $\bDs$ is then
\[
\bDs = \Set{(\zeta,z)\in E\times F_\C\colon \Im z-\Phi(\zeta)=0}.
\]
We shall denote with $\beta_{\bDs}$ the Hausdorff measure $\Hc^{2n+m}$ on $\bDs$, where $n\coloneqq \dim_\C E$ and $m\coloneqq \dim_\R F$. Observe that $\beta_{\bDs}$ is the image of the Lebesgue measure on $E\times F$ under the homeomorphism $E\times F\ni (\zeta,x)\mapsto (\zeta,x+i \Phi(\zeta))\in \bDs$.

We shall denote with $D$ the bounded domain $\Set{z\in Z\colon \norm{z}<1}$, and with $\gamma\colon \Ds\to D$ a (necessarily birational) biholomorphism such that $\gamma(0, i e)=0$. Observe that $\gamma$ is necessarily the composition of   the (inverse) Cayley transform defined in Proposition~\ref{prop:22} and a linear automorphism of $D$ (since every biholomorphism of $D$ which fixes the origin is necessarily linear by Cartan's uniqueness theorem).
We shall still denote by $\gamma$ and $\gamma^{-1}$ the canonical extensions of $\gamma$ and $\gamma^{-1}$ to $Z$ as (generically defined) rational mappings. 
\begin{deff}
	We denote with $\Cs$ the Cauchy--Szeg\H o kernel of $\Ds$, that is, the reproducing kernel of the Hardy space\footnote{In other words, $\Cs(\,\cdot\,,(\zeta,z))\in H^2(\Ds)$  and $f(\zeta,z)=\langle f\vert\Cs(\,\cdot\,,(\zeta,z))\rangle_{H^2(\Ds)}$ for every $f\in H^2(\Ds)$ and for every $(\zeta,z)\in \Ds$.}
	\[
	H^2(\Ds)=\Set{f\in \Hol(\Ds)\colon  \sup_{h\in \Omega} \int_{\bDs} \abs{f(\zeta,z+i h)}^2\,\dd \beta_{\bDs}(\zeta,z)<\infty }.
	\] 
	In addition, we denote with $\Ps$ the Poisson kernel, that is,
	\[
	\Ps((\zeta,z),(\zeta',z'))=\frac{ \abs{\Cs((\zeta,z),(\zeta',z'))  }^2}{\Cs((\zeta,z),(\zeta,z))}
	\]
	for every $(\zeta,z)\in \Ds$ and for every $(\zeta',z')\in \bDs$.
	
	We also define the Schwarz kernel (associated with the point $(0,i e)$)
	\[
	\Ss((\zeta,z),(\zeta',z'))=2\frac{ \Cs((\zeta,z),(\zeta',z')) \overline{\Cs((0, i e),(\zeta',z'))}}{\Cs((\zeta,z),(0,i e))} - \Ps((0,i e),(\zeta',z'))
	\]
	for every $(\zeta,z)\in \Ds$ and for every $(\zeta',z')\in \bDs$.
\end{deff}

\begin{ex}\label{ex:5}
	Fix a frame of idempotents of $F=A(e)$, and define $(\Delta^{\vect s})_{\vect s\in \C^r}$ accordingly  (cf.~Proposition~\ref{prop:38}). Then, there are $\vect s \Meg\vect 0$ and $c>0$ such that $\Delta^{2 \vect s}$ is a polynomial of degree $2n+2m$ and
	\[
	\Cs((\zeta,z),(\zeta',z'))=c \Delta^{-\vect s}\left( \frac{z-\overline{z'}}{2 i}- \Phi(\zeta,\zeta')  \right)
	\]
	for every $(\zeta,z),(\zeta',z')\in \Ds$ (cf., e.g.,~\cite[Corollary 1.42 and Propositions 2.19 and 2.30]{CalziPeloso}), so that
	\[
	\Ps((\zeta,z),(\zeta',z'))= c\frac{\Delta^{\vect s}( \Im z- \Phi(\zeta))  }{\abs*{\Delta^{2\vect s}\left( \frac{z-\overline{z'}}{2 i}-\Phi(\zeta,\zeta') \right)}}
	\]
	and
	\[
	\Ss((\zeta,z),(\zeta',z'))=2 c \frac{ \Delta^{\vect s}\left( \frac{z+ i e}{2 i} \right) }{\Delta^{\vect s}\left(\frac{z-\overline{z'}}{2 i}- \Phi(\zeta,\zeta')  \right) \Delta^{\vect s}\left(\frac{z'+i e}{2 i}\right)}- c \frac{1}{\abs*{ \Delta^{2 \vect s}\left(\frac{i e - \overline{z'}}{2 i}  \right) }}
	\]
	for every $(\zeta,z)\in \Ds$ and for every $(\zeta',z')\in \bDs$.
\end{ex}

\begin{ex}\label{ex:4}
	If $\Ds$ is the Siegel domain of type $(I_{p,q})$ described in Example~\ref{ex:2}, then there is a constant $c>0$ such that (cf., e.g.,~\cite[Proposition 5.3]{Koranyi} or Example~\ref{ex:5} and the references therein)
	\[
	\Cs((\zeta,z),(\zeta',z'))= c \,{\det}^{-(n+m)/r}\Big( \frac{z-z'^*}{2 i}- \zeta\zeta'^* \Big),
	\]
	for every $(\zeta,z),(\zeta',z')\in \Ds$, so that
	\[
	\Ps((\zeta,z),(\zeta',z'))= c \frac{ \det^{(n+m)/r}\big(\frac{z-z^*}{2 i}\big) }{\abs{\det^{(n+m)/r}\big(\frac{z-z'^*}{2 i}- \zeta\zeta'^* \big)}^2}
	\]
	and
	\[
	\Ss((\zeta,z),(\zeta',z'))=2 c \frac{ \det^{(n+m)/r} \big( \frac{z+i I}{2 i} \big)   }{ \det^{(n+m)/r} \big( \big(\frac{z-z'^*}{2 i}- \zeta\zeta'^*    \big)\frac{ z'+i I}{2 i}   \big)   }-c \frac{ 1 }{\abs{\det^{(n+m)/r}\big(\frac{i I-z'^*}{2 i} \big)}^2}
	\]
	for every $(\zeta,z)\in \Ds$ and for every $(\zeta',z')\in \bDs$.
\end{ex}

\begin{oss}\label{oss:10}
	When $\Ds=\C_+$, then 
	\[
	\Cs(z,z')=\frac{i}{2 \pi (z-\overline{z'})}
	\]
	for every $z,z'\in \C_+$, while
	\[
	\Ps(z,x)=\frac{\Im z}{\pi \abs{z-x}^2}
	\]
	and
	\[
	\Ss(z,x)=\frac{1}{\pi i}\Big( \frac{1}{x-z}-\frac{x}{1+x^2}\Big)
	\]
	for every $z\in \C_+$ and for every $x\in \R$. Notice that the above expressions coincide with the usual Cauchy, Poisson, and Schwarz kernels on $\C_+$. 
	Observe that the general expression we chose for $\Ss$ is quite different from the usual one. 
	Its main advantage lies in the fact that it is relatively clear that $\Ss((\zeta,z),\,\cdot\,)$ has (at least) the same decay as $\Ps((\zeta,z),\,\cdot\,)$, a fact which is obscured in the classical formula by the cancellation between $\frac{1}{x-z}$ and $\frac{x}{1+x^2}$. 
\end{oss}

In the following proposition we collect some consequences of~\cite[Theorem 2.4]{Koranyi}.

\begin{prop}\label{prop:1bis}
	The following hold:
	\begin{enumerate}
		\item[\textnormal{(1)}] $\Ps$ is continuous, bounded, and nowhere vanishing on $\Ds\times \bDs$;
		
		\item[\textnormal{(2)}] $\Ps((\zeta,z),\,\cdot\,)\cdot \beta_{\bDs}$ is a probability measure on $\bDs$ for every $(\zeta,z)\in \Ds$.
	\end{enumerate}
\end{prop}

\begin{lem}\label{lem:7}
	One has:
	\[
	\Cs((\zeta,z),(\zeta',z'))=\frac{ \Cs((\zeta,z),(0,i e))\overline{\Cs((\zeta',z'),(0, i e))}}{\Cs((0, i e),(0, i e))} \Cc(\gamma(\zeta,z),\gamma(\zeta',z')) 
	\]
	for every $(\zeta,z),(\zeta',z')\in \Ds$. In addition,
	\[
	\Ps((\zeta,z),(\zeta',z'))=\Ps((0, i e),(\zeta',z')) \Pc(\gamma(\zeta,z), \gamma(\zeta',z'))
	\]
	and
	\[
	\Ss((\zeta,z),(\zeta',z'))=\Ps((0, i e),(\zeta',z')) (2\Cc(\gamma(\zeta,z), \gamma(\zeta',z'))-1)
	\]
	for every $(\zeta,z)\in \Ds$ and for every $(\zeta',z')\in \bDs$.
\end{lem}

Recall that $\Cc$ and $\Pc$ were defined in Section~\ref{sec:2}.  

\begin{proof}
	Observe that, if $\widehat\gamma$ is the generalized Cayley transform considered in~\cite[\S\ 4]{Koranyi}, then the first two equalities follow from the arguments contained in the cited reference. In addition, there are an affine automorphism $\eta$ of $\Ds$ and a linear automorphism $\eta'$ of $D$ such that $\gamma=\eta'\circ \widehat \gamma\circ \eta$ (this follows from the fact that the affine automorphisms of $\Ds$ act transitively on $\Ds$, and that every biholomorphism of $D$ which fixes the origin is necessarily linear by Cartan's uniqueness theorem). Since there is $\chi(\eta)>0$ such that $\Cs(\eta(\zeta,z),\eta(\zeta',z'))\chi(\eta)=\Cs((\zeta,z),(\zeta',z'))$ for  every $(\zeta,z),(\zeta',z')\in \Ds$ (cf., e.g.,~\cite[Proposition 3.17, the remarks after the statement of Lemma 5.1,  and Proposition 2.29]{CalziPeloso} for an  appropriate choice of $\eta$) and since $\Cc\circ (\eta'\times \eta')=\Cc$ (cf.~\cite[Remark 4.12]{Koranyi}), the first two equalities follow also in the general case. The third equality then follows.
\end{proof}

\begin{cor}\label{cor:7}
	For every $(\zeta,z)\in \Ds$ there is $C>0$ such that
	\[
	\frac{1}{C}\Ps((\zeta,z),\,\cdot\,)\meg \Ps((0, i e),\,\cdot\,)\meg C \Ps((\zeta,z),\,\cdot\,) 
	\]
	on $\Ds$.
\end{cor}

\begin{proof}
	This follows from Lemma~\ref{lem:7}, since $\Pc$ is continuous and nowhere vanishing on $D\times \bD$ by Proposition~\ref{prop:1}. One may also combine Example~\ref{ex:5} with the so-called `Kor\'anyi's lemma (cf., e.g.,~\cite[Theorem 2.47]{CalziPeloso}).
\end{proof}

\begin{deff}
	We define $\Mc_\Poisson(\bDs, \Ds)$ as the set of Radon measures on $\bDs$ such that $\int_{\bDs} \Ps((0, i e),\,\cdot\,) \,\dd \abs{\mi}$ is finite. We shall also write $\Mc_\Poisson(\bDs)$ when no confusion is to be feared. 
	We endow the space $\Mc_\Poisson(\bDs)$ with the weak topology associated with the duality
	\[
	\Mc_\Poisson(\bDs)\times C_\Poisson(\bDs)\ni (\mi, f)\mapsto \int_{\bDs} f\,\dd \mi,
	\]	
	where $C_\Poisson(\bDs)$ is the space of continuous functions  $f$ on $\bDs$ such that $\frac{f}{\Ps((0,i e),\,\cdot\,)}$ is bounded. 
	Since this topology corresponds to the `tight' topology on the space of bounded measures, we shall call it the tight topology on $\Mc_\Poisson(\bDs)$.  Analogously, we define the weak topology $\Mc_\Poisson(\bDs)$ as the weak topology $\sigma(\Mc_\Poisson(\bDs), C_{\Poisson,0}(\bDs))$ (with respect to the previous   pairing), where $C_{\Poisson,0}(\bDs)$ denotes the space of  continuous functions  $f$ on $\bDs$ such that $\frac{f}{\Ps((0,i e),\,\cdot\,)}$ vanishes at the point at infinity of $\bDs$. 
	
	Given $\mi\in \Mc_\Poisson(\bDs)$,  we define 
	\[
	(\Ps \mi)(\zeta,z)=\int_{\bDs} \Ps((\zeta,z),\,\cdot\,)\,\dd \mi(\zeta,z)
	\]
	and
	\[
	(\Ss \mi)(\zeta,z)=\int_{\bDs} \Ss((\zeta,z),\,\cdot\,)\,\dd \mi(\zeta,z)
	\]
	for every $(\zeta,z)\in \Ds$. We shall also write $\Ps f$ and $\Ss f$ when $\mi=f\cdot \beta_{\bDs}$ for some $f\in L^1_\loc(\bDs)$.
\end{deff}

Notice that  a subset $B$ of $\Mc_\Poisson(\bDs)$ is bounded if and only if $\sup_{\mi\in B} \norm{\Ps((0,i e),\,\cdot\,) \cdot \mi }_{\Mc^1(\bDs)}=\sup_{\mi \in B} (\Ps\abs{\mi})(0, i e)$ is bounded: on the one hand, it is clear that this condition is sufficient; on the other hand, this condition is necessary since boundedness in $\Mc_\Poisson(\bDs)$ implies boundedness in the weak topology, which is the weak dual topology when $\Mc_\Poisson(\bDs)$ is interpreted as the dual of $C_{\Poisson,0}(\bDs)$. In particular,  the bounded subsets of $\Mc_\Poisson(\bDs)$ are  relatively compact for the weak topology (but not for the tight topology, in general).

\begin{cor}\label{cor:4}
	Set $N\coloneqq \bD\setminus  \gamma(\bDs)$. Then, the mapping $\Gamma\colon \mi \mapsto   \gamma_*(\Ps((0,i e),\,\cdot\,)\cdot  \mi)$ is a continuous bijection of $\Mc_\Poisson(\bDs)$ onto $\Set{\nu \in \Mc(\bD)\colon \abs{\nu}(N)=0}$, endowed with the vague topology. This mapping induces a homeomorphisms from $\Mc_\Poisson(\bDs)\cap \Mc_+(\bDs)$ onto its image.
\end{cor}

\begin{proof}
	Observe first that the mapping $\mi \mapsto \Ps((0,i e),\,\cdot\,)\cdot  \mi $ is an isomorphism of $\Mc_\Poisson(\bDs)$ onto $\Mc^1(\bDs)$, endowed with the tight topology. 
	Since $  \gamma$ induces a homeomorphism of $\bDs$ onto $\bD\setminus N$, it is clear that $\Gamma$ induces a bijection of $\Mc_\Poisson(\bDs)$ onto $\Mc^1(\bD\setminus N)$. 
	This latter space may then be identified with $\Set{\nu \in \Mc(\bD)\colon \abs{\nu}(N)=0}$. 
	The continuity of $\Gamma$ then follows from the fact that  $  f\mapsto  \Ps((0, i e),\,\cdot\,)(f\circ  \gamma)$ maps $C(\bD)$ into $C_\Poisson(\bDs)$.
	In order to conclude the proof, it will then suffice to show that, if $\Ff$ is a filter on $\Mc_\Poisson(\bDs)\cap \Mc_+(\bDs)$, $\mi\in \Mc_\Poisson (\bDs)\cap \Mc_+(\bDs)$ and $\Gamma(\Ff)$ converges vaguely to $\Gamma(\mi)$, then $\Ff$ converges to $\mi$ in $\Mc_\Poisson(\bDs)$. 
	Then, take $f\in C_\Poisson(\bDs)$, and observe that, if we define $g\colon \bD\to \C $ so that $g=\frac{f}{\Ps((0, i e),\,\cdot\,)}\circ   \gamma^{-1}$ on $\bD\setminus N$ and $g=0$ on $N$, then $g$ is bounded and the set of points of discontinuity of $g$ is (contained in $N$, hence) $\Gamma(\mi')$-negligible for every $\mi'\in \Mc_\Poisson (\bDs)$. 
	In addition, it is clear that $\norm{\Gamma(\mi')}_{\Mc^1(\bD)}=(\Ps \mi')(0, i e) $ converges to  $(\Ps \mi)(0, i e)=\norm{\Gamma(\mi)}_{\Mc^1(\bD)}$ as $\mi'$ runs along $\Ff$. Then~\cite[Proposition 22 of Chapter IV, \S\ 5, No.\ 12]{BourbakiInt1} shows that 
	\[
	\lim_{\mi', \Ff}\int_{\bDs} f\,\dd \mi'=\lim_{\mi', \Ff} \int_{\bD} g\,\dd \Gamma(\mi')= \int_{\bD} g\,\dd \Gamma(\mi)= \int_{\bDs} f\,\dd \mi,
	\]
	whence the conclusion by the arbitrariness of $f$.
\end{proof}

\begin{oss}\label{oss:9}
	With the notation of Corollary~\ref{cor:4},
	\[
	\Pc(\Gamma(\mi))\circ \gamma=\Ps \mi \qquad \text{and} \qquad \Hc(\Gamma(\mi))\circ \gamma=\Ss \mi
	\]
	on $\Ds$, thanks to Lemma~\ref{lem:7}.
\end{oss}

\begin{oss}\label{oss:7}
	With the notation of Corollary~\ref{cor:4}, $\Gamma(\beta_{\bDs})=\beta_{\bD}$.
	
	It suffices to observe that 
	\[
	\Pc(\Gamma(\beta_{\bDs}))\circ \gamma= \Ps(\beta_{\bDs})=1=\Pc(\beta_{\bD})\circ \gamma
	\]
	on $\Ds$, thanks to Remark~\ref{oss:9}.
\end{oss}

\begin{prop}\label{prop:18}
	Take a real measure $\mi $ in $ \Mc_\Poisson(\bDs)$. Then, the following conditions are equivalent:
	\begin{enumerate}
		\item[\textnormal{(1)}] $\Ps\mi$ is pluriharmonic;
		
		\item[\textnormal{(2)}]  $\Ps \mi=\Re \Ss \mi$;
		
		\item[\textnormal{(3)}] $\mi$ belongs to the closure of $\Re    H^\infty(\bDs)\cdot \beta_{\bDs}$ in $\Mc_\Poisson(\bDs)$.
	\end{enumerate}
\end{prop}

Here, we denote by $H^\infty(\bDs)$ the space of boundary values of bounded holomorphic functions on $\Ds$. 
Notice that in condition (3) we require $\mi$ to be approximated by the (bounndary values of the) real parts of holomorphic functions which are at least bounded in $\Ds$. In particular, it is \emph{not} sufficient to approximate $\mi$ in $\Mc_\Poisson(\bDs)$ with measures of the form $f\cdot \beta_{\bDs}$, where $f$ is the real part of some holomorphic function $g$ defined on $\Ds- (0, i h)$ for some $h\in\Omega$ (cf.~Example~\ref{ex:1bis} and Theorem~\ref{teo:4}). In a sense, this is due to the fact that $\Im \cdot \beta_{\bDs}$ does \emph{not} belong to $\Mc_\Poisson(\bDs)$, so that one need \emph{not} be able to approximate $g$ with bounded holomorphic functions in such a way that the corresponding real parts (`multiplied' by $\beta_{\bDs}$) converge in $\Mc_\Poisson(\bDs)$.

Notice that condition (3) is also equivalent to requiring that $\mi$ belongs to the closure of $\Re [ \Sc(\bDs)\cap  H^\infty(\bDs)]\cdot \beta_{\bDs}$ in $\Mc_\Poisson(\bDs)$. In order to see this fact, observe that $\Sc(\bDs)\cap H^\infty(\bDs)$ is dense in $H^\infty(\bDs)$ in the weak topology $\sigma(H^\infty(\bDs), L^1(\beta_\bDs))$ (argue, e.g., as in the proof of~\cite[Lemma 8.1]{OgdenVagi}), and that  the canonical mapping  $H^\infty(\bDs)\to \Mc_\Poisson(\bDs)$ is continuous for the topology $\sigma(H^\infty(\bDs), L^1(\beta_\bDs))$ since $\Ps((0, i e),\,\cdot\,)\in L^1(\beta_{\bDs})$.

\begin{oss}\label{oss:13}
Notice that we do not provide analogues of the condition (4) of Proposition~\ref{prop:4}, since we have not found any reasonably simple class of functions which is (weakly) dense in the polar (in $C_\Poisson(\bDs)$) of the space of pluriharmonic measures in $\Mc_\Poisson(\bDs)$. 
A tempting choice could be the following one:
\begin{enumerate}
	\item[\textnormal{(4)}] $\int_{\bDs} \phi\,\dd \mi=0$ for every $\phi\in \Sc(\bDs)\cap (L^2(\beta_{\bDs})\ominus \Re H^2(\bDs))$, where $\Sc(\bDs)$  is defined as the push-forward of $\Sc(E\times F)$ under the diffeomorphism $E\times F\ni(\zeta,x)\mapsto (\zeta,x+i\Phi(\zeta))\in \bDs$,  while $H^2(\bDs)$ denotes the space of boundary values of the Hardy space $H^2(\Ds)$.
\end{enumerate} 
When  $\Ds$ is of tube type, this condition is equivalent to the following one:
\begin{enumerate}
	\item[\textnormal{(4$'$)}] $\Fc \mi$ is supported in $-\overline{\Omega}\cup \overline\Omega$, where $\Fc$ denotes the Fourier transform on $\bDs=F$.
\end{enumerate}
This follows from the fact that $H^2(\bDs)=\Set{f \in L^2(\beta_{\bDs})\colon \Supp{ f\cdot \beta_{\bDs} }\subseteq \overline\Omega}$ (cf., e.g.,~\cite[Theorem 3.1]{SteinWeiss2}), and that conditions on the support of the Fourier transform may be described by testing against Schwartz functions. 
When $\Ds=\C_+^k$, $k\in\N$, condition (4$'$) has been shown to characterize pluriharmonic measures on $\R^k=\bDs$ in~\cite[Theorem 3.2]{NedicSaksman}.\footnote{The cited result only deals with positive  Radon measures, but the proof extends to general Radon measures.}
In addition, (4$'$) may be used to characterize pluriharmonic measures when $\mi\in \Mc^1(\bDs)$ or $\mi\in L^p(\beta_{\bDs})\cdot \beta_{\bDs}$ for $p\in [1,2]$ (roughly speaking, because the Fourier transform of $\mi$ does not `see' the boundary of $\Omega$). 
In addition, using (3) of Proposition~\ref{prop:18}, it is not difficult to show that, if $\mi$ is a pluriharmonic measure on $\bDs$, then condition (4$'$) holds (when $\Ds$ is tubular).
Nonetheless, as Example~\ref{ex:1bis} shows, (4$'$) does \emph{not} characterize pluriharmonic measures on any tubular irreducible symmetric Siegel domain of rank $2$. It is possible that this behaviour should be related to the fact that $\chi_\Omega$ is \emph{not} a Fourier multiplier of $L^p(F)$, $p\neq 2$ unless $\Ds=\C^k$ (when $\Ds$ is tubular; cf.~\cite{Fefferman}).

Because of Theorem~\ref{teo:2}, one may still wonder whether condition (4)  characterizes plurihamonic measures when $\Ds$ has no irreducible \emph{tubular} factor of rank $\Meg 2$, even though in this case one no longer has the analogue of condition (4$'$). As a matter of fact, allowing not necessarily tubular domains poses severe restrictions on the class of distributions for which the Fourier transform may be computed (roughly speaking, only measures in $\Mc^1(\bDs)+ L^2(\beta_{\bDs})\cdot \beta_{\bDs}$, with the available techniques\footnote{As a matter of fact, when $\Ds$ has rank $1$ and $\bDs$ is then identified with a Heisenberg group, the Fourier transform has been successfully extended to general tempered distributions, cf.~\cite{BCD}, but is still quite hard to handle.}), since $\bDs$ becomes a \emph{non-commutative} nilpotent Lie group. 

One may, of course, simply transfer the polynomial conditions determined in Proposition~\ref{prop:4} through $\gamma$, but the resulting conditions would be poorly intelligible.
Cf.~\cite{LugerNedic1,LugerNedic2} for a lengthier discussion of this problem when $\Ds$ is a product of upper half-planes. Cf.~also~\cite{LugerNedic3,Nedic,NedicSaksman} for further charaterizations and deeper properties of positive pluriharmonic measures when  $\Ds$ is a product of upper half-planes, and~\cite{Vladimirov} for an  involved characterization of positive pluriharmonic measures on tube domains.

\end{oss}

\begin{proof}
	The equivalence of (1) and (2) follows from Proposition~\ref{prop:4} and Remark~\ref{oss:9}. In addition, observe that (1) is equivalent to saying that $\Gamma(\mi)$ is a pluriharmonic measure on $\bD$, with the notation of Corollary~\ref{cor:4}. 
	
	Define $\Gamma'$ as the correspondence $ f \mapsto   f\circ   \gamma^{-1} $ (extended by $0$ on $N$) from functions on $\bDs$ to functions on $\bD$, and observe that $\Gamma(f\cdot \beta_{\bDs})=\Gamma'(f)\cdot \beta_{\bD}$ when $f\cdot \beta_{\bDs}\in \Mc_\Poisson(\bDs)$, thanks to Remark~\ref{oss:7}. 
	In particular, if $f\in \Re  H^\infty(\bDs)$, then $\Gamma'(f)\in \Re H^\infty(\bD)$, so that
	\[
	\Ps f =\Pc(\Gamma' (f))\circ \gamma=\Re \Hc(\Gamma'(f))\circ \gamma=\Re \Ss f
	\]
	by Remark~\ref{oss:9}. By continuity, this shows that  (3) implies  (2).
	
	Conversely, assume that (1) holds and let us prove that (3) holds. Observe that
	\[
	(\Pc \Gamma(\mi))_{1-2^{-j}}\cdot \beta_{\bD}
	\]
	converges vaguely to $\Gamma(\mi)$ for $j\to \infty$. 
	Set $f_j\coloneqq (\Pc \Gamma(\mi))_{1-2^{-j}}\circ  \gamma$, so that $\Gamma(f_j\cdot \beta_{\bDs})=(\Pc \Gamma(\mi))_{1-2^{-j}}\cdot \beta_{\bD} $ for every $j\in \N$. 
	Let us prove that $\Set{f_j\cdot \beta_{\bDs}\colon j\in\N}$ is relatively compact in $\Mc_\Poisson(\bDs)$. By~\cite[Theorem 1 of Chapter IX, \S\ 5, No.\ 5]{BourbakiInt2}, it will suffice to observe that 
	\[
	\sup_{j\in \N} \int_{\bDs} \abs{f_j} \Ps((0, i e),\,\cdot\,)\,\dd\beta_{\bDs}=\sup_{j\in \N}  \norm{(\Pc \Gamma( \mi))_{1-2^{-j}}}_{L^1(\beta_{\bD})}= \norm{\Gamma(\mi)}_{\Mc^1(\bD)}<+\infty,
	\]
	thanks to Remark~\ref{oss:7},
	and to prove that for every $\eps>0$ there is a compact subset $K_\eps$ of $\bDs$ such that, for every $j\in \N$,
	\[
	\int_{\bDs \setminus K_\eps} \abs{f_j }\Ps((0, i e),\,\cdot\,)\,\dd\beta_{\bDs} <\eps.
	\]
	Assume by contradiction that there are $\eps>0$, a covering $(U_k)$ of $\bDs$ by relatively compact open subsets of $\bDs$, with $\overline{U_k}\subseteq U_{k+1}$ for every $k\in\N$, and a strictly increasing sequence $(j_k)_{k\in \N}$ such that
	\[
	\int_{\bD\setminus   \gamma(U_k)} \abs{(\Pc \mi)_{1-2^{-j_k}}}\,\dd \beta_{\bD} =\int_{\bDs \setminus U_k} \abs{f_{j_k} }\Ps((0, i e),\,\cdot\,)\,\dd\beta_{\bDs} \Meg \eps
	\]
	for every $k\in\N$ (cf.~Remark~\ref{oss:7}). 
	Observe that $\gamma(\bDs)$ is open in $\bD$, since $\bD \setminus \gamma(\bDs)$ is the intersection of $\bD$ with the (closed) algebraic set where $\gamma^{-1}$ is not defined. 
	Observe that, by the previous remarks,
	\[
	\lim_{j\to \infty}  \int_{\bD} \abs{(\Pc \mi)_{1-2^{-j}}}\,\dd \beta_{\bD}=\norm{\Gamma(\mi)}_{\Mc^1(\bD)},
	\]
	while, by semicontinuity,
	\[
	\abs{ \Gamma(\mi)}( \gamma(U_{k'}))\meg \liminf_{k\to \infty} \int_{\gamma(U_{k'})} \abs{(\Pc \mi)_{1-2^{-j_k}}}\,\dd \beta_{\bD}\meg \liminf_{k\to \infty} \int_{  \gamma(U_k)}  \abs{(\Pc \mi)_{1-2^{-j_k}}}\,\dd \beta_{\bD}\meg \norm{\Gamma(\mi)}_{\Mc^1(\bD)}-\eps
	\]
	for every $k'\in \N$. Passing to the limit for $k'\to \infty$, this shows that
	\[
	\norm{\Gamma(\mi)}_{\Mc^1(\bD)}=\abs{ \Gamma(\mi)}(\bD\setminus N)\meg \norm{\Gamma(\mi)}_{\Mc^1(\bD)}-\eps
	\]
	since $\abs{\Gamma(\mi)}(N)=0$: contradiction. 
	Thus, $\Set{f_j\cdot \beta_{\bDs}\colon j\in\N}$ is relatively compact in $\Mc_\Poisson(\bDs)$. 
	Since clearly $f_j\cdot \beta_{\bDs}$ converges vaguely to $\mi$, this is sufficient to conclude that $f_j\cdot \beta_{\bDs}$ converges to $\mi$ in $\Mc_\Poisson(\bDs)$. Since  $f_j=\Re[(\Hc \Gamma(\mi))_{1-2^{-j}}\circ \gamma]\in \Re H^\infty(\bDs)$ for every $j\in\N$, this shows that (3) holds.
\end{proof}

\begin{deff}
	We say that a measure $\mi\in \Mc_\Poisson(\bDs)$ is pluriharmonic if it satisfies the equivalent conditions of Proposition~\ref{prop:18}.
\end{deff}

One may then consider disintegration theorems for pluriharmonic measures in $L^p(\beta_{\bDs})\cdot \beta_{\bDs}$, $p\in [1,\infty]$, and also in $\Mc^1(\bDs)$ (the latter case requires more care), in analogy with Proposition~\ref{prop:25}. Since our main interest is on positive plurihamonic measures, we shall not pursue this task.

\section{Positive Pluriharmonic Functions on Symmetric Siegel Domains}\label{sec:6}

\begin{teo}\label{teo:2}
	Assume that $\Ds$  has no irreducible \emph{tubular} factor of rank $\Meg 2$, and let $f$ be a positive pluriharmonic function on $\Ds$. Then, there are a unique positive  pluriharmonic measure $\mi$ in $\Mc_\Poisson(\bDs)$ and a unique $\lambda\in \overline{\Omega}\cap \Phi(E)^\perp$ such that  
	\[
	f(\zeta,z)=  \langle \Im z\vert \lambda \rangle+ (\Ps\mi)(\zeta,z)
	\]
	for every $(\zeta,z)\in \Ds$. 
\end{teo}

Notice that   $\mi$ coincides with the boundary value measure of $f$ in the sense of Theorem~\ref{teo:3}.  

The proof of the last step is conceptually similar to that of~\cite{LugerNedic1,LugerNedic2} (which deal with the case in which $\Ds$ is a product of upper half-planes), even though somewhat simplified by the use of disintegrations. As shown in Example~\ref{ex:1},   it is not possible to extend this result to general symmetric domains: counterexamples arise in every tubular irreducible symmetric Siegel domain of rank $2$. Cf.~Corollary~\ref{cor:10}, though.

We first need two simple lemmas. The first one is most likely well known, but we include a proof for lack of a reference.

\begin{lem}\label{lem:6}
	Let $V$ be a (complex) algebraic subset of $F_\C$. Then, $\dim_\HB (V\cap F)=\dim_\R (V\cap F)\meg \dim_\C V=2\dim_{\HB} V$, where $\dim_\HB$ denotes the Hausdorff--Besicovitch dimension, $\dim_\R (V\cap F)$ denotes the Krull dimension of the ring of (real) polynomial functions on $V\cap F$, and $\dim_\C V$ denotes the Krull dimension of the ring of (holomorphic) polynomial functions on $V$.
\end{lem}

\begin{proof}
	Let $\Ic(V)$ be the ideal of holomorphic polynomials on $F_\C$ which vanish on $V$, so that the ring $\C[V]$ of (holomorphic) polynomial functions on $V$ is $\C[F_\C]/\Ic(V)$, where $\C[F_\C]$ denotes the ring of holomorphic polynomials on $F_\C$. Analogously, denote by $\Ic(V\cap F)$ the ideal of the ring $\R[F]$ of real polynomials on $F$ which vanish on $V\cap F$, so that $\R[V\cap F]\coloneqq \R[F]/\Ic(V\cap F)$ is the ring of polynomial functions on $V\cap F$. Observe that the set $I$ of restrictions to $F$ of the polynomials $\Re p$, for $ p \in \Ic(V)$, is an ideal of $\R[F]$ contained in $\Ic(V\cap F)$. In addition, if we identify $\C[F]$ canonically with the complexification of $\R[F]$, then the complexification $I_\C$ of $I$ becomes an ideal in $\C[F]$ which contains $\Ic(V)$ (explicitly, $I_\C=\Set{p+q^*\colon p,q\in \Ic(V)}$, where $q^*(z)=\overline{q(\overline z)}$). 
	Then,~\cite[Corollary to Proposition 5 of Chapter VIII, \S\ 2, No.4]{BourbakiAC3} shows that $\dim_\C V=\dim \C[V]\Meg \dim \C[F_\C]/I_\C=\dim \R[F]/I\Meg \dim \R[V\cap F]=\dim_\R V\cap F$. The equality $\dim_\R V\cap F=\dim_\HB V\cap F$ follows (for example) from~\cite[Corollary 2.8.9 and Proposition 2.9.10]{BCR}, whereas the equality  $\dim_\C V=2 \dim_\R V(=2\dim_\HB V)$ follows from~\cite[Proposition 3.1.1]{BCR}.
\end{proof}

\begin{cor}\label{cor:5}
	Let $V$ be a complex algebraic subset of $  F_\C$. Then, $\dim_\R [(E\times V)\cap \bD]\meg 2n+\dim_\C V$. 
\end{cor}

\begin{proof} 
	It will suffice to show that the (real) local dimension of $V\cap \bD$ is $\meg \dim_\C V$ at every $(\zeta_0,z_0)\in \bD\cap V$.  
	Let $D_0$ be the bounded symmetric  domain  defined by the Jordan triple system induced on  $F_\C$, and observe that $z_0$ belongs to the closure of $D_0$ by Remark~\ref{oss:12}. 
	In addition, we may find a maximal tripotent $e$ of $F_\C$ such that $e-z_0$ is invertible in the Jordan algebra with product $(z,z')\mapsto\{z,e,z'\}$.\footnote{It suffices to take the spectral decomposition $z_0=\sum_{j=1}^k \lambda_j e_j$ (Proposition~\ref{prop:20}), to fix a tripotent $e_0$  orthogonal to $e_1,\dots, e_k$ and such that $e_0+\cdots+e_k$ is a maximal tripotent, and to set $e= -(e_0+\cdots+e_k)$. 
		In fact, $z_0-e=\sum_{j=0}^k (\lambda_j+1) e_j$ (setting $\lambda_0\coloneqq 0$), so that  $\sum_{j=0}^k (\lambda_j+1)^{-1} e_j$ is an inverse of $z_0-e$ in the complex algebra generated by $z_0-e$ (which is the complex vector space generated by $\sum_j (\lambda_j+1)^h e_j$, $h\in \N$, that is, $\sum_j \C e_j$). } 
	Let $\gamma\colon \Ds \to D$ and  $\gamma_0\colon F+i \Omega\to D_0$ be the (inverse) Cayley transforms associated with $(0,e)$ and $e$ (cf.~Proposition~\ref{prop:22}). 
	Then, $\gamma^{-1}(E\times V)=E\times \gamma^{-1}_0(V)$, so that we are reduced to proving that $\dim_\R [(E\times V')\cap \bDs]\meg 2 n+\dim_\C V'$, where $V'\coloneqq \gamma^{-1}_0(V)$ (so that $\dim_\C V'\meg \dim_\C V$ since $\gamma_0$ is birational). 
	Observe that, for every $\zeta\in E$,
	\[
	(\Set{\zeta}\times V')\cap \bDs=\Set{\zeta}\times (V'\cap (F+i \Phi(\zeta))), 
	\]
	so that
	\[
	\dim_\R [(\Set{\zeta}\times V')\cap \bDs]= \dim_\R [V'\cap (F+i \Phi(\zeta))]\meg \dim_\C V'
	\]
	by Lemma~\ref{lem:6}. By the arbitrariness of $\zeta$, this proves that\footnote{One may, for instance, argue that, if $k=\dim_\R [(E\times V')\cap \bDs]$, then $W=(E\times V')\cap \bDs$ is a closed subset of $E\times F_\C$ which is countably $\Hc^k$-rectifiable, thanks to~\cite[Proposition 2.9.10]{BCR}. Then, the coarea formula (cf.~\cite[Theorem 3.2.22]{Federer}) shows that     the sets $(\Set{\zeta}\times V')\cap \bDs$ are countably $\Hc^{k-h}$-rectifiable (and of dimension $k-h$) for $\Hc^h$-almost every $\zeta\in \pr_E W$, where $\pr_E\colon E\times F_\C\to E$ denotes the canonical projection and $h=\dim_\HB \pr_E W\meg 2 n$. Thus, $\dim_\R[(E\times V')\cap \bDs]=k \meg 2 n+ \dim_\C V'$.   }
	\[
	\dim_\R [(E\times V')\cap \bDs]\meg 2 n + \dim_\C V',
	\]
	whence the conclusion by the above discussion.
\end{proof}

\begin{proof}[Proof of Theorem~\ref{teo:2}.]
	\textsc{Step I.} By assumption, (up to a rotation) there are $k\in\N$ and a symmetric Siegel domain $\Ds_0$ such that $\Ds=\Ds_0\times \C_+^k$, and such that no irreducible factor of $\Ds_0$ is of tube type. Then, we may choose $\gamma$ so that
	\[
	\gamma((\zeta,z), (w_{j})_{j=1,\dots, k})= \bigg( \gamma_0(\zeta,z), \Big(\frac{w_{j}-i}{w_{j}+i}  \Big)_{j=1,\dots, k}  \bigg) 
	\]
	for every $((\zeta,z), (w_{j})_{j=1,\dots, k})\in \Ds$ for some (birational) biholomorphism $\gamma_0$ of $\Ds_0$ onto some bounded convex and circular symmetric domain $D_0$ such that $D=D_0\times \Db^k$.
	Then, $f\circ \gamma^{-1}$ is a positive pluriharmonic function on $D$, so that Proposition~\ref{prop:5} shows that there is a unique positive pluriharmonic measure $\nu$ on $\bD$ such that $\Pc(\nu)=f\circ \gamma^{-1}$. 
	Define $N\coloneqq \bD\setminus \gamma(\bDs)$,
	and 
	\[
	\mi\coloneqq \Gamma^{-1}(\chi_{\bD\setminus N}\cdot\nu),
	\]
	with the notation of Corollary~\ref{cor:4}. Then, $\mi$ is positive, belongs to $\Mc_\Poisson(\bDs)$, and Remark~\ref{oss:9} shows that
	\[
	\begin{split}
		f(\zeta,z)&=\int_{\bD} \Pc(\gamma(\zeta,z),(\zeta',z')) \,\dd \nu(\zeta',z')\\
		&=\int_N \Pc(\gamma(\zeta,z),(\zeta',z')) \,\dd \nu(\zeta',z')+\int_{\bDs} \Ps((\zeta,z),\,\cdot\,)\,\dd \mi\\
		&=  g(\zeta,z)+ (\Ps \mi)(\zeta,z)
	\end{split}
	\]
	for every $(\zeta,z)\in \Ds$, where
	\[
	g(\zeta,z)\coloneqq \int_N \Pc(\gamma(\zeta,z),(\zeta',z')) \,\dd \nu(\zeta',z').
	\]
	
	\textsc{Step II.} Assume that $\Ds$ is irreducible and that $E\neq \Set{0}$, and let us prove that $\dim_\HB N\meg 2n+m-2$ or, equivalently, that $\dim_\HB (-N)\meg 2n+m-2$. 
	Let us first prove that $\dim_\HB (N\cap (-N))\meg 2n+m-2$. 
	Observe first that $N\cap (- N)$ is the intersection of $\bD$ with $E\times V$, where $V$ is the complex algebraic set defined by the two polynomials $z\mapsto\det(e-z)$ and $z\mapsto\det(e+z)$, where $\det$ denotes the determinant polynomial of the Jordan algebra  $F_\C$ (with product $(x,y)\mapsto \{x,e,y\}$ and identity $e$). Observe that  $F_\C$ is simple since $\Ds$ is irreducible (this can be checked directly by means of the classification of irreducible symmetric domains).  Since these two polynomials are distinct and irreducible (cf.~Proposition~\ref{prop:23}), $\dim_\C V \meg m-2$, so that $\dim_\HB(N\cap (-N))=\dim_\R((E\times V)\cap \bD)\meg 2n+ m-2$, thanks to Corollary~\ref{cor:5}.
	
	Therefore, in order to prove our claim, it will suffice to prove that $\dim_\HB ((-N)\setminus N)\meg 2n+m-2$. 
	It will therefore suffice to show that $\dim_\HB \widetilde\gamma^{-1}(-N)\meg 2n+m-2$, where $\widetilde \gamma\colon E\times F\ni (\zeta,x)\mapsto \gamma(\zeta,x+ i \Phi(\zeta))$. 
	Assume for simplicity that   $\gamma$ is the Cayley transform associated with $e$, and endow $E\times F_\C$ with the structure of a Jordan algebra with product $(x,y)\mapsto \{x, e, y\}$. 
	Then, $N'\coloneqq\widetilde\gamma^{-1}(-N)$ is the set of points of $E\times F$ where the inversion 
	\[
	(\zeta,x)\mapsto (-2 i (x+ i\Phi(\zeta))^{-1}\zeta,-(x+ i\Phi(\zeta))^{-1})
	\]
	is not defined (as a rational function). In other words,
	\[
	N'\coloneqq \Set{(\zeta,x)\in E\times F\colon \det(x+i \Phi(\zeta))=0}.
	\]
	As we observed earlier, $\det$ is an irreducible holomorphic polynomial on $F_\C$; let us prove that also $q(\zeta,x)\coloneqq \det(x+ i \Phi(\zeta))$ is irreducible (as a complex polynomial on the real space $E\times F$). To see this, take two polynomials $q_1,q_1\colon E\times F\to \C$ such that $q=q_1 q_2$. 
	Observe that, for every $\zeta\in E$, $q_1(\zeta,\,\cdot\,)q_2(\zeta,\,\cdot\,)=\det(\,\cdot\,+i \Phi(\zeta))$. 
	If $\widetilde q_{1,\zeta}$ and $\widetilde q_{2,\zeta}$ denote the holomorphic extensions of $q_1(\zeta,\,\cdot\,)$ and $q_2(\zeta,\,\cdot\,)$, respectively, to $F_\C$, then $\det=\widetilde q_{1,\zeta}(\,\cdot\,-i\Phi(\zeta))\widetilde q_{2,\zeta}(\,\cdot\,-i \Phi(\zeta))$ on $F_\C$. Since $\det$ is irreducible, there is $j_\zeta\in \Set{1,2}$ such that $\widetilde q_{j_\zeta,\zeta}(\,\cdot\,-i\Phi(\zeta))$ is constant on $F_\C$. 
	Hence, $q_{j_\zeta}$ is constant on $\Set{\zeta}\times F$, that is, $(\partial_F q_{j_\zeta})(\zeta,\,\cdot\,)=0$ on $F$. 
	Consequently, if we define $C_j=\Set{\zeta\in E\colon (\partial_F q_j)(\zeta,\,\cdot\,)=0 \text{ on $F_\C$}}$, then $C_1$ and $C_2$ are (real) algebraic sets such that $E=C_1\cup C_2$, so that either $C_1=E$ or $C_2=E$. 
	Assume, for the sake of definiteness, that $C_1=E$. Then, $q(\zeta,x)=\widetilde q_1(\zeta) q_2(\zeta,x)$ for every $(\zeta,x)\in E\times F$, where $\widetilde q_1=q_1(\,\cdot\,,0)$. Let $d$, $d_1$, and $d_2$ be the degrees of $q$, $q_1$, and $q_2$, respectively, so that $d=d_1+d_2$. Then,
	\[
	t^d=\det(t e)=q(0, t e)=\widetilde q_1(0) q_2(\zeta,t e)=O(t^{d_2})
	\] 
	for $t\to \infty$, so that $d=d_2$ and $d_1=0$. Hence, $q_1$ is constant. By the arbitrariness of $q_1$ and $q_2$, this proves that $q$ is irreducible (as a complex polynomial on the real space $E\times F$).
	Now, let us prove that $\Re q$ and $\Im q$ are  coprime (as complex polynomials on the real space $E\times F$); equivalently, let us prove that their holomorphic extensions $\widetilde{\Re q}$ and $\widetilde{\Im q}$ on $(E\times F)_\C$ are coprime. 
	Take polynomials  $p,q_3,q_4\colon E\times F\to \C$ so that $\Re q=p q_3$ and $\Im q=p q_4$. Then, $q=p (q_3+i q_4)$. Since $q$ is irreducible, this proves that either $p$ is constant or $q_3+i q_4$ is constant, in which case both $q_3$ and $q_4$ must be constant. Since clearly both $\Re q$ and $\Im q$ are non-trivial, also $q_3$ and $q_4$ must be non-zero constants in the latter case. In any case, the decompositions $\Re q=p q_3$ and $\Im q=p q_4$ are both trivial. 
	Therefore, the joint zero locus $V$ of $\widetilde{\Re q}$ and $\widetilde{\Im q}$ has (complex) codimension $2$ in $(E\times F)_{\C}$, so that $N'=(E\times F)\cap V$ has (real) codimension $\Meg 2$, thanks to Lemma~\ref{lem:6}. This concludes the proof of our claim. 
	
	\textsc{Step III.} Assume that $k=0$, that is (up to a rotation), $\Ds=\prod_{j=1}^h \Ds_j$ where $\Ds_1,\dots, \Ds_h$ are irreducible symmetric Siegel domains which are \emph{not} tubular. 
	Let $D=\prod_{j=1}^h D_j$ be the corresponding decomposition of $D$ into irreducible factors; we may assume that $\gamma=\prod_{j=1}^h \gamma_j$, where the $\gamma_j\colon \Ds_j\to D_j$ are birational biholomorphisms. 
	If we set $N_j\coloneqq \bD_j\setminus\gamma_j(\bDs_j)$, then $N=\bigcup_{j=1}^h \prod_{j'=1}^h N_{j,j'}$, where $N_{j,j'}=\bD_{j'}$ if $j\neq j'$, while $N_{j,j}=N_j$, for every $j,j'=1,\dots,h$. 
	Now,~\textsc{step II} shows that $\dim_\HB N_j\meg \dim_\HB \bD_j-2$ for every $j=1,\dots,h$, so that $\dim_\HB N\meg 2 n+m-2$ (since the sets $N_{j,j'}$ are real algebraic sets, hence sufficiently regular for the product formula to hold, cf., e.g.,~\cite[Proposition 2.8.5]{BCR}).

	\textsc{Step IV.} Now, consider the general case.  
	Observe that, with the notation of~\textsc{step I}, $N=(N_0\times \T^k)\cup \bigcup_{j=1}^k (\bD_0 \times \pr_j^{-1}(\Set{1}) )$, where $\pr_j\colon \T^k\to \T$ is the projection on the $j$-th component, for $j=1,\dots, k$, and $N_0$ is a real algebraic subset of $\bD_0$ with $\dim_\HB N_0\meg \dim_\HB \bD_0-2$ (cf.~\textsc{step III}). Then,  $\dim_\HB (N_0\times \T^k)\meg\dim_\HB \bD-2$, so that $\nu(N_0\times \T^k)=0$. 
	
	Then, fix $j\in \Set{1,\dots, k}$ and let us prove that $\chi_{N_j}\cdot \nu=c_j \beta_j$, where $N_j=\bD_0\times \pr_j^{-1}(\Set{1})$, $c_j\Meg 0$, and $\beta_j$ is the canonical measure on $N_j$ (considered as the \v Silov boundary of the symmetric domain $\Set{( z, (w_{j'})_{j'=1,\dots,k})\colon z\in D_0, (w_{j'})_{j'\neq j}\in \Db^{k-1}, w_j=1  }$).\footnote{In the bounded realization there is in general no reason to distinguish between the $E$ and $F_\C$ components, so that we write $z$ to denote the elements of $D_0$ and $\zeta$ to denote the elements of $\bD_0$.}
	To this aim, observe that  by Proposition~\ref{prop:6}   $\nu$ has a  disintegration $(\nu_{j,\alpha})_{\alpha\in\T}$ under   $\pr_j'\colon \bD_0\times \T^k\ni ( \zeta,(\alpha_{j'})_{j'=1,\dots,k})\mapsto \alpha_j\in \T$. Consequently, if $\nu'_j=(\pr_j')_*(\nu)$, then
	\[
	(\Hc\nu)(z,(w_{j'})_{j'=1,\dots, k})=\int_{ \T } (\Hc \nu_{j,\alpha})(z,(w_{j'})_{j'=1,\dots,k})\,\dd \nu'_j(\alpha)
	\]
	for every $z\in D_0$ and for every $w_1,\dots, w_k\in \Db$.  Observe that, if we identify $\bD$ with $(\bD_0\times \T^{k-1})\times \T$ by means of the mapping $(\zeta, (\alpha_{j'})_{j'})\mapsto ((\zeta, (\alpha_{j'})_{j'\neq j}), \alpha_j)$, then  $\nu_{j,\alpha}=\nu'_{j,\alpha}\otimes \delta_\alpha$ for some measure $\nu'_{j,\alpha}$ on $\bD_0\times \T^{k-1}$, for $\nu'_j$-almost every $\alpha\in \T$. 
	Denoting by $\Hc'$ the Herglotz integral operator corresponding to $\bD_0\times \T^{k-1}$, this proves that
	\[
	(\Hc\nu)(z,(w_{j'})_{j'})=\int_{ \T } (\Hc'\nu'_{j,\alpha})(z,(w_{j'})_{j'\neq j}) \frac{\alpha+ w_j}{\alpha-w_j} \,\dd \nu'_j(\alpha)
	\]
	for every $z\in D_0$ and for every $w_1,\dots, w_k\in \Db$. Now, observe that, by dominated convergence,
	\[
	\lim_{\rho\to 1^-}(1-\rho)\int_{ \T } (\Hc'\nu'_{j,\alpha})(z,(w_{j'})_{j'\neq j}) \frac{\alpha+\rho}{\alpha-\rho} \,\dd \nu'_j(\alpha)=2 \Hc'(\nu'_{j,1})(z,(w_{j'})_{j'\neq j}) \nu'_j(\Set{1})
	\]
	for every $z\in D_0$ and for every $(w_{j'})_{j'\neq j}\in \Db^{k-1}$.
	By Remark~\ref{oss:8} (applied to the $j$-th $\Db$ in $D=D_0\times \Db^k$), this proves that $\Hc'(\nu'_{j,1})(z,(w_{j'})_{j'\neq j}) \nu'_j(\Set{1})\Meg 0$ for every $z\in D_0$ and for every $(w_{j'})_{j'\neq j}\in \Db^{k-1}$. 
	By holomorphy,  there is $c_j\Meg 0$ such that $(\Hc' \nu'_{j,1})(z,(w_{j'})_{j'\neq j}) \nu'_j(\Set{1})=c_j$ for every $z\in D_0$ and for every $(w_{j'})_{j'\neq j}\in \Db^{k-1}$. Consequently, $\chi_{N_j}\cdot \nu=\nu_{j, 1}\nu'(\Set{1})=c_j\beta_j$, as claimed. 
	
	Now, observe that $(\Pc \beta_j)(z, (w_{j'}))=\frac{1-\abs{w_j}^2}{\abs{1-w_j}^2}$ for every $z\in D_0$, for every $w_1,\dots, w_k\in \Db$, and for every $j=1,\dots,k$. It then follows that	
	\[
	g((\zeta,z),(w_j)_{j=1,\dots,k})=\sum_{j=1}^k c_j \Im w_j
	\]
	for every $(\zeta,z)\in \Ds_0$ and for every $w_1,\dots, w_k\in \C_+$, so that $g$ is positive and pluriharmonic. Consequently, also $\Ps \mi$ is  pluriharmonic, so that $\mi$ is a positive pluriharmonic measure on $\bDs$. 
	To conclude, it suffices to define $\lambda$ as the elements $(0,(c_1,\dots,c_k))$ of $F_0\times \C^k$, where $F_0$ is chosen so that $\Ds_0$ is an open subset of $E\times F_0$, so that $\lambda\in \overline{\Omega}\cap \Phi(E)^\perp$ and $g(\zeta,z)=\langle \Im z\vert \lambda \rangle$ for every $(\zeta,z)\in \Ds$.
\end{proof}

\begin{cor}\label{cor:12}
	Assume that $\Ds$ has no irreducible \emph{tubular} factors of rank $\Meg 2$, and let $M$ be a bounded subset of $\Mc_\Poisson(\bDs)$ whose elements are positive and pluriharmonic. Then, the vague closure of $M$ consists of positive pluriharmonic measures on $\bDs$.
\end{cor}

We do not know if the assumption on $\Ds$ is necessary.

\begin{proof}
	Take $\mi_0$ in the vague closure of $M$, and let $\Uf$ be an ultrafilter on $M$ which converges vaguely to $\mi_0$. Observe that, with the notation of Corollary~\ref{cor:4}, $\Gamma(M)$ is then a bounded (in $\Mc^1(\bD)$) set of positive pluriharmonic measures on $\bD$. Therefore, $\Gamma(\Uf)$ must converge vaguely to some positive pluriharmonic measure $\mi_1$ on $\bD$, so that $\mi_1=\Gamma(\mi_0)$ \emph{on $\gamma(\bDs)$}. Then, (the proof of) Theorem~\ref{teo:2} shows that $\mi_0$ is pluriharmonic.
\end{proof}

\begin{oss}\label{oss:15}
	Keep the hypotheses and the notation of Theorem~\ref{teo:2}. As we shall see in Theorem~\ref{teo:3}, $\mi$ is the weak limit (in $\Mc_\Poisson(\Ns)$) of $f_h\cdot \beta_{\bDs}$, for $h\to 0$, $h\in \Omega$ (cf.~Definition~\ref{def:2}). In addition, if $\Ds=\Ds_0\times \C^k$, where $\Ds_0$ is a product of irreducible symmetric Siegel domains which are \emph{not} of tube type, and if $(\vect e_1,\dots, \vect e_k)$ is the canonical basis of $\Set{0}\times\C^k$, then the proof of Theorem~\ref{teo:2} shows that\footnote{Actually, in the proof the limit was considered in the bounded realization $D$ of $\Ds$. However, one may transfer the problem back to $\Ds$ by means of the Cayley transform and use the fact that the convergence is not only radial, but actually non-tangential.}
	\[
	\langle \lambda\vert \vect e_j\rangle=\lim_{y\to +\infty} \frac{f(\zeta,z+y \vect e_j )}{y}
	\]
	for every $j=1,\dots, k$ and for every $(\zeta,z)\in \Ds$. In addition, $\lambda=\sum_j \langle \lambda\vert \vect e_j\rangle \vect e_j$.
	
	Notice, though, that in general
	\[
	\langle h\vert \lambda\rangle\meg \lim_{y\to +\infty}  \frac{f(\zeta,z+y h )}{y}
	\]
	for every $h\in \Omega$ and for every $(\zeta,z)\in \Ds$, and that the strict inequality may occur (cf.~Example~\ref{ex:6})
\end{oss}

\begin{oss}\label{ex:3}
	When $D$ is an irreducible bounded symmetric domain of type $(I_{p,q})$, $p\meg q$, (cf.~Example~\ref{ex:2}), we may compute explicitly the codimension (in $\bD$) of the set $N$ described in the proof of~\textsc{step II} of Theorem~\ref{teo:2}. In fact, by the discussion of Example~\ref{ex:2}, $\bD$ is the set of isometries of $\C^p$ into $\C^q$, whereas $N$ is the set of isometries from $\C^p$ into $\C^q$ with an eigenvalue equal to $1$ (choosing $e=(I,0)$).  
	Now,  observe that an element of $\bD$ is uniquely determined once we choose an orthonormal sequence $(v_1,\dots, v_p)$ of elements of $\C^q$ (the images of the elements of a fixed orthonormal basis). Thus, $v_j$ may be chosen (inductively) as any element of the unit sphere of $\big(\sum_{j'<j} \C v_{j'}\big)^\perp$, which has dimension $2 (q-j+1)-1$. Thus, $\bD$ has dimension $\sum_{j=1}^p [2 (q-j)+1]$.
	On the other hand, an element of $N$ is uniquely determined once we choose a (complex) line $L$ in $\C^p$ (an eigenvector  with eigenvalue $1$, up to multiplicative constants), and then an isometry from $\C^p\cap L^\perp$ into $\C^q\cap L^\perp$. Since the (real) dimension of the projective space on $\C^p$ is $2 p-2$ and the dimension of the space of isometries from $\C^{p-1}$ into $\C^{q-1}$ is $\sum_{j=1}^{p-1} [2(q-1-j)+1]=\sum_{j=2}^p [2(q-j)+1]$ by the previous remarks, this allows to conclude that the codimension of $N$ in $\bD$ is $ 2 (q-1)+1- 2 p+2= 2 (q-p)+1$, which is (in general, much) larger than the codimension $2$ proved in the general estimates of~\textsc{Step II} of the proof of Theorem~\ref{teo:2} (when $q>p$, that is, when $D$ is \emph{not} of tube type). The situation is not so simple, though, in the remaining irreducible symmetric domains which are not of tube type.
\end{oss}

We shall now consider what can be said of positive pluriharmonic functions on $\Ds$ when $\Ds$ is an arbitrary symmetric Siegel domain. We start with a disintegration theorem; we shall investigate later the meaning of the `boundary measure' $\mi$.

\begin{deff}\label{def:2}
	Given a function $f\colon \Ds \to \C$, we shall define $f_h\colon \bDs\ni (\zeta,z)\mapsto f(\zeta,z+i h)\in \C$ for every $h\in\Omega$.
\end{deff}

\begin{prop}\label{prop:25}
	Let $f\colon \Ds\to \R$ be a positive pluriharmonic function and take $h\in \Omega$.
	Then, the following hold:
	\begin{enumerate}
		\item[\textnormal{(1)}] $f_{y h}\cdot \beta_{\bDs}$ has a vague limit $\mi$ for $y\to 0^+$;
		
		\item[\textnormal{(2)}]  if $\pi_h$ denotes the (real) orthogonal projector of $E\times F_\C\to (\R h)^\perp$,\footnote{In other words, $\pi_h(\zeta,z)=(\zeta,z- \langle \Re z\vert h \rangle h/\abs{h}^2 )$ for every $(\zeta,z)\in E\times F_\C$.} then $\beta_h\coloneqq \chi_{\pi_h(\bDs)} \cdot \Hc^{2n+m-1}$ is a weak pseudo-image measure of $\mi$ under $\pi_h$, and is actually a pseudo-image measure if $\mi\neq 0$; 
		
		\item[\textnormal{(3)}] there are a vaguely continuous disintegration $(\mi_{(\zeta,z)})$ of $\mi$ (relative to $\beta_h$) and $a\Meg 0$ such that
		\begin{equation}\label{eq:5}
			f(\zeta,z+w h)=\frac{\abs{h}^{2}}{\pi}\int_{\bDs} \frac{\Im w}{\abs{(\zeta,z+w h)-x}^2}\,\dd \mi_{(\zeta,z)}(x)+a  \Im w
		\end{equation}
		for every $(\zeta,z)\in \pi_h(\bDs)$ and for every $w\in \C_+$;
		
		\item[\textnormal{(4)}] if $\mi'_{(\zeta,z)}$ denotes the unique element of $\Mc_\Poisson(\R,\C_+)$ of which the  $\mi_{(\zeta,z)}$ defined in~\textnormal{(3)} is the image under the mapping $\R\ni x \mapsto (\zeta,z + x h)\in \bDs$, then the mapping $(\zeta,z)\mapsto \mi'_{(\zeta,z)}\in \Mc_\Poisson(\R,\C_+)$ is (tightly) continuous.
	\end{enumerate}
\end{prop}

As we shall see later (cf.~Theorem~\ref{teo:4}), $\mi$ is actually the  (tight) limit of  $(f_h\cdot \beta_{\bDs})$ in $\Mc_\Poisson(\bDs)$, for $h\to 0$, $h\in \Omega$. In addition, as Example~\ref{ex:1} shows, $\mi$ need \emph{not} be pluriharmonic. As we shall see in Theorem~\ref{teo:4} this is the case if and only if  $f_h\cdot \beta_{\bDs}$ is plurihamonic for \emph{no} $h\in \Omega$.

We first need some lemmas. 

\begin{lem}\label{lem:12}
	Let $f\colon \Ds\to \C$ be a holomorphic function with positive real part, and take $h\in \Omega$. Then, there is $a\Meg 0$ such that
	\[
	\lim_{y\to +\infty} \frac{f(\zeta,z+ i y h)}{y}=a
	\] 
	for every $(\zeta,z)\in \Cl(\Ds)$.
\end{lem}

Observe that, since for every $(\zeta,z)\in E\times F_\C$ and for every $h\in\Omega$ there is $C>0$ such that $(\zeta,z+ i y h)\in \Ds$ (that is, $\Im z +y h - \Phi(\zeta)\in  \Omega$) for $y\Meg C$, we may have equivalently chosen $(\zeta,z)\in E\times F_\C$ in the above statement. Nonetheless, considering only $(\zeta,z)\in \Cl(\Ds)$ seems more natural.

\begin{proof}
	\textsc{Step I.} Assume first that $h=e$. Throughout this step, we shall write $\zeta+z$ instead of $(\zeta,z)$, in order to simplify the notation. 
	In addition, we shall endow $Z$ with the Jordan algebra structure with product $(x,x')\mapsto \{x,e,x'\}$.  Notice that, up to replace $f$ with $\zeta+z\mapsto f(\zeta+z)- i \langle z\vert e \rangle$, we may assume that $\Re f(\zeta+z)\Meg \Im \langle z \vert e \rangle$ for every $\zeta+z\in \Ds$. The statement then amounts to saying that there is $a\Meg 0$ such that
	\[
	\lim_{y\to +\infty} \frac{f(\zeta+z+ i y e)}{y}=a+\abs{e}^2
	\] 
	for every $\zeta+z\in \Cl(\Ds)$. To this aim, observe first that for every $\zeta+z\in \Cl(\Ds)$   there is $a_{\zeta+z}\Meg 0$ such that
	\[
	\lim_{y\to +\infty} \frac{f(\zeta+z+ i y e)}{y}=a_{\zeta+z}+\abs{e}^2
	\]
	(cf.~for example,~\cite[Proposition 2.3]{LugerNedic1}, applied to $\C_+\ni w \mapsto i f( \zeta+ z+ i w e)$). 
	Furthermore, we may assume that $\gamma(\zeta+z)= (z+i e)^{-1} (4 i\zeta + z- i e)= e-2 (e- iz)^{-1} (-2\zeta + e)$ for every $\zeta+z\in \Ds$, so that $\gamma^{-1}(\zeta+z)=  (e-z)^{-1} (2 \zeta+i(e+z))$ for every $\zeta+z\in D$ (cf.~Proposition~\ref{prop:22}). Define 
	\[
	g\coloneqq \frac{1-f\circ \gamma^{-1}}{1+f\circ \gamma^{-1}},
	\]
	so that $g\colon D\to \Db$ (since $f\neq 0$) and $g(\rho e)\to -1$ for $\rho \to 1^-$ (since $f( i y e)\to +\infty$ for $y\to + \infty$). In addition,
	\[
	\lim_{\rho \to 1^-} \frac{1- \abs{g(\rho e)}^2}{1-\norm{\rho e}^2}=\lim_{\rho \to 1^-} \frac{4 \Re f (\gamma^{-1}(\rho e))}{2 (1-\rho) \abs{1+f(\gamma^{-1}(\rho e))}^2}=\lim_{y\to +\infty} \frac{y \Re f(i y e)}{\abs{1+ f(i y e)}^2}=\frac{1}{a_0+\abs{e}^2}
	\]
	by the previous remarks. Therefore,~\cite[Lemma 3.3 and Theorem 4.6]{MackeyMellon} show that 
	\[
	\sup_{\zeta+z\in D_k(e)}\frac{\abs{ 1+g(\zeta+z) }^2}{\norm{B(\zeta+z,e)Q(e)^2}'}
	\]
	is finite for every $k>0$. Then, observe that, arguing as in the proof of Lemma~\ref{lem:8},~\cite[Lemma 3.3]{MackeyMellon} and Remark~\ref{oss:12} show that
	\[
	\norm{B(\zeta+z,e)Q(e)^2}'\Meg \norm{Q(e)^2 B(\zeta+z,e)Q(e)^2}'=\norm{B(z,e)Q(e)^2}'\Meg \norm{z-e}^2.
	\] 
	Consequently,
	\[
	\sup_{\zeta+z\in D_k( e)}\frac{\abs{ 1+g(\zeta+z) }}{\norm{z-e}}
	\]
	is finite for every $k>0$. 
	Now, let us prove that, if $K$ is a compact subset of $\Ds$ such that $\norm{\zeta}\meg \frac 1 3$ for every $\zeta+z\in K$, then there is $C_K>0$ such that $\gamma(\zeta+z+i y e)\in D_{6\sqrt 3}(e)$ for every $\zeta+z\in K$ and for every $y\Meg C_K$. We wish to apply Lemma~\ref{lem:11}. Then, consider $\widetilde \gamma(\zeta+z+ i y e)=\frac{1}{\abs{e}^2} \langle \gamma(\zeta+z+i y e)\vert e \rangle$, and observe that
	\[
	\widetilde \gamma(\zeta+z+ i y e)= 1 -\frac{2}{\abs{e}^2} \langle(e (y+1)- iz)^{-1}\vert  e\rangle=1-\frac{2}{y } + O\left(\frac{1}{y^2}\right)
	\]
	for $y\to +\infty$, uniformly as $\zeta+z$ runs through  $K$,
	since $e$ is orthogonal to $Z_{1/2}(e)$, and since $(e (y+1)- iz)^{-1}=(y+1)^{-1} \sum_{k\in \N} (i z/(y+1))^k $ if $y\Meg \sup_{\zeta+z\in K} \norm{z}$. In addition,
	\[
	\gamma(\zeta+z+ i y e)=e - 2 ((y+1) e- i z)^{-1}(-2 \zeta+e)=e-\frac{2}{y}(-\zeta+e)+  O\left(\frac{1}{y^2}\right)
	\]
	for $y\to +\infty$, uniformly as $\zeta+z$ runs through  $K$, since $e \zeta=\frac 1 2 \zeta$.
	Therefore,   
	\[
	\norm{\gamma(\zeta+z+ i y e)-\widetilde \gamma(\zeta+z+ i y e) e}=\frac{2}{y}\norm{ \zeta } + O\left(\frac{1}{y^2}\right)\meg \frac{1}{2}\left(\frac{2}{y} + O\left(\frac{1}{y^2}\right)\right) =\frac 1 2 (1-\abs{\widetilde \gamma(\zeta+z+ i y e)})
	\]
	for $y\to +\infty$, uniformly as $\zeta+z$ runs through $K$. In addition,
	\[
	\abs{1-\widetilde \gamma(\zeta+z+ i y e)}=\frac{2}{y}+O\left(\frac{1}{y^2}\right)\meg\left( \frac{4}{y}+O\left(\frac{1}{y^2}\right) \right)= (1-\abs{\widetilde \gamma(\zeta+z+ i y e)}^2)
	\]
	for $y\to +\infty$, uniformly as $\zeta+z$ runs through $K$. Therefore, Lemma~\ref{lem:11} shows that there is a constant $C_K>0$ such that $\gamma(\zeta+ z+ i y)\in D_{6 \sqrt 3}(e)$ for every $\zeta+z\in K$ and for every $y\Meg C_K$. Hence,
	\[
	\sup_{\zeta+z\in K, y \Meg C_K} \frac{\abs{1+ g(\gamma(\zeta+z+ i y e))}}{\norm{Q(e)^2 \gamma(\zeta+z+i y e)-e}}
	\]
	is finite.
	Furthermore,
	\[
	\norm{Q(e)^2 \gamma(\zeta+z+i y e)- e}=2\norm{(e(y+1) - i z)^{-1}  }=\frac{2}{y} + O\left(\frac{1}{y^2}\right)
	\] 
	for $y\to +\infty$,	uniformly as $\zeta+z$ runs through $K$.
	Consequently,  
	\[
	\sup_{\zeta +z\in K, y\Meg 1} \frac{y}{\abs{1+f(\zeta+z+ i y h)}}
	\] 
	is finite for every compact subset $K$ of $\Ds$ such that $\norm{\zeta}\meg \frac 1 3$ for every $\zeta+e\in K$. Hence, the family $(y/(1+ f(\,\cdot\,+  i y e)))_{y\Meg 1}$ is bounded in $\Hol(  U)$, where $U=\Set{\zeta+z\in \Ds\colon \norm{\zeta}<\frac 1 3}$, and converges pointwise to the positive function $\zeta+z\mapsto \frac{1}{a_{\zeta+z}+\abs{e}^2}$ for $y\to +\infty$.  Since the limit must   be holomorphic (on $U$), it is constant. We have thus proved that there is $a\Meg 0$ such that
	\[
	\lim_{y\to + \infty} \frac{f(\zeta+z+ i y e)}{y}=a+\abs{e}^2
	\]
	for every $\zeta+z\in U$.
	
	\textsc{Step II.} Assume that $h=e$ and define $U$ as in~\textsc{step I}, so that $(0, i e)\in U$.
	Observe that $E\times F_\C$ may be endowed with a product $(\zeta,z)(\zeta',z')\coloneqq (\zeta+\zeta', z+z'+2 i \Phi(\zeta',\zeta))$, with respect to which it is a $2$-step nilpotent Lie group. In addition, $\bDs$ is a subgroup of $E\times F_\C$ and $\Cl(\Ds)\Ds\subseteq \Ds$.
	Then, for every $(\zeta,z)\in \Cl(\Ds)$,~\textsc{step I}, applied  to $f( (\zeta,z)\,\cdot\, )$, shows that there is $a'_{(\zeta,z)}\Meg 0$ such that
	\begin{equation}\label{eq:7}
	\lim_{y\to +\infty} \frac{f( (\zeta,z)(\zeta',z'+ i y e) )}{y}=\lim_{y\to + \infty} \frac{f( (\zeta,z)(\zeta',z'+i e+ i y e) )}{y}=a'_{(\zeta,z)},
	\end{equation}
	for every $(\zeta',z')\in U-(0, i e)$. 
	Observe that (taking $(\zeta',z')=(0,0)$ in~\eqref{eq:7}) $a'_{(\zeta,z)}=\lim_{y\to +\infty} \frac{f( \zeta,z + i y e) }{y}$, so that~\eqref{eq:7} shows that the function $(\zeta,z)\mapsto a'_{(\zeta,z)}$ is locally constant on $\Cl(\Ds)$. Since this latter set is connected,  this shows that there is $a\Meg 0$ such that $a=a'_{(\zeta,z)}$ for every $(\zeta,z)\in \Cl(\Ds)$, that is
	\[
	\lim_{y\to +\infty} \frac{f( \zeta,z + i y e )}{y}=a,
	\]
	whence the conclusion in this case.
	
	\textsc{Step III.} Take $t\in GL_\C(F_\C)$ and $g\in GL_\C(E)$ such that  $t\Omega=\Omega$, $t e=h$, and $(g\times t)\Ds=\Ds$ (cf., e.g.,~\cite[Lemma 3.1 of Chapter V]{Satake}). Then, applying~\textsc{step II} to $f\circ (g\times t)$, we see that there is $a\Meg 0$ such that 
	\[
	\lim_{y\to + \infty} \frac{f(\zeta,z+ i y h)}{y}=\lim_{y\to +\infty} \frac{(f\circ (g\times t))(g^{-1}\zeta,g^{-1}z+ i y e)  }{y}=a
	\]
	for every $(\zeta,z)\in \Cl(\Ds)$, whence the conclusion.
\end{proof}

We propose an elementary proof of the following two lemmas for lack of a reference. 

\begin{lem}\label{lem:13}
	The mappings $C_c(\R)\ni f \mapsto (\Ps f)_y\in C_{\Poisson}(\R)$, $y\in (0,1]$, are equicontinuous, where $C_\Poisson(\R)$ is endowed with the norm $f \mapsto \sup_{x\in \R} (1+\abs{x}^2) \abs{f(x)}$.
\end{lem}

\begin{proof}
	Observe first that, since $C_c(\R)$ is barrelled, it will suffice to show that $((\Ps f)_y)_{y\in (0,1]}$ is bounded in $C_\Poisson(\R)$ for every $f\in C_c(\R)$. Then, take $f\in C_c(\R)$ and take $\rho>0$ so that $\Supp{f}\subseteq [-\rho,\rho]$. Observe that there is a constant $C>0$ such that
	\[
	\arctan\Big(\frac{x+\rho}{y}\Big)-\arctan\Big(\frac{x-\rho}{y}\Big)\meg \frac{C}{1+x^2}
	\]
	for every $x\in \R$ and for every $y\in (0,1]$.  Indeed, on the one hand, it is clear that, if $\abs{x}\meg 2\rho$, then $(1+x^2)\big(\arctan\big(\frac{x+\rho}{y}\big)-\arctan\big(\frac{x-\rho}{y}\big)\big) \meg \pi(1+4 \rho^2)$ for every $y\in (0,1]$. On the other hand, if $\abs{x}\Meg 2\rho$, then
	\[
	\begin{split}
		\arctan\Big(\frac{x+\rho}{y}\Big)-\arctan\Big(\frac{x-\rho}{y}\Big) &=\arctan\Big(\frac{y}{x-\rho}\Big)-\arctan\Big(\frac{y}{x+\rho}\Big)\\
		&=\frac{2y \rho }{x^2-\rho^2}+\omega(\xi_{y,x,\rho})\frac{y^3}{(x-\rho)^3}-\omega(\xi'_{y,x,\rho}) \frac{y^3}{(x+\rho)^3} 
	\end{split}
	\]
	with $\omega=\frac{1}{6}\arctan'''$ and $\xi_{y,x,\rho},\xi'_{y,x,\rho}\in [-1/\rho,1/\rho]$, so that our claim follows.
	Consequently,
	\[
	\begin{split}
		\abs{(\Ps f)_y(x)}\meg\int_{\R} \abs*{f(x'+i y)} \frac{y}{(x-x')^2+ y^2}\,\dd \Hc^1(x')&\meg \norm{f}_{L^\infty(\R)} \int_{x-\rho  }^{x+\rho} \frac{y}{x'^2+y^2}\,\dd \Hc^1(x')\\
		&\meg \norm{f}_{L^\infty(\R)}  ( \arctan((x+\rho)/y)-\arctan((x-\rho)/y) )\\
		&\meg \frac{C}{1+x^2} \norm{f}_{L^\infty(\R)} 
	\end{split}
	\]
	for every $x\in \R$ and for every $y\in (0,1]$, whence the result.
\end{proof}

\begin{lem}\label{lem:14}
	Let $f\colon \C_+\to \R$ be a positive pluriharmonic function. Take a positive $\mi\in \Mc_\Poisson(\R)$ and $a\Meg 0$ such that
	\[
	f(z)= (\Ps \mi)(z)+ a \Im z
	\]
	for every $z\in \C_+$. Then, $f_y\cdot \Hc^1$ converges vaguely to $\mi$ for $y\to 0^+$.
\end{lem}

\begin{proof}
	Take $g\in C_c(\R)$, and observe that
	\[
	\int_{\R} f_y(x) g(x)\,\dd \Hc^1(x)= a y\int_{\R} g(x)\,\dd \Hc^1(x) + \int_\R (\Ps g)_y\,\dd \mi,
	\]
	so that
	\[
	\lim_{y\to 0^+}\int_{\R} f_y(x) g(x)\,\dd \Hc^1(x)=\lim_{y\to 0^+}\int_\R (\Ps g)_y\,\dd \mi.
	\]
	In addition, observe that $(\Ps g)_y\to g$ locally uniformly for $y\to 0^+$, and that  the $(\Ps g)_y$, $y\in (0,1)$, are uniformly bounded in $C_\Poisson(\R)$, thanks to Lemma~\ref{lem:13}. Then, the dominated convergence theorem implies that
	\[
	\lim_{y\to 0^+}\int_{\R} f_y(x) g(x)\,\dd \Hc^1(x)=\int_\R g(x)\,\dd \mi(x),
	\]
	whence the result by the arbitrariness of $g\in C_c(\R)$.
\end{proof}

\begin{proof}[Proof of Proposition~\ref{prop:25}.]
	Take $(\zeta,z)\in \pi_h(\bDs)$, and observe that the mapping $\C_+\ni w\mapsto f(\zeta,z+ w h)\in \R$ is positive and pluriharmonic, so that  Theorem~\ref{teo:2} shows that there are $a_{(\zeta,z)}\Meg 0$ and a positive measure $\mi'_{(\zeta,z)}\in \Mc_\Poisson(\R,\C_+)$ such that
	\[
	f(\zeta,z+w h)=a_{(\zeta,z)} \Im w+\frac 1 \pi\int_{\R} \frac{\Im w}{\abs{w-x}^2}\,\dd \mi'_{(\zeta,z)}(x).
	\]
	Observe that, if $g$ is a holomorphic function with $\Re g=f$, then
	\[
	a_{(\zeta,z)}=\lim_{y\to +\infty} \frac{f(\zeta,z+ i y h)}{y}=\lim_{y\to +\infty} \frac{g(\zeta,z+i y h)}{y},
	\]
	by~\cite[Proposition 2.3]{LugerNedic1}, applied to the holomorphic function $\C_+\ni w \mapsto i g(\zeta,z+w h)\in \C$. Therefore, Lemma~\ref{lem:12} shows that there is $a\Meg 0$ such that $a_{(\zeta,z)}=a$ for every $(\zeta,z)\in \pi_h(\bDs)$.
	Then, define $\mi_{(\zeta,z)}$ as the image of $\mi'_{(\zeta,z)}$ under the mapping $\R \ni x\mapsto (\zeta,z+x h)\in \bDs$, so that~\eqref{eq:5} holds.
	Observe that $\mi_{(\zeta,z)}$ is the vague limit of $\chi_{(\zeta,z)+ \R h} f_{y h}\cdot \Hc^1 $ for $y\to 0^+$ by Lemma~\ref{lem:14}.  
	
	Observe, the mapping $(\zeta,z)\mapsto (\Ps_\R \mi'_{(\zeta,z)})(w)= f(\zeta,z+ w h)- a \Im w$ is continuous for every $w\in \C_+$, where $\Ps_\R$ denotes the Poisson integral operator on $\R$. 
	In particular, the mapping $(\zeta,z)\mapsto \int_\R \frac{1}{1+x^2}\,\dd \mi'_{(\zeta,z)}(x)$ is continuous. 
	Now, consider the (inverse) Cayley transform $\gamma_0\colon \C_+\ni w \mapsto \frac{w-i}{w+i}\in \Db$, and observe that
	\[
	[\Pc_\T \Gamma_0(\mi'_{(\zeta,z)})]\circ \gamma_0= \Ps_\R \mi'_{(\zeta,z)}
	\]
	for every $(\zeta,z)\in \pi_h(\bDs)$, where $\Pc_\T$ denotes the Poisson integral operator on $\T$ and $\Gamma_0$ is defined as in Corollary~\ref{cor:4}. Since the mapping $(\zeta,z)\mapsto [\Pc_\T \Gamma_0(\mi'_{(\zeta,z)})](w)$ is continuous for every $w\in \Db$, it is readily seen that the mapping $(\zeta,z)\mapsto \Gamma_0(\mi'_{(\zeta,z)})$ is vaguely continuous, so that the mapping $(\zeta,z)\mapsto \mi'_{(\zeta,z)}$ is tightly continuous by Corollary~\ref{cor:4}.

	Next, take   $g\in C_c(\bDs)$, and observe that Lemma~\ref{lem:13} shows that there is a constant $C>0$ such that
	\[
	\begin{split}
		\int_{\R} \abs*{g(\zeta,z+x' h)} \frac{y}{(x-x')^2+ y^2}\,\dd \Hc^1(x')\meg \frac{C }{1+x^2} \quad\text{and}\quad \int_\R \abs*{g(\zeta,z+x' h)} \,\dd \Hc^1(x')\meg C
	\end{split}
	\]
	for every $(\zeta,z)\in \bDs$, for every $x\in \R$ and for every $y\in (0,1]$.
	Hence, 
	\[
	\begin{split}
		\abs*{\int_{\R} g(\zeta,z+x h) f_{y h}(\zeta,z +x h)\,\dd \Hc^1(x)}&\meg \int_\R a \abs{g(\zeta,z+x h)}\,\dd \Hc^1(x) \\
		&\qquad+ \int_\R \int_\R \frac{y}{(x-x')^2+y^2}\abs{g(\zeta,z+x h)} \,\dd \Hc^1(x) \,\dd \mi'_{(\zeta,z)}(x')\\
		&\meg C \bigg(a+ \frac{1}{\pi} \int_\R \frac{1}{1+x^2}\,\dd \mi'_{(\zeta,z)}(x) \bigg)\\
		& = C f(\zeta,z+ i h)
	\end{split}
	\]
	for every $(\zeta,z)\in \pi_h(\bDs)$ and for every $y\in (0,1)$. In addition,  
	\[
	\lim_{y\to 0^+} \int_{\R} g(\zeta,z+x h) f_{y h}(\zeta,z +x h)\,\dd \Hc^1(x)=\int_{\R} g(\zeta,z+x h) \,\dd \mi'_{(\zeta,z)}(x)=\int_{\bDs} g\,\dd \mi_{(\zeta,z)}
	\]
	for every $(\zeta,z)\in \pi_h(\bDs)$. Consequently, by dominated convergence,
	\[
	\begin{split}
		&\lim_{y\to 0^+} \int_{\bDs} f_{y h}(\zeta,z) g(\zeta,z)\,\dd \beta_{\bDs}(\zeta,z)\\
		&\qquad= \lim_{y\to 0^+} \int_{\pi_h(\bDs) } \int_\R f_{y h}(\zeta,z + x h) g(\zeta,z+ x h) \,\dd \Hc^1(x)\,\dd \beta_h(\zeta,z)\\
		&\qquad= \int_{\pi_h(\bDs) } \int_{\bDs} g \,\dd \mi_{(\zeta,z)} \,\dd \beta_h(\zeta,z).
	\end{split}
	\]
	By Proposition~\ref{prop:7} and the arbitrariness of $g$, this proves that $f_{y h}\cdot \beta_{\bDs}$ converges vaguely to the positive Radon measure
	\[
	\mi\coloneqq \int_{\pi_h(\bDs)} \mi_{(\zeta,z)}\,\dd \beta_h(\zeta,z)
	\]
	for $y\to 0^+$. In particular,  $\beta_h $ is a weak pseudo-image measure of $\mi$ under $\pi_h$, and   $(\mi_{(\zeta,z)})$ is a disintegration of $\mi$ relative to $\beta_h$. 
	
	Finally, assume that  $\beta_h$ is not a pseudo-image measure of $\mi$, and let us prove that $\mi=0$. Observe that there is a Borel subset $A$ of $\pi_h(\bDs)$ such that $\beta_h(A)>0$ and $\mi_{(\zeta,z)}=0$ for every $(\zeta,z)\in A$ (cf.~Proposition~\ref{prop:6}). 
	This implies that $(\partial_h^2 f)(\zeta,z  + w h)=0$ and that (say) $ 2 f(\zeta,z  + i h)= f(\zeta,z+ 2 i h)$ for every $(\zeta,z)\in A$ and for every $w\in \C_+$. 
	By (real) analyticity, this implies that $(\partial_h^2 f)(\zeta,z+ w h)=0$ and $ 2 f(\zeta,z+ i h)= f(\zeta,z+ 2 i h)$ for every $(\zeta,z)\in \pi_h(\bDs)$ and for every $w\in \C_+$. 
	Then,  $f(\zeta,z+w h)=a \Im w$, that is, $\mi_{(\zeta,z)}=0$, for every $(\zeta,z)\in \pi_h(\bDs)$.
\end{proof}

\begin{prop}\label{prop:24}
	Let $(f^{(j)})$ be a sequence of positive pluriharmonic functions on $\Ds$, and let $\mi_j$, for every $j\in\N$, be the vague limit of $f^{(j)}_{y h}\cdot \beta_{\bDs}$ for $y\to 0^+$, for some  $h\in \Omega$. If $(f_j)$ converges pointwise to a pluriharmonic function $f$, then $\mi_j$ converges vaguely to the vague limit $\mi$ of $f_{y h}\cdot \beta_{\bDs}$ for $y \to 0^+$.
\end{prop}

As a consequence of the following Theorem~\ref{teo:3}, $(\mi_j)$ actually converges to $\mi$ in the weak topology of $\Mc_\Poisson(\bDs)$. Cf.~Corollary~\ref{cor:9}

\begin{proof}
	\textsc{Step I.} Assume first that $\Ds=\C_+$ and let us prove that $\mi_j$ converges vaguely to $\mi$. We may assume $\gamma(w)= \frac{w-i}{w+i}$ for every $w\in \C_+$. 
	Define $\nu_j $ and $\nu$ as the unique positive measures on $\T$ such that $\Pc \nu_j=f^{(j)}\circ \gamma^{-1}$ and $\Pc \nu =f\circ \gamma^{-1}$. Then, $\nu_j(\T)=f^{(j)}(i)$ is uniformly bounded. 
	Since, in addition, $\Pc \nu_j$ converges pointwise to $\Pc \nu$,  it is clear that $\nu_j$ converges vaguely to $\nu$. Since $\Gamma(\mi_j)=\nu_j-\nu_j(\Set{1})\delta_1$ and $\Gamma(\mi)=\nu-\nu(\Set{1})\delta_1 $, with the notation of Corollary~\ref{cor:4}, this  proves that $\mi_j$ converges vaguely to $\mi$.
	
	\textsc{Step II.} Define $(\mi_{(\zeta,z),j})$   as the unique vaguely continuous disintegration  of $\mi_j$ relative to $\beta_h$, with the notation of Proposition~\ref{prop:25}, so that
	\[
	f^{(j)}(\zeta,z+ w h)=\frac{\abs{h}^2}{\pi}\int_{\bDs} \frac{\Im w}{\abs{(\zeta,z+ w h)-x}^2}\,\dd \mi_{(\zeta,z),j}(x)+a_{j} \Im w,
	\]
	for some $a_{j}\Meg0$, for every $w\in \C_+$, and for every $(\zeta,z)\in \pi_h(\bDs)$. 
	Define $(\mi_{(\zeta,z)})$ and $a$ similarly.
	Then,~\textsc{step I} shows that $\mi_{(\zeta,z),j}$ converges vaguely to $\mi_{(\zeta,z)}$ for every $(\zeta,z)\in \pi_h(\bDs)$. 
	
	Now, take $g\in C_c(\bDs)$. Let   $K_1$ and $K_2$ be the images of $\Supp{g}$ under the mappings $ (\zeta,z)\mapsto \langle \Re z, h\rangle/\abs{h}^2$ and $\pi_h$. Then, 
	\[
	\begin{split}
	\abs*{\int_{\bDs} g\,\dd \mi_{(\zeta,z),j}}& \meg\chi_{K_2}(\zeta,z) \norm{g}_{L^\infty(\bDs)}  \mi_{(\zeta,z),j} (\Supp{g})\\
	&\meg \chi_{K_2}(\zeta,z) \norm{g}_{L^\infty(\bDs)} f^{(j)}(\zeta,z + i h) \max_{x\in K_1}  \pi (1+x^2)
	\end{split}
	\]
	for every $(\zeta,z)\in \pi_h(\bDs)$. Observe that $\sup_{j\in \N} \max_{(\zeta,z)\in K_2} f^{(j)}(\zeta,z+ i h) \max_{x \in K_1}  \pi(1+x^2)$ is finite. Then,  by dominated convergence,
	\[
	\begin{split}
		\int_{\bDs} g\,\dd \mi& =\int_{\pi_h(\bDs)} \int_{\bDs} g \,\dd \mi_{(\zeta,z)}\,\dd \beta_h(\zeta,z) \\
		&=\lim_{j\to \infty} \int_{\pi_h(\bDs)} \int_{\bDs} g \,\dd \mi_{(\zeta,z),j}\,\dd \beta_h(\zeta,z)\\
		&= \lim_{j\to \infty}\int_{\bDs} g\,\dd \mi_j.
	\end{split}
	\]
	By the arbitrariness of $g$, this proves that $(\mi_j)$ converges vaguely to $\mi$. 
\end{proof}

\begin{teo}\label{teo:3} 
	Keep the hypotheses and the notation of Proposition~\ref{prop:25} and Corollary~\ref{cor:4}, and let $\nu$ be the unique positive plurihamonic measure on $\bD$ such that $\Pc \nu= f\circ \gamma^{-1}$. 
	Then, the following hold:
	\begin{itemize}
		\item[\textnormal{(1)}] 	$\Gamma(\mi)=\chi_{ \gamma(\bDs)} \cdot \nu$;
		
		\item[\textnormal{(2)}] $\mi$ is the largest (positive) measure in $\Mc_\Poisson(\bDs)$ such that $\Ps \mi\meg f$ on $\Ds$;
		
		\item[\textnormal{(3)}]  $(f_{ h'}\cdot \beta_{\bDs})$ converges (to $\mi$) \emph{weakly}   in $\Mc_\Poisson(\bDs)$ for $h'\to 0$, $h'\in \Omega$.
	\end{itemize}
\end{teo}

As we shall see later (cf.~Theorem~\ref{teo:4}), weak convergence in the last assertion may be improved to tight convergence.

\begin{proof}
	(1) Set $f^{(j)}\colon \Ds\ni (\zeta,z)\mapsto f(\gamma^{-1}( (1-2^{-j}) \gamma(\zeta,z) ) )\in \R$ for every $j\in\N$, so that $(f^{(j)})$ is a sequence of positive pluriharmonic functions on $\Ds$ and converges locally uniformly to $f$. 
	In addition, if we set   $\nu_j\coloneqq f(\gamma^{-1}((1-2^{-j})\,\cdot\,))\cdot \beta_{\bD}$ for every $j\in\N$, then $\Pc \nu_j= f^{(j)}\circ \gamma^{-1}$ and $\nu_j(N)=0$ for every $j\in\N$, where $N\coloneqq \bD\setminus \gamma(\bDs)$. Furthermore, $(\nu_j)$ converges vaguely to $\nu$. 
	Next, observe that $f^{(j)}$ extends by continuity to a neighbourhood of $\bDs$, so that  $\mi_j\coloneqq f^{(j)}_0\cdot \beta_{\bDs}$ is the vague limit of $f^{(j)}_{y h}\cdot \beta_{\bDs}$ for $y \to 0^+$. 
	Since $\mi$ is the vague limit of $(\mi_j)$ by Proposition~\ref{prop:24}, it is clear that $\Gamma(\mi_j)$ converges vaguely to $\Gamma(\mi)$ \emph{on $\bD\setminus N$}. Since $\Gamma(\mi_j)=\nu_j$ converges vaguely to $\nu$, and since $\Gamma(\mi)(N)=0$, this is sufficient to conclude $\Gamma(\mi)=\chi_{\bD \setminus N} \cdot \nu=\chi_{\gamma(\bDs)}\cdot \nu$.
	
	(2) Notice that (1) implies that $\Ps \mi= \Pc(\chi_{\gamma(\bDs)}\cdot \nu) \circ \gamma\meg (\Pc \nu)\circ \gamma=f $. Now, let $\mi'$ be a (positive) measure in $\Mc_\Poisson(\bDs)$ such that $\Ps \mi'\meg f$. Then, $\nu'\coloneqq \Gamma(\mi')$ is a (positive) Radon measure on $\bD$ such that $\Pc \nu'\meg f\circ \gamma^{-1}= \Pc \nu$ and such that $\nu'(N)=0$. Since $\Pc (\nu-\nu')\Meg 0$, $\nu-\nu'$ must be a positive measure, so that $\nu'\meg \chi_{\bD\setminus N}\cdot \nu$, that is, $\mi'\meg \mi$.
	
	(3) By (2), applied to $f(\,\cdot\,+ (0 , i h'))$,
	\[
	\Ps[(f_{h'}\cdot \beta_{\bDs})](0, i e)\meg f( 0, i e+i h' ).
	\]
	Therefore,  the family $(f_{h'}\cdot \beta_{\bDs})_{h' \in \Omega, \abs{h'}\meg 1}$ is bounded in $\Mc_\Poisson(\bDs)$, so that it must converge weakly to $\mi$ in $\Mc_\Poisson(\bDs)$ for $h'\to 0$.
\end{proof}

\begin{cor}
	Assume that $n+m>1$, and let $\mi$ be a \emph{bounded} positive pluriharmonic measure on $\bDs$.  Then $\mi=0$.
\end{cor}

Cf.~\cite[Proposition 5.1]{LugerNedic2} for the case in which $\Ds$ is a product of upper half-planes.

\begin{proof}
	Observe that there are $c>0$ and a holomorphic function $\Delta\colon  \Omega+ i F\to \C$ such that  $\Delta( e)=1$,  $\Delta^2$ is a homogeneous polynomial of degree $2 n+2 m$, and 
	\[
	\Ps((\zeta,z+ i y e),(\zeta',z'))= c \frac{\Delta( y e )}{ \abs*{\Delta\big(\frac{z+ i y e-\overline{z'}}{2 i}- \Phi(\zeta,\zeta')  \big) }^2 }
	\]
	for every $(\zeta,z),(\zeta',z')\in \bDs$, and for every $y>0$ (cf.~Example~\ref{ex:5}). Therefore,
	\[
	\lim_{y \to +\infty} y^{n+m}\Ps((\zeta,z+ i y e),(\zeta',z'))=4^{n+m} c
	\]
	for every $(\zeta,z),(\zeta',z')\in \bDs$. In addition, 
	\[
	\sup_{\substack{y>0\\(\zeta,z),(\zeta',z')\in \bDs}}y^{n+m}\Ps((\zeta,z+ i y e),(\zeta',z'))=\sup_{\substack{y>0\\(\zeta,z),(\zeta',z')\in \bDs}}\frac{4^{n+m}c}{\abs*{\Delta\big( e + \frac{\Phi(\zeta-\zeta')}{y} + i \frac{\Re z'-\Re z - 2 \Im \Phi(\zeta,\zeta') }{y}  \big)  }^2}
	\]
	is finite (cf.~Example~\ref{ex:5} and~\cite[Corollary 2.36 and Lemma 2.37]{CalziPeloso}). Hence, by dominated convergence,
	\[
	(\Ps \mi)(\zeta,z+ i y e)\sim c\Big(\frac 4 y\Big)^{ n+m} \norm{\mi}_{\Mc^1(\bDs)}
	\]
	as $y\to +\infty$, for every $(\zeta,z)\in \bDs$.  In addition, if we set $f\coloneqq \Ps \mi$ and choose $(\mi_{(\zeta,z)})_{(\zeta,z)\in \pi_h(\bDs)}$ and $a$ as in Proposition~\ref{prop:25}, then
	\[
	0= \lim_{y \to +\infty} y f(\zeta,z+ i y e)= \lim_{y\to + \infty}\Big( \frac 1 \pi\int_{\bDs} \frac{y^2 \abs{e}^2}{ \abs{ (\zeta,z)-x}^2+ y ^2 \abs{e}^2}\,\dd \mi_{(\zeta,z)}(x) +a y^2\Big).
	\]
	Since
	\[
	 \lim_{y \to + \infty}  \frac 1 \pi\int_{\bDs} \frac{y^2\abs{e}^2}{  \abs{ (\zeta,z)-x}^2+ y ^2 \abs{e}^2}\,\dd \mi_{(\zeta,z)}(x)= \frac{1}{\pi} \norm{\mi_{(\zeta,z)}}_{\Mc^1(\bDs)}
	\]
	by monotone convergence, this proves that $a =0$ and that $\mi_{(\zeta,z)}=0$ for every $(\zeta,z)\in \pi_h(\bDs)$. Since $(\mi_{(\zeta,z)})$ is a disintegration of $\mi$ (cf.~Theorem~\ref{teo:3}), this proves that $\mi=0$.
\end{proof}

\begin{cor} \label{cor:10}
	Let $f$ be a positive pluriharmonic function on $\Ds$, and assume that the largest positive $\mi\in \Mc_\Poisson(\bDs)$ such that $\Ps \mi\meg f$ is pluriharmonic. Then, there is a unique   $\lambda\in \overline{\Omega}\cap \Phi(E)^\perp$  such that
	\[
	f(\zeta,z)= \langle \Im z\vert \lambda\rangle+ (\Ps\mi)(\zeta,z)
	\]
	for every $(\zeta,z)\in \Ds$. 
\end{cor}

\begin{proof}
	Observe that $g\coloneqq f-\Ps \mi$ is a positive pluriharmonic function on $\Ds$. In addition, $(g_h\cdot \beta_{\bDs})$ converges weakly to $0$ for $h\to 0$, thanks to Theorem~\ref{teo:3}. Therefore, Proposition~\ref{prop:25} shows that for every unit vector $h\in \Omega$ there is $a_h\Meg 0$ such that
	\[
	g(\zeta,z+ w h)= a_h \Im w
	\]
	for every $(\zeta,z)\in \pi_h(\bDs)$ and for every $w\in \C_+$. Consequently, the same holds for every $(\zeta,z)\in \bDs$ and for every $w\in \C_+$.  In particular, if $\widetilde g$ denotes a holomorphic function on $\Ds$ such that $\Re \widetilde g=g$, then
	\[
	(\partial^2_h \widetilde g)(\zeta,z+i h)=0
	\]
	for every $h\in\Omega$ and for every $(\zeta,z)\in \bDs$. Since  $\bDs + i h$ is a set of uniqueness for holomorphic  functions on $\Ds$, this implies that $\partial_h^2 \widetilde g=0$ on $\Ds$ for every $h\in\Omega$. Consequently, $\partial_{h}^2 \widetilde g=0$ on $\Ds$ for every $h\in F$, so that there are two holomorphic functions $b\colon E\to \C$ and $\widetilde\lambda\colon E\mapsto (F_\C)'$ such that
	\[
	\widetilde g(\zeta,z)=b(\zeta)-i\widetilde\lambda(\zeta) z
	\]
	for every $(\zeta,z)\in \Ds$.  Consequently,
	\[
	a_h y=g(\zeta,z+ i y  h) = \Re b(\zeta)+  \Im \widetilde\lambda(\zeta) z +y \Re \widetilde\lambda(\zeta) h
	\]
	for every $(\zeta,z)\in \bDs$, for every $y>0$, and for every $h\in \Omega$, so that $\Re b=0$, $\Im \widetilde\lambda(\zeta) z=0$ for every $(\zeta,z)\in \bDs$, and $\Re\widetilde\lambda(\zeta) h=a_h$ for every $\zeta\in E$ and for every $h\in \Omega$. Thus, there is a unique $\lambda\in \overline{\Omega}$ such that $\Re \widetilde\lambda(\zeta) x=\langle x\vert\lambda\rangle$ for every $\zeta\in E$ and for every $x\in F$. In addition, for every $x\in F$ and for every $\zeta\in E$,
	\[
	0=\Im \widetilde \lambda(\zeta) (x+i \Phi(\zeta))=\Im \widetilde \lambda(\zeta) x +\langle \Phi(\zeta)\vert \lambda \rangle.
	\]
	The case $x=0$ then shows that $\lambda\in \Phi(E)^\perp$. Consequently, $\Im \widetilde \lambda(\zeta) x=0$ for every $x\in F$, so that $\widetilde\lambda(\zeta)x=\Re \widetilde \lambda(\zeta) x=\langle x\vert \lambda \rangle$ for every $x\in F$. Thus,
	\[
	g(\zeta,z)=\Re \widetilde g(\zeta,z)= \langle \Im z \vert \lambda \rangle
	\]
	for every $(\zeta,z)\in \Ds$, whence the conclusion.
\end{proof}

\begin{cor}\label{cor:11}
	Let $f$ be a positive pluriharmonic function on $\Ds$, and let $\mi$ be the largest positive measure in  $ \Mc_\Poisson(\bDs)$ such that $\Ps \mi\meg f$. Then, the following are equivalent:
	\begin{enumerate}
		\item[\textnormal{(1)}] $\mi=0$;
		
		\item[\textnormal{(2)}]  $f$ is linear and $f(0, i\Phi(\zeta))=0$ for every $\zeta\in E$;
		
		\item[\textnormal{(3)}] there is a unique   $\lambda\in \overline{\Omega}\cap \Phi(E)^\perp$ such that 
		\[
		f(\zeta,z)=  \langle \Im z \vert\lambda\rangle
		\]
		for every $(\zeta,z)\in \Ds$.
	\end{enumerate}
\end{cor}

Notice that (2) cannot be simplified, since the function $(\zeta,z)\mapsto \langle \Im z \vert \lambda \rangle$ is positive and pluriharmonic for every $\lambda\in \overline \Omega$, even when $\lambda\not \in \Phi(E)^\perp$.

\begin{proof}
	(1) $\implies$ (3) This follows from Corollary~\ref{cor:10}.
	
	(3) $\implies$ (2)  This is obvious.
	
	(2) $\implies$ (1)   Since $f$ is linear, $(f_{y e})$ converges locally uniformly to $f_0$ for $y\to 0^+$, so that $\mi=f_0\cdot \beta_{\bDs}$ by Theorem~\ref{teo:3}.  
	Observe that the mapping
	\[
	\C_+\ni w \mapsto f(0, w h) 
	\]
	is positive for every $h\in \Omega$. Consequently, $f(0, h)=0$ for every $h\in\Omega$, hence for every $h\in F$.  It will then suffice to show that $f(\zeta, i \Phi(\zeta))=0$ for every $\zeta\in E$. To this aim, observe that  the function
	\[
	E\ni \zeta\mapsto f(\zeta,i \Phi(\zeta))= f(\zeta,0)+ f(0,i \Phi(\zeta))=f(\zeta,0)
	\]
	must be positive. Since $f(E\times \Set{0})$ is a vector subspace of $\R$,   it must reduce to $\Set{0}$. The proof is therefore complete.
\end{proof}

\begin{cor}\label{cor:9}
	Keep the hypotheses and the notation of Proposition~\ref{prop:24}. Then, $(\mi_j)$ converges weakly to $\mi$. 
\end{cor}

Notice that we cannot deduce that $(\mi_j)$ converges to $\mi$ in $\Mc_\Poisson(\bDs)$. In fact, already when $\Ds=\C_+$  the sequence $(\delta_j)_{j\in\N}$ converges weakly to $0$, but not in $\Mc_\Poisson(\R,\C_+)$, for $j\to \infty$, and $(\Ps \delta_j)(z)= \frac{\Im z}{\pi\abs{z-j}^2}$ converges to $0$ for every $z\in \C_+$.

\begin{proof}
	Observe that $\Ps\mi_j\meg f^{(j)}$ by Theorem~\ref{teo:3}, so that $(\mi_j)$ is bounded in $\Mc_\Poisson(\bDs)$. Hence, vague convergence is equivalent to weak convergence.
\end{proof}

\begin{teo}\label{teo:4}
	Let $f$ be a positive pluriharmonic function on $\Ds$, and let $\mi$ be the largest positive measure in $\Mc_\Poisson(\bDs)$ such that $\Ps \mi\meg f$. For every $h\in\Omega\cup \Set{0}$, let $\nu_h$ be the positive pluriharmonic measure on $\bD$ such that $(\Pc \nu_h)(z)= f(\gamma^{-1}(z)+ i h)$ for every $z\in D$. Then, the following hold:
	\begin{enumerate}
		\item[\textnormal{(1)}] $\chi_{N}\cdot \nu_h=\chi_N\cdot \nu_0$ for every $h\in\Omega$, where $N\coloneqq \bD\setminus \gamma(\bDs)$;
		
		\item[\textnormal{(2)}] for every $(\zeta,z)\in \Ds$and for every $h\in\Omega$,
		\[
		f(\zeta,z)-f(\zeta,z+i h)=(\Ps \mi)(\zeta,z)- (\Ps f_h)(\zeta,z)
		\]
		that is,
		\[
		f(\zeta,z)-(\Ps \mi)(\zeta,z)=f(\zeta,z+i h)-(\Ps f_h)(\zeta,z);
		\]	
		
		\item[\textnormal{(3)}] $f_h\cdot \beta_{\bDs}$ converges to $\mi$ \emph{in $\Mc_\Poisson(\bDs)$} for $h\to 0$, $h\in\Omega$. In addition, the mapping $\Omega\ni h\mapsto f_h\cdot \beta_{\bDs}\in \Mc_\Poisson(\bDs)$ is (tightly) continuous.
	\end{enumerate}
\end{teo}

In particular, $\mi$ is pluriharmonic if and only if $f_h\cdot \beta_{\bDs}$ is pluriharmonic for some/every $h\in\Omega$. As we shall see in Example~\ref{ex:1bis}, this may not be the case.

Notice that a similar assertion may be proved allowing all $h\in\overline \Omega+ i F$; we restricted our attention to $h\in \Omega\cup \Set{0}$ since it seemed more natural.

\begin{proof}
	(1) Let us extend the definition of $\nu_h$ to every $h\in \overline\Omega+ i F$.	
	Let $(\nu_{h,\xi})_{\xi\in \widehat \bD}$ be the vaguely continuous disintegration of $\nu_h$ relative to $\widehat \beta$, for every $h\in \overline \Omega+ i F$. 
	In addition, set $g\coloneqq (\Hc \nu_0)\circ \gamma$, so that $g$ is holomorphic, $\Re g=f$, and $ g(0,i e)=f(0, i e)$. In addition, 
	\begin{equation}\label{eq:3}
	(\Hc \nu_h)\circ \gamma= g(\,\cdot\,+ (0,i h))- i \Im g(0, i e+ i h)
	\end{equation}
	for every $h\in \overline\Omega+ i F$, since both sides of the equality have the same real part (namely, $f(\,\cdot\,+ (0, i h))$) and are real at $(0,i e)$.
	Then, for every $\zeta\in N$, Remark~\ref{oss:8} shows that
	\begin{equation}\label{eq:8}
	\nu_{h, \pi(\zeta)}(\Set{\zeta})=\lim_{\rho \to 1^-} \frac{1-\rho}{2} g(\gamma^{-1}(\rho \zeta)+(0, i h))
	\end{equation}
	and that
	\[
	\abs*{\frac{1-\rho}{2} g(\gamma^{-1}(\rho \zeta)+(0, i h))}\meg (\Hc \nu_h)(0)= f(0, i e+ i h)
	\]
	for every $h\in \overline \Omega+ i F$. 
	Since the functions $ \Omega+ i F\ni h\mapsto g(\gamma^{-1}(\rho \zeta)+(0, i h))$ are holomorphic and locally uniformly bounded, also their limit (for $\rho \to 1^-$)
	\[
	 \Omega+ i F\ni h\mapsto \nu_{h, \pi(\zeta)}(\Set{\zeta})\Meg 0
	\]
	must be holomorphic, hence constant. We have thus proved that $\nu_{h,\pi(\zeta)}(\Set{\zeta})=\nu_{e,\pi(\zeta)}(\Set{\zeta})$ for every $h\in \Omega+ i F$.  Observe that $N\cap \pi^{-1}(\xi)$ is finite for every $\xi\in \widehat \bD$: in fact, it may contain at most $r$ points since it corresponds to the zero locus (in $\T$) of the polynomial map $\C\ni w\mapsto \det(w Q(e)^2\zeta-e)\in \C$ for every $\zeta\in \pi^{-1}(\xi)$, and this polynomial map is not identically $0$ and has degree $\meg r$. Therefore, $\chi_N\cdot \nu_{h,\xi}$ is discrete for every $\xi\in \widehat \bD$, so that the previous argument shows that $\chi_N\cdot \nu_{h,\xi}=\chi_N\cdot \nu_{e,\xi}$ for every $h\in \Omega+i F$ and for every $\xi\in \widehat \bD$.
	This is sufficient to conclude that $\chi_N\cdot \nu_h=\chi_N\cdot \nu_e$ for every $h\in\Omega+ i F$. It only remains to prove that $\chi_N\cdot \nu_0= \chi_N\cdot \nu_e$. This will require some care.
	
	Take $\zeta\in N$. Observe that the function $\Cl(\C_+)\ni w \mapsto g(\gamma^{-1}(\rho \zeta)+ (0, w e))$ is continuous and holomorphic on $\C_+$ for every $\rho\in [0,1)$ (in fact, for every $\rho\in [0,1)$ it extends to a holomorphic function on a neighbourhood of $\Cl(\C_+)$). Since these functions are locally uniformly bounded and converge pointwise to $w \mapsto \nu_{-i w e, \pi(\zeta)}(\Set{\zeta})$, by~\eqref{eq:3} and~\eqref{eq:8}, by means of Cauchy's integral theorem (and the fact that the limit function is constant on $\C_+$) we see that $\int_{x}^{x'}  \nu_{-i t e, \pi(\zeta)}(\Set{\zeta}) \,\dd t=(x'-x)\nu_{e, \pi(\zeta)}(\Set{\zeta}) $ for every $x,x'\in \R$. Consequently, $\nu_{-i x e, \pi(\zeta)}(\Set{\zeta})=\nu_{ e, \pi(\zeta)}(\Set{\zeta})$ for almost every $x\in \R$.
	Observe that the mappings $h\mapsto \nu_h$ and $h\mapsto \nu_{h,\xi}$ are vaguely continuous on $\overline \Omega+ i F$, since the mapping $h\mapsto (\Pc\nu_h)(z)=f(\gamma^{-1}(z)+ i h)$ is continuous for every $z\in D$, while the mapping $h\mapsto \int_{\bD} \frac{1-\abs{z}^2}{\abs{\zeta-z}^2}\,\dd \nu_{h,\xi}(\zeta)=f(\gamma^{-1}(z)+ i h)$ is continuous for every $z\in \Db_\xi$ (cf.~\cite[the proof of Proposition 3.4]{Calzi}).
	Then, take a positive even function $\phi\in C_c(\R)$ with integral $1$ (with respect to Lebesgue measure), and define, for every $k\Meg 1$
	\[
	\widetilde f^{(k)}(\zeta,z)\coloneqq k\int_{\R} f(\zeta,z+x e)\phi(k x)\,\dd x
	\]
	for every $(\zeta,z)\in \Ds$, 
	\[
	\widetilde \nu^{(k)}\coloneqq k\int_\R \nu_{- i x}\,\phi(k x)\,\dd x
	\]
	and
	\[
	\widetilde \nu_{\xi}^{(k)}\coloneqq k\int_\R \nu_{- i x, \xi}\,\phi(k x)\,\dd x
	\]
	for every $\xi\in \widehat \bD$. Then, $(\Pc \widetilde \nu^{(k)})\circ \gamma=\widetilde f^{(k)}$ and $(\widetilde \nu_\xi^{(k)})$ is the vaguely continuous disintegration of $\widetilde \nu^{(k)}$ relative to $\widehat \beta$ for every $k\Meg 1$. In addition, by~\cite[Theorem 1 of Chapter V, \S\ 3, No.\ 3]{BourbakiInt1},
	\[
	\widetilde \nu^{(k)}_{\pi(\zeta)}(\Set{\zeta})=k\int_\R \nu_{- i x, \pi(\zeta)}(\Set{\zeta})\phi(k x)\,\dd x=\nu_{e, \pi(\zeta)}(\Set{\zeta})
	\]
	for every $k\Meg 1$ and for every $\zeta\in N$. Therefore, the previous arguments allow to show that $\chi_N\cdot\widetilde \nu^{(k)}=\chi_N\cdot \nu_e$. 
	
	Now, let $\tau_{x e}\mi$ denote the translation of $\mi$ by $x e$, and let us prove that the mapping $\R \ni x \mapsto \tau_{x e}\mi \in \Mc_\Poisson(\bDs)$ is (tightly) continuous. Notice that $\Ps \tau_{x e}\mi\meg f(\,\cdot\,+ (0, xe))$ on $\Ds$ for every $x\in \R$, by Theorem~\ref{teo:3}, so that $\tau_{x e}\mi\in \Mc_\Poisson(\bDs)$ for every $x\in \R$.  Since the mapping $x \mapsto \tau_{x e} \mi$ is vaguely continuous, by~\cite[Proposition 9 of Chapter IX, \S\ 5, No.\ 3]{BourbakiInt2}, it will suffice to show that the mapping $\R\ni x \mapsto\Ps(\tau_{x e} \mi)(0, i e)$ is continuous. To this aim, observe that, combining Example~\ref{ex:5} with the so-called Kor\'anyi's lemma (cf., e.g.,~\cite[Theorem 2.47]{CalziPeloso}), one may show that there is a constant $C>0$ such that
	\[
	\abs{\Cs((\zeta,z),(\zeta',z'))^{2}- \Cs((\zeta'',z''),(\zeta',z'))^{2}}\meg C d((\zeta,z),(\zeta'',z'')) \abs{\Cs((\zeta'',z''),(\zeta',z'))}^{2}
	\]
	for every $(\zeta,z),(\zeta'',z'')\in \Ds$ with $d((\zeta,z),(\zeta'',z''))\meg 1$ and for every $(\zeta',z')\in \Cl(\Ds)$, where $d$ denotes the Bergman distance on $\Ds$. In addition, 
	\[
	\Cs((\zeta,z),(\zeta,z))=\Cs((\zeta,z+ x e),(\zeta,z+x e))
	\]
	for every $(\zeta,z)\in \Ds$ and for every $x\in \R$, thanks to Example~\ref{ex:5}.
	Thus,
	\[
	\abs{\Ps((\zeta,z),(\zeta',z'))- \Ps((\zeta,z+x e),(\zeta',z'))}\meg C d((\zeta,z),(\zeta,z+x e)) \Ps((\zeta,z),(\zeta',z'))
	\]
	for every $(\zeta,z)\in \Ds$, for every $(\zeta',z')\in \bDs$, and for every $x\in \R$ such that $d((\zeta,z),(\zeta,z+x e)) \meg 1$.
	In addition, by Example~\ref{ex:5} again, $\Ps((\zeta,z),(\zeta',z'- x e))=\Ps((\zeta,z+ x e), (\zeta',z'))$ for every $(\zeta,z)\in \Ds$, for every $(\zeta',z')\in \bDs$, and for every $x\in \R$. Consequently,
	\[
	\abs{(\Ps \tau_{x e} \mi)(0, i e)- (\Ps \tau_{x' e}\mi)(0, i e)} \meg C d((0, (x-x'+i)e),(0, i e)) (\Ps  \mi)(0, (x'+i) e)
	\]
	for every $x,x'\in\R$ such that $d((0, i e),(0, (x-x'+i) e))\meg 1$. Our claim then follows.

	Then, define
	\[
	\widetilde\mi^{(k)}\coloneqq k \int_\R \tau_{x e} \mi\, \phi(x)\,\dd x
	\]
	for every $k\Meg 1$, and observe that  $\mi^{(k)}\in \Mc_\Poisson(\bDs)$ and that
	\[
	\Gamma(\mi^{(k)})=\chi_{\bD \setminus N} \cdot \widetilde \nu^{(k)}
	\]
	for every $k\Meg 1$. In addition, $\mi^{(k)}$ converges to $\mi$ in $\Mc_\Poisson(\bDs)$ for $k\to \infty$, so that Corollary~\ref{cor:4} shows that $\chi_N\cdot \widetilde \nu^{(k)}$ converges vaguely to $\chi_N\cdot \nu_0$ for $k\to \infty$. Consequently, $\chi_N\cdot \nu_0=\chi_N\cdot \nu_e$. The proof of (1) is therefore complete. 
	
	(2) This follows from (1), since
	\[
	\begin{split}
		f(\zeta,z)-(\Ps \mi)(\zeta,z)&=(\Pc \chi_N\cdot\nu_0)(\gamma(\zeta,z))\\
			&=(\Pc \chi_N\cdot\nu_h)(\gamma(\zeta,z))\\
			&=f(\zeta,z+i h)-(\Ps f_h)(\zeta,z) 
	\end{split}
	\]
	for every $(\zeta,z)\in \Ds$, thanks to Theorem~\ref{teo:3}. 
	
	(3) Observe that $\Pc \nu_h$ converges pointwise to $\Pc \nu_0$ on $D$, and that the $\nu_h$ are uniformly bounded for $h\in \Omega$, $\abs{h}\meg 1$ (cf., e.g.,~the proof of Theorem~\ref{teo:3}). Therefore, $\nu_h$ converges vaguely to $\nu_0$. Since, by (1),  $\chi_{N}\cdot \nu_h=\chi_N\cdot \nu_0$ for every $h\in\Omega$, this proves that $\chi_{\bD\setminus N}\cdot \nu_h$ converges vaguely to $\chi_{\bD\setminus N}\cdot \nu_0$ for $h\to 0$, $h\in\Omega$. By means of Corollary~\ref{cor:4} and Theorem~\ref{teo:3}, this proves that 	 $f\cdot \beta_{\bDs}$ converges to $\mi$ in $\Mc_\Poisson(\bDs)$ for $h\to 0$, $h\in\Omega$. In a similar way, one shows that the mapping $\Omega\ni h\mapsto f_h\cdot \beta_{\bDs}\in \Mc_\Poisson(\bDs)$ is (tightly) continuous.
\end{proof}

\begin{ex}\label{ex:6}
	With the notation of Proposition~\ref{prop:24}, let us provide an example with $a > 0$ (for every choice of $h\in \Omega$) even if $\Ps \mi=f$ (that is, even if $\nu(\bD \setminus \gamma(\bDs))=0$). Cf.~\cite[Remarks after Corollary 3.9]{LugerNedic3}.
	
	Assume that $D=\Db^2$ and that $\Ds=\C_+^2$, and take
	\[
	f\colon \C_+^2\ni (w_1,w_2)\mapsto \Re  \frac{w_1 w_2}{i(w_1+w_2)} =\frac{ \abs{w_1}^2 \Im w_2+ \abs{w_2}^2 \Im w_1 }{\abs{w_1+w_2}^2},
	\]
	so that $f$ is a positive pluriharmonic function on $\C_+^2$. Observe that, chosen any unit vector  $h=(h_1,h_2)$ in $\Omega=(\R_+^*)^2$, so that $h_1,h_2>0$, 
	\[
	a = \lim_{y \to +\infty} \frac{f(i h_1 y, i h_2 y)}{y}=\frac{h_1h_2 }{h_1+h_2}>0
	\]
	In addition, choosing 
	\[
	\gamma\colon \C_+^2\ni (w_1,w_2)\mapsto \Big(\frac{w_1-i}{w_1+i} ,\frac{w_2-i}{w_2+i} \Big)\in \Db^2,
	\]
	it is clear that
	\[
	f\circ \gamma^{-1}\colon \Db^2\ni (w_1,w_2)=\Re \frac{ (1+w_1)(1+w_2) }{(1+w_1)(1-w_2)+(1+w_2)(1-w_1)}=\Re\frac{1+w_1+w_2+w_1 w_2 }{2-2 w_1w_2 }\in \R
	\]
	extends to a pluriharmonic function on $U=\Set{(w_1,w_2)\in \C^2\colon w_1 w_2\neq 1}$. 
	Furthermore, $N\coloneqq \bD\setminus \gamma(\bDs)$ is the set $ (\T\times \Set{1})\cup (\Set{1}\times \T)$. Since $ N\setminus U=\Set{(1,1)}$, by means of Proposition~\ref{prop:8} we see that $\nu(N)=0$, where $\nu$ is the unique positive pluriharmonic measure on $\bD=\T^2$ such that $\Pc \nu= f\circ \gamma^{-1}$. 
\end{ex} 

 \begin{ex}\label{ex:1}
 	Assume that $E=\Set{0}$ and that $F$, with the induced real Jordan algebra structure, is simple and has rank $2$. Then, after fixing a frame of idempotents and considering the corresponding joint Peirce decomposition, we may assume that $F=\R\times \R\times H$, where $H$ is a (non-trivial, finite-dimensional) real Hilbert space (with scalar product $(c,c')\mapsto c\cdot c'$), endowed with the product defined by
 	\[
 	(a,b,c)(a',b',c')=\bigg(a a'+ c\cdot c', b b'+ c\cdot c', \frac{(a+b)c'+(a'+b')c}{2} \bigg)
 	\]
 	for every $(a,b,c),(a',b',c')\in F$. The same expression then defines the product on the complex Jordan algebra $F_\C$, provided that $\cdot$ denote the (symmetric) $\C$-\emph{bilinear} extension of the scalar product of $H$ to $H_\C$. We endow $F_\C$ with the scalar product defined by $\langle(a,b,c)\vert (a',b',c')\rangle\coloneqq a \overline{a'}+ b\overline{b'}+ 2 c\cdot \overline{c'}$, which is a scalar multiple of the canonical scalar product on the Hermitification of $F$. 	
 	Let $\Omega=\Set{(a,b,c)\in F\colon a, b>0, ab>c \cdot c}$ be the symmetric cone associated with $F$, and consider the symmetric Siegel domain $\Ds=F+i \Omega$. 
 	Consider the holomorphic function
 	\[
 	f\colon \Ds \ni (a,b,c)\mapsto i\frac{a b}{(a+b)(a b-c\cdot c)}\in \C.
 	\]
 	In this example we shall find the boundary value measure $\mi$ of $f$, its vaguely continuous disintegration (associated with $h=(1,1,0)$) and show that  $\Ps\mi=f$.
 	
 	Observe that $f$ is well defined since $\Im (a+b)>0$ and $a b-c\cdot c\neq0$ for every $(a,b,c)\in \Ds$.
 	Observe that the mapping $\iota\colon z \mapsto -z^{-1} $ is a biholomorphism of $\Ds$ onto itself, and that 
 	\[
 	(f\circ \iota)(a,b,c)=-i \frac{a b }{a+b}
 	\]	
 	for every $(a,b,c)\in \Ds$.
 	Observe that
 	\[
 	\Re (f\circ \iota)(a,b,c)=\frac{\abs{a}^2 \Im b+ \abs{b}^2 \Im a}{\abs{a+b}^2}\Meg 0
 	\]
 	for every $(a,b,c)\in\Ds$. Therefore, $\Re f$ is a positive pluriharmonic function  on $\Ds$.
 	
 	In addition,
 	\[
 	(f\circ \gamma^{-1})(a,b,c)
 	=\frac{(1-a b+c \cdot c)^2-(a-b)^2}{2(1-a b+c\cdot c)(1+ab-c\cdot c+a+b)}
 	\]
 	for every $(a,b,c)\in D$. 
 	Denote by $\nu$ the vague limit of $\Re (f\circ \gamma^{-1})_\rho\cdot \beta_{\bD}$ for $\rho\to 1^-$, and let us prove that $\nu(N)=0$ (recall that $N$ is the zero locus (in $\bD$) of the polynomial $Q(a,b,c)=1+a b-c\cdot c-a-d$). 
 	Observe first that $\dim_\HB N\meg   \dim \bD-1=m-1$ (cf.~Corollary~\ref{cor:5}), so that  Proposition~\ref{prop:8} shows that $\chi_N\cdot \nu= \alpha\cdot \Hc^{m-1}$ for some $\Hc^{m-1}$-integrable function $\alpha$ on $F_\C$ which vanishes on the complement of $N$. 
 	Since $f\circ \gamma^{-1}$  extends by continuity on $U\coloneqq\Set{(a,b,c)\in F_\C\colon (1-ab+c\cdot c)(1+ab-c\cdot c+a+b)\neq 0 }$, we see that $\nu$ is absolutely continuous with respect to $\Hc^m$ on  $U$, so that $\nu(N\cap U)=0$. 
 	Since $\dim_\HB (N\setminus U)\meg m-2$ (cf.~Corollary~\ref{cor:5}, since the (holomorphic) polynomials $1+ab-c\cdot c-a-b$ and $ (1-ab+c\cdot c)(1+ab-c\cdot c+a+b)$ are coprime on $F_\C$), this is sufficient to conclude that $\nu(N)=0$. Therefore, if $\mi=\Gamma^{-1}(\nu)$ (with the notation of Corollary~\ref{cor:4}), then $f=\Ps \mi$ on $\Ds$ (cf.~Remark~\ref{oss:9}).
 	
 	Now, observe that the orthogonal complement of $\R (1,1,0)$ in $F$   is $\Set{(a,-a,c)\colon a\in \R, c\in H}$. In addition, define, for every $(a,c)\in \R\times H$,
 	\[
 	\mi_{(a,c)}\coloneqq \begin{cases}
 		\frac{\pi a^2}{2 (a^2+c\cdot c)}\delta_0 +\frac{\pi c\cdot c}{4 (a^2+c\cdot c)}\big(\delta_{\sqrt{a^2+c\cdot c}}+\delta_{-\sqrt{a^2+c\cdot c}}\big) & \text{if $a^2+c\cdot c>0$}\\
 		\frac \pi 2 \delta_0 & \text{if $a=0$ and $c=0$.}
 	\end{cases}
 	\]
 	Observe that the mapping $(a,c)\mapsto \mi_{(a,c)}$ is vaguely continuous, and that
 	\[
 	(\Ss_{\C_+} \mi_{(a,c)})(w)= f(a+w,-a+w,c)
 	\]
 	for every $(a,c)\in \R\times H$ and for every $w\in \C_+$, where $\Ss_{\C_+}$ denotes the Schwarz integral operator associated with $\C_+$ (cf.~Remark~\ref{oss:10}). Consequently, 
 	\[
 	\mi=\int_{\R\times H} \mi'_{(a,c)}\,\dd \Hc^{m-1}(a,c)
 	\]
 	by Proposition~\ref{prop:25}, where $\mi'_{(a,c)}$ denotes the image of  $\mi_{(a,c)}$ under the mapping $\R\ni x \mapsto (a+x,-a+x,c) \in F$ (cf.~Proposition~\ref{prop:25}). 
 \end{ex}

\begin{ex}\label{ex:1bis}
	Keep the notation of Example~\ref{ex:1}. We shall prove that the largest positive measure $\mi'\in \Mc_\Poisson(\bDs)$ such that $\Ps \mi' \meg \Re g$ on $\Ds$ is \emph{not} pluriharmonic, and that $\Fc  \mi'$ is supported in $-\overline \Omega\cup \overline \Omega$.
	
	For the first assertion, observe that,  with the notation of Example~\ref{ex:1}, if $\nu'$ denotes the image of $\nu$ under the reflection  $z \mapsto -z$, then $(\Pc \nu')\circ \gamma=\Re g$. Therefore, saying that $\mi$ is pluriharmonic is equivalent to saying that $\chi_N\cdot \nu'$ is pluriharmonic, thanks to Theorem~\ref{teo:3}. This is in turn equivalent to saying that $\chi_{-N}\cdot \nu$ is pluriharmonic, hence also to saying that $\chi_{\gamma^{-1}(-N)}\cdot \mi$ is pluriharmonic. Now, observe that $N'\coloneqq\gamma^{-1}(-N)=\Set{z\in F\colon \det z=0}=\Set{(a,b,c)\in \R\times \R\times H\colon a b=c\cdot c}$. In addition, $(\chi_{N'}\cdot \mi'_{(a,c)})$ is a disintegration of $\chi_{N'}\cdot \mi$, and
	\[
	\chi_{N'}\cdot \mi'_{(a,c)}=\begin{cases}
		\frac{\pi b\cdot b}{4(a^2+ b\cdot b)}(\delta_{(a+\sqrt{a^2+c\cdot c},-a+\sqrt{a^2+c\cdot c},c)} + \delta_{(a-\sqrt{a^2+c\cdot c},-a-\sqrt{a^2+c\cdot c},c)}) &\text{if $a^2+c\cdot c>0$}\\
		\frac \pi 2 \delta_{(0,0,0)} &\text{if $a=0$ and $c=0$}
	\end{cases}
	\]
	for every $(a,c)\in \R\times H$. It is then clear that the function $(a,c)\mapsto\chi_{N'}\cdot \mi'_{(a,c)}$ is   vaguely continuous  on $(\R\times H)\setminus \Set{(0,0)}$; since it is not vaguely continuous at $(0,0)$, Proposition~\ref{prop:25} shows that $\chi_{N'}\cdot \mi$ is \emph{not} pluriharmonic. Thus, $\mi'$ is \emph{not} pluriharmonic.
	
	Now, let us prove that $\Fc\mi'$ is supported in  $-\overline\Omega \cup \overline \Omega$. Since $\mi'$ is the limit of $\Re g_h\cdot \beta_{\bDs}$ in $\Mc_\Poisson(\bDs)$ for $h\to 0$, $h\in\Omega$, thanks to Theorem~\ref{teo:4}, it will suffice to show that 
	$\Fc(\Re g_h\cdot \beta_{\bDs})$ is supported in $-\overline\Omega \cup \overline \Omega$ for every $h\in\Omega$. 
	To do that, take holomorphic functions $\theta^{(\eps)}\colon \Ds\to \C$ such that $\abs{\theta^{(\eps)}}\meg \ee^{-C\eps \abs{\,\cdot\,}^{1/4}}$ for every $\eps\in (0,1)$, and such that $\theta^{(\eps)}\to 1$ locally uniformly on $\Ds$ for $\eps\to 0^+$, for a suitable constant $C>0$ (cf.~\cite[Lemma 1.22]{CalziPeloso} or the proof of~\cite[Lemma 8.1]{OgdenVagi}). Then, $( g \theta^{(\eps)})(\,\cdot\,+ i h) \in H^2(\Ds)$ for every $h\in \Omega$ and for every $\eps>0$, so that $g_h \theta^{(\eps)}_h$ is supported in $ \overline\Omega$ by, e.g.,~\cite[Theorem 3.1]{SteinWeiss}, so that $\Fc[\Re(g_h \theta^{(\eps)}_h)]$ is supported in $-\overline\Omega\cup \overline \Omega$. 
	Since $\Re (g_h \theta^{(\eps)}_h)\cdot \beta_{\bDs}$ converges to $\Re g_h\cdot \beta_{\bDs}$ in the space of tempered distributions (by dominated convergence) for $\eps\to 0^+$, this proves that $\Fc(\Re g_h\cdot \beta_{\bDs})$ is supported in $-\overline\Omega \cup \overline \Omega$ for every $h\in\Omega$.
\end{ex}

\section{Clark Measures on  Symmetric Siegel Domains}\label{sec:7}

\begin{deff}
	Take a holomorphic function $\phi\colon \Ds\to \Db $ and $\alpha\in \T$. Then, we define $\mi_\alpha[\phi]$ (or simply $\mi_\alpha$) as the largest positive measure in $\Mc_\Poisson(\bDs)$ such that 
	\[
	\Ps \mi_\alpha\meg \Re \left( \frac{\alpha+\phi}{\alpha-\phi} \right)=\frac{1-\abs{\phi}^2}{\abs{\alpha-\phi}^2}.
	\]
	
	In addition,  we define   $\phi_{0,h}$ as the pointwise limit of $\phi_{y h}$ for $y\to 0^+$ (where it exists), for every $h\in\Omega$ (recall that $\phi_h\colon \bDs\ni (\zeta,z)\mapsto (\zeta,z+ i h)\in \Db$). We shall simply write $\phi_0$ instead of $\phi_{0, e}$.
\end{deff}

Notice that, since $\mi_\alpha$ need not be plurihamonic, $\Ss \mi_\alpha$ need \emph{not} have a positive real part. In particular, it need not be related to $\frac{\alpha+\phi}{\alpha-\phi}$ in any simple way.

In addition, the various $\phi_{0, h}$ may differ as $h$ varies, even though they must agree $\beta_{\bDs}$-almost everywhere (with the `admissible' limit). Since our methods rely heavily on the disintegration of the $\mi_\alpha$, we sometimes need to distinguish   the various $\phi_{0, h}$.

\begin{oss}\label{oss:19}
	By Proposition~\ref{prop:25} and Theorem~\ref{teo:3}, if $\phi^{(\zeta,z)}$ denotes the restriction of $\phi$ to $(\zeta,z)+ \C_+ h$ for some unit vector $h\in \C_+$, then $(\mi_\alpha[\phi^{(\zeta,z)}])_{(\zeta,z)\in \pi_h(\bDs)}$ is a disintegration of $\mi_\alpha[\phi]$.
\end{oss}

\begin{prop}\label{prop:27}
	For every $\alpha\in \T$, the absolutely continuous part of $\mi_\alpha$ (with respect to $\beta_{\bDs}$) has density
	\[
	\Re \left( \frac{\alpha+\phi_0}{\alpha-\phi_0} \right)=\frac{1-\abs{\phi_0}^2}{\abs{\alpha-\phi_0}^2}.
	\]
\end{prop}

Notice that, in this case, there is no need to distinguish between the various $\phi_{0, h}$, since they all agree $\beta_{\bDs}$-almost everywhere.

In particular, this shows that $\phi$ is inner if and only if $\mi_\alpha$ is (not necessarily plurihamonic, but) singular with respect to $\beta_{\bDs}$ for some (equivalently, every) $\alpha\in \T$. Conversely, if $\mi$ is a positive \emph{pluriharmonic} measure which is singular with respect to $\beta_{\bDs}$, then $\psi\coloneqq \alpha\frac{\Ss \mi-1}{\Ss \mi+1}$  is a holomorphic function from $\Ds$ into $\Db$ such that $\Re\left( \frac{\alpha+ \psi}{\alpha-\psi} \right)=\Re \Ss \mi=\Ps \mi$, so that $\mi=\mi_\alpha[\psi]$. Clearly, $\psi$ is inner if and only if $\mi$ is singular.

\begin{proof}
	This follows from~\cite[Proposition 4.6]{Calzi}, Remark~\ref{oss:7} and Theorem~\ref{teo:3}.
\end{proof}

\begin{prop}\label{prop:26}
	For every $\alpha\in \T$, the singular part $\sigma_\alpha$ of $\mi_\alpha$ is concentrated in $\phi_{0,h}^{-1}(\alpha)$ for every $h\in \Omega$. In particular, if $\alpha'\in \T$ and $\alpha\neq \alpha'$, then $\sigma_\alpha$ and $\sigma_{\alpha'}$ are singular with respect to each other.
\end{prop}

In the one-dimensional case there is a nice relation between the Hardy--Littlewood maximal function on $\R$ and the non-tangential `Poisson' maximal function (that is, the non-tangential maximal function of the harmonic extension to $\C_+$) of every $\mi\in \Mc_\Poisson(\R,\C_+)$. This guarantees that the singular part of $\mi$ is concentrated where the `Poisson' maximal function is infinite, whence the above result in this case. 
In the general case, though, it is unclear whether the (restricted admissible) `Poisson--Szeg\H o' maximal function of the elements of $\Mc_\Poisson(\bDs)$ should be related in a simple yet quantitative way to the differentiation of measures, so that the above result is proved by disintegration instead (and is consequently considerably weak). 

\begin{proof}
	When $\Ds=\C_+$, this follows from~\cite[Corollary 9.1.24]{CimaMathesonRoss}, Remark~\ref{oss:7} and Theorem~\ref{teo:3}, together with~\cite[Theorem 4.1]{MackeyMellon}. The general case then follows by disintegration, thanks to Remark~\ref{oss:19}.
\end{proof}

\begin{prop}\label{prop:31}
	The mapping $\alpha \mapsto \mi_\alpha$ is vaguely continuous, and
	\[
	\int_\T \mi_\alpha \,\dd\beta_\T(\alpha)=\beta_{\bDs}.
	\]
	If, in addition, $\phi$ is inner, then $\beta_\T$ is a pseudo-image measure of $\beta_{\bDs}$ under $\phi_0$, and $(\mi_\alpha)$ is a disintegration of $\beta_{\bDs}$ relative to $\beta_\T$.
\end{prop}

Of course, $\beta_\T$ cannot be the image measure of $\beta_{\bDs}$ under $\phi$.

\begin{proof}
	The first assertion follows from~\cite[Theorem 4.9]{Calzi} and Theorem~\ref{teo:3}. The second assertion follows again from~\cite[Theorem 4.9]{Calzi} and Theorem~\ref{teo:3}, but this time one has to invoke also Corollary~\ref{cor:4} and Remark~\ref{oss:7}. The third assertion is then a consequence of Propositions~\ref{prop:27} and~\ref{prop:26}.
\end{proof}

\begin{cor}\label{cor:8}
	The set of $\alpha\in \T$ such that $\Ps \mi_\alpha \neq \frac{1-\abs{\phi}^2}{\abs{\alpha- \phi}^2}$ is $\beta_\T$-negligible. 
\end{cor}

In particular, $\mi_\alpha$ must be pluriharmonic for $\beta_\T$-almost every $\alpha\in \T$.

As a consequence of Theorem~\ref{teo:2} and its proof, if $\Ds$ is a product of irreducible domains which are not of tube type and of $\C_+^k$ for some $k\in\N$, then $\Ps \mi_\alpha \neq \frac{1-\abs{\phi}^2}{\abs{\alpha- \phi}^2}$ for at most $k$ $\alpha\in \T$.

\begin{proof}
	Denote with $N'$ the set defined in the statement. Observe that, if $\alpha \in N'$, then $(\Ps\mi_\alpha)(\zeta,z)<\frac{1-\abs{\phi(\zeta,z)}^2}{\abs{\alpha- \phi(\zeta,z)}^2} $ for every $(\zeta,z)\in \Ds$, thanks to Theorem~\ref{teo:3} and to the fact that $\Pc \nu$ is everywhere strictly positive on $D$ if $\nu$ is a non-zero positive measure on $\bD$ (cf.~Proposition~\ref{prop:1}). In addition, Proposition~\ref{prop:31} shows that
	\[
	1= (\Ps  \beta_{\bDs})(0 , i e)=\int_\T (\Ps \mi_\alpha)(0, i e) \,\dd \beta_\T(\alpha)\meg \int_\T \frac{1-\abs{\phi(0, i e)}^2}{\abs{\alpha- \phi(0, i e)}^2}\,\dd \beta_\T(\alpha)=1, 
	\]
	so that $N'$ is $\beta_\T$-negligible.
\end{proof}

\begin{prop}\label{prop:29}
	Let $\psi\colon \Db \to \Db $ be a holomorphic map. Then, for every $\alpha\in \T$,
	\[
	\mi_{\alpha}[\psi\circ \phi]=\int_{ \T} \mi_{\alpha'}[\phi]\,\dd \mi_\alpha[\psi](\alpha').
	\]
\end{prop}

\begin{proof}
	This follows from~\cite[Proposition 4.10]{Calzi}, using Theorem~\ref{teo:3} and Remark~\ref{oss:9}.
\end{proof}

\begin{deff}
	We define
	\[
	\Cs_\phi\colon \Ds \times \Ds \ni ((\zeta,z),(\zeta',z'))\mapsto (1-\phi(\zeta,z)\overline{\phi(\zeta',z')})\Cs((\zeta,z),(\zeta',z'))\in \C,
	\]
	so that $\Cs_\phi$ is the reproducing kernel of a Hilbert space $\phi^*(H^2(\Ds))$.\footnote{Cf., e.g.,~\cite[Lemma 3.9]{Timotin}.} We also define, for every $\alpha\in \T$ and for every  $f\in L^2(\mi_\alpha)$,
	\[
	\Cs_{\phi,\alpha}(f)\coloneqq (1-\overline \alpha \phi) \Cs(f\cdot \mi_\alpha)=(1-\overline \alpha \phi) \int_{\bDs} \Cs(\,\cdot\,,(\zeta,z))  f(\zeta,z)\,\dd\mi_\alpha(\zeta,z),
	\]
	and we define $H^2(\mi_\alpha)$ as the closed vector subspace of $L^2(\mi_\alpha)$ generated by the $\Cs((\zeta,z),\,\cdot\,)$, for $(\zeta,z)\in \Ds$.
\end{deff}

Notice that $\Cs_{\phi,\alpha}$ and $H^2(\mi_\alpha)$ are well defined, since $\Cs((\zeta,z),\,\cdot\,)\in L^2(\mi_\alpha)$ for every $(\zeta,z)\in \Ds$ (cf.~Corollary~\ref{cor:7}).  

We shall use similar notation (replacing $\Ds$ with $D$ and $\Cs$ with $\Cc$) in order to describe the analogous objects on the bounded realization $D$ of $\Ds$. The following result shows how one may relate these two settings.

\begin{oss}\label{oss:14}
	The mapping $\Gamma_1\colon f \mapsto (f\circ \gamma) \frac{\Cs(\,\cdot\,,(0, i e))}{\sqrt{\Cs((0, i e),(0,i e))}}$ induces isometries of $H^2(D)$ and $(\phi\circ \gamma^{-1})^*(H^2(D))$ onto $H^2(\Ds)$ and $\phi^*(H^2(\Ds))$, respectively. In addition, the mapping $\Gamma_2\colon f \mapsto  \sqrt{\Cs((0, i e),(0,i e))}\frac{ f\circ  \gamma^{-1}}{\Cs(\,\cdot\,,(0, i e))_0\circ \gamma^{-1}}$ (extended by $0$ to $\bD$) induces an isometry of  $L^2(\mi_\alpha[\phi])$ and $H^2(\mi_\alpha[\phi])$ \emph{into}  $ L^2(\mi_\alpha[\phi\circ \gamma^{-1}])$ and  $ H^2(\mi_\alpha[\phi\circ \gamma^{-1}])$, respectively, such that
	\[
	\Gamma_1\circ\Cc_{\phi\circ \gamma^{-1},\alpha}\circ \Gamma_2=\Cs_{\phi,\alpha}.
	\]
	These isometries are \emph{onto} if  and only if $\Ps \mi_\alpha=\frac{1-\abs{\phi}^2}{\abs{\alpha-\phi}^2} $, hence for $\beta_\T$-almost every $\alpha\in \T$.
\end{oss}

\begin{proof}
	The first assertion follows directly from the relation between the reproducing kernels of $H^2(D)$ and $(\phi\circ \gamma^{-1})^*(H^2(D))$ and $H^2(\Ds)$ and $\phi^*(H^2(\Ds))$, respectively (cf.~Lemma~\ref{lem:7}).  The second assertion then follows by means of Theorem~\ref{teo:3}. The last assertion follows by means of Theorem~\ref{teo:3} and Corollary~\ref{cor:8}.
\end{proof}

Combining Remark~\ref{oss:14} with~\cite[Proposition 4.22]{Calzi}, one gets the following result.

\begin{oss}\label{oss:17}
	Assume that $\Ds=\C_+$ and that $\phi$ is inner. Then, for every $f\in \phi^*(H^2(\C_+))$,
	\[
	(\Cs_{\phi,\alpha} f_0) (w)= f(w)-   \frac{1-\overline \alpha \phi(w)}{w+i} \lim_{y\to \infty}\frac{1-\abs{\phi(i y)}^2 }{\abs{\alpha- \phi(iy)}^2} f(i y).
	\]
	This follows applying Remark~\ref{oss:14} and computing $\Cc_{\phi,\alpha}^{-1}\Gamma_1^{-1}(\Cs_{\phi,\alpha} (f_0)- f)$. 
\end{oss}

\begin{prop}\label{prop:32}
	Take $\alpha\in \T$. Then, $\Cs_{\phi,\alpha}$ induces a partial isometry of $L^2(\mi_\alpha)$ into $\phi^*(H^2(\Ds))$, with kernel $L^2(\mi_\alpha)\ominus H^2(\mi_\alpha)$, so that $\Cs_{\phi,\alpha}(f)=0$ if and only if $\Cs(f\cdot \mi_\alpha)=0$.
	
	If, in addition,  $\Ps \mi_\alpha=\frac{1-\abs{\phi}^2}{\abs{\alpha-\phi}^2}$, then, $\Cs_{\phi,\alpha}$ maps $L^2(\mi_\alpha)$ onto $\phi^*(H^2(\Ds))$ and
	\[
	\Cs_{\phi,\alpha}(\Cs(\,\cdot\,,(\zeta,z)))=\frac{1}{1-\alpha \overline{\phi(\zeta,z)}}\Cs_\phi(\,\cdot\,,(\zeta,z))
	\]
	for every $(\zeta,z)\in \Ds$.
\end{prop}

\begin{proof}
	This follows from~\cite[Theorem 4.16]{Calzi} and Remark~\ref{oss:14}.
\end{proof}

\begin{prop}\label{prop:37}
	Assume that $\phi$ is inner, and  take $f\in \phi^{*}(H^2(\Ds))$ and $h\in\Omega$. Then, $\Cs_{\phi,\alpha}^* f=f_{0,h}$ $\mi_\alpha$-almost everywhere for $\beta_\T$-almost every $\alpha\in \T$. If, in addition, $f_{0,h}\in C_{\Poisson,0}(\bDs)$, then $f_{0,h}\in H^2(\mi_\alpha)$ and $\Cs_{\phi,\alpha}(f_{0,h})= f$ for every $\alpha\in \T$.
\end{prop}

Concerning the second assertion, notice that in general $\Cs_{\phi,\alpha}$ (is a partial isometry, but) need \emph{not} map $L^2(\mi_\alpha)$ \emph{onto} $\phi^*(H^2)$ for every $\alpha\in \T$, so that saying that $\Cs_{\phi,\alpha}(f_{0,h})=f$ and $f_{0,h}\in H^2(\mi_\alpha)$ is more precise than saying that $\Cs^*_{\phi,\alpha}(f)=f_{0,h}$ $\mi_\alpha$-almost everywhere.

\begin{proof}
	The first assertion follows from~\cite[Proposition 4.17]{Calzi}, Corollary~\ref{cor:8}, and Remark~\ref{oss:14}.
	The second assertion follows from~\cite[Proposition 4.17]{Calzi} and Remark~\ref{oss:14}.
\end{proof}

\begin{prop}\label{prop:36}
	Assume that $\phi$ is inner. Take $\alpha\in \T$, $h\in \Omega$, and $f\in \phi^*(H^2)$. Assume that $f(\zeta,z+ \,\cdot\,h)\in \phi(\zeta,z+  \,\cdot\,h)^*(H^2(\C_+))$ for $\beta_{\bDs}$-almost every $(\zeta,z)\in\pi_h(\bDs)$.
	Then, the following hold:
	\begin{enumerate}
		\item[\textnormal{(1)}]  $f\in \phi^*(H^2)$;
		
		\item[\textnormal{(2)}] $f_{\rho h}$ converges to $f_{0,h}$ pointwise  $\mi_\alpha$-almost everywhere 		for $\rho \to 0^+$;
		
		\item[\textnormal{(3)}] for $\beta_{\bDs}$-almost every $(\zeta,z)\in \pi_h(\bDs)$ and for every $w\in \C_+$, 
		\[
		\begin{split}
		f(\zeta,z+ w h)&= (1-\overline \alpha \phi(\zeta,z+ w h)) \int_{\bDs} \frac{ f_{0,h}(\zeta',z')}{2 \pi i (\langle h\vert z'\rangle/\abs{h}^2-w)}\,\dd \mi_{\alpha,(\zeta,z)}(\zeta',z') \\
		&\quad+    \frac{1-\overline \alpha \phi(\zeta,z+w h)}{w+i} \lim_{y\to \infty}\frac{1-\abs{\phi(\zeta,z+i y h)}^2 }{\abs{\alpha- \phi(\zeta,z+i y h)}^2} f(\zeta,z+i y h)
		\end{split}
		\]
		where $(\mi_{\alpha, (\zeta,z)})_{(\zeta,z)\in \pi_h(\bDs)}$ denotes the vaguely continuous disintegration of $\mi_\alpha$ relative to $\beta_h$ (cf.~Remark~\ref{oss:19}).
	\end{enumerate}
\end{prop}

\begin{proof}
	\textsc{Step I} Assume that $\Ds=\C_+$. Then, (1) is trivial,  (3) is a trivial consequence of (2) and Remark~\ref{oss:17}, and (2) follows from~\cite[Theorem 1.4]{Poltoratski}   and Theorem~\ref{teo:3}. 
	
	\textsc{Step II} Consider the general case. Then, (1) follows checking that $f$ is orthogonal in $H^2(\Ds)$ to $\phi H^2(\Ds)$ by disintegration. In addition,~\textsc{step I} shows that $f_{y h}$ converges poinwise $\mi_{\alpha,(\zeta,z)}$-almost everywhere
	(necessarily to $f_{0,h}$), for $y\to 0^+$, for  $\beta_{\bDs}$-almost  every $(\zeta,z)\in \pi_h(\bDs)$, and that (3) holds. This is sufficient to prove that $f_{y h}$ converges to $f_{0,h}$ pointwise  $\mi_\alpha$-almost everywhere for $y\to 0^+$, so that also (2) holds. 
\end{proof}

\section{Appendix: Disintegration}\label{sec:app}

In this section we recall some basic facts on the disintegration of complex Radon measures on locally compact spaces with a countable base. Cf.~\cite[Appendix]{Calzi} for the proofs.

\begin{deff}\label{def:1}
	Let $X,Y$ be two locally compact spaces with a countable base, and let $\mi$ and $\nu$ be complex Radon measures on $X$ and $Y$, respectively.  Let $p\colon X\to Y$ be a $\mi$-measurable map. Then, we say that $\nu$ is a (resp.\ weak) pseudo-image measure of $\mi$ under $p$ if  a subset $N$ of $Y$ is  $\nu$-negligible if and only if (resp.\ only if) $p^{-1}(N)$ is $\mi$-negligible.
\end{deff}

\begin{prop}\label{prop:6}
	Let $X,Y $ be two locally compact spaces with a countable base, and let $\mi$ and $\nu$ be two complex Radon measures on $X$ and $Y$, respectively. Let $p\colon X\to Y$ be a $\mi$-measurable map, and assume $\nu$ is a weak pseudo-image measure of $\mi$ under $p$. Then, there is a family $(\mi_y)_{y\in Y}$ of complex Radon measures on $X$ such that the following hold:
	\begin{enumerate}
		\item[\textnormal{(1)}] for $\nu$-almost every $y\in Y$, the measure $\mi_y$ is concentrated on $p^{-1}(y)$;
		
		\item[\textnormal{(2)}]  for every  $f\in \Lc^1(X)$, the function $f$ is   $\mi_y$-integrable for $\nu$-almost every $y\in Y$, the mapping $y\mapsto \int_X f\,\dd \mi_y$ is   $\nu$-integrable, and
		\[
		\int_X f\,\dd \mi= \int_Y \int_X f\,\dd \mi_y\,\dd \nu(y).
		\]
		
		\item[\textnormal{(3)}]  for every   positive Borel measurable function $f$ on $X$, the positive function   $y\mapsto \int_X f\,\dd \abs{\mi_y}$ is Borel measurable, and
		\[
		\int_X f\,\dd \abs{\mi}= \int_Y \int_X f\,\dd \abs{\mi_y}\,\dd \abs{\nu}(y).
		\]
		
		\item[\textnormal{(4)}] if we define $A\coloneqq \Set{y\in Y\colon \mi_y\neq 0}$, then $A$ is a Borel subset of $Y$ and $\chi_A\cdot \nu$ is a pseudo-image measure of $\mi$ under $p$.
	\end{enumerate}
	
	Furthermore, if $\nu'$ is a weak pseudo-image measure of $\mi$ under $p$ and $(\mi'_y)$ is a family of complex Radon measures on $X$   such that \textnormal{(1)} and \textnormal{(2)} hold with $\nu$ and $(\mi_y)$ replaced by $\nu'$ and $(\mi'_y)$, respectively, then there are a $\nu'$-measurable subset $A'$ of $Y$ such that $\chi_{A'}\cdot \nu'$ is a pseudo-image measure of $\mi$ under $p$, and a locally $\nu$-integrable complex function $g$ on $Y$  such that   $\chi_{A'}\cdot\nu' = g\cdot \nu$, and $\mi_y = g(y) \mi'_y$   for   $\chi_A\cdot\nu$-almost every $y\in Y$.
\end{prop}

Notice that if $p$ is $\mi$-proper (that is, if $p_*(\abs{\mi})$ is a Radon measure) and $\nu=p_*(\mi)$, then $\mi_y$ is a finite measure and $\mi_y(X)=1$ for $\nu$-almost every $y\in Y$.\footnote{Notice that, in general, $\abs{\nu}\neq p_*(\abs{\mi})$, so that the $\abs{\mi_y}$ need not be   probability measures.}

In addition, observe that the function $g$ in the second part of the statement necessarily vanishes $\nu$-almost everywhere on the complement of $A$, so that the last assertion may be equivalently stated as `$\mi_y = g(y) \mi'_y$   for   $\nu$-almost every $y\in Y$'. We chose to use $\chi_A\cdot \nu$ in order to stress the fact that uniqueness is essentially related to pseudo-image measures (and not simply weak pseudo-image measures).

\begin{deff}
	Let $X,Y $ be two locally compact spaces with a countable base, and let $\mi$ and $\nu$ be two complex Radon measures on $X$ and $Y$, respectively. Let $p\colon X\to Y$ be a $\mi$-measurable map, and assume $\nu$ is a weak pseudo-image measure of $\mi$ under $p$. Then, we say that a family $(\mi_y)_{y\in Y}$ of complex Radon measures on $X$ is  a disintegration of $\mi$ relative to   $\nu$ if (1) and (2) of Proposition~\ref{prop:6} hold.  
	
	We say that $(\mi_y)$ is a disintegration of $\mi$ under $p$ if, in addition, $\nu=p_*(\mi)$ (that is, if $p_*(\abs{\mi})$ is a Radon measure and $\mi_y(X)=1$ for $\nu$-almost every $y\in Y$).
\end{deff}

\begin{oss}\label{oss:3}
	Let $X, Y$ be two locally compact spaces with a countable base,   $\mi_1,\mi_2$ two  complex Radon measures on $X$, $\nu$ a complex Radon measures on $Y$, and $p\colon X\to Y$ a $\mi_1$- and $\mi_2$-measurable map such that $\nu$ is a weak pseudo-image measure of both $\mi_1$ and $\mi_2$ under $p$.\footnote{Notice that, if $\nu_1$ and $\nu_2$ are \emph{positive} weak pseudo-image measures of $\mi_1$ and $\mi_2$ under $p$, respectively, then one may choose $\nu=\nu_1+\nu_2$. In other words, this assumption is not restrictive.} 
	Let $(\mi_{j,y})$ be a disintegration of $\mi_j$ relative to its   $\nu$, for $j=1,2$. In addition, let $\mi_{1,y}=\mi_{1,y}^a+\mi_{1,y}^s$ be the Lebesgue decomposition of $\mi_{1,y}$ with respect to $\mi_{2,y}$ (so that $\mi_{1,y}^a$ is absolutely continuous and $\mi_{1,y}^s$ is singular with respect to $\mi_{2,y}$). Analogously, let $\mi_1=\mi_1^a+\mi_1^s$ be the Lebesgue decomposition of $\mi_1$ with respect to $\mi_2$. 	
	Then, $(\mi_{1,y}^a)$ and $(\mi_{1,y}^s)$ are disintegrations of $\mi_1^a$ and $\mi_1^s$, respectively, relative to their common weak pseudo-image measure $\nu$. 
\end{oss}

\begin{prop}\label{prop:7}
	Let $X, Y$ be two locally compact spaces with a countable base,   $\mi$ a complex    Radon measure on $X$, and $\nu$ a complex Radon measures on $Y$.  Let   $(\mi_y)_{y\in Y}$ be a family of complex Radon measures on $X$ such that
	\begin{equation}\label{eq:2}
		\int_Y \abs{\mi_y}(K)\,\dd \abs{\nu}(y)<+\infty
	\end{equation}
	for every compact subset $K$ of $X$, and such that, for every $f\in C_c(X)$, the mapping $y\mapsto \int_X f\,\dd \mi_y$ is  $\nu$-integrable and
	\begin{equation}\label{eq:1}
		\int_X f\,\dd \mi= \int_Y \int_X f\,\dd \mi_y\,\dd \nu(y).
	\end{equation}
	Assume, in addition, that there is a Borel measurable map $p\colon X\to Y$ such that $\mi_y$ is concentrated on $p^{-1}(y)$ for $\nu$-almost every $y\in Y$.
	Then, $\nu$ is a weak pseudo-image measure of $\mi$ under $p$ and $(\mi_y)$ is a disintegration of $\mi$ relative to  $\nu$.
\end{prop}

\end{document}